\documentclass{amsart}

\usepackage[all]{xy}
\usepackage{amsmath}
\usepackage{hyperref}
\usepackage{amsfonts,graphics,amsthm,amsfonts,amscd,latexsym}
\usepackage{epsfig}
\usepackage{flafter}
\usepackage{mathtools}
\usepackage{comment}
\usepackage{stmaryrd}
\usepackage{tabularx}
\usepackage{pst-node}
\usepackage{tikz-cd}
\usepackage{enumitem}
\hypersetup{
    colorlinks=true,    
    linkcolor=blue,          
    citecolor=blue,      
    filecolor=blue,      
    urlcolor=blue           
}

\usepackage{pgfplots}
\pgfplotsset{compat=1.15}
\usepackage{mathrsfs}
\usepackage{tikz}
\usetikzlibrary{graphs,positioning,arrows,shapes.misc,decorations.pathmorphing}

\tikzset{
    >=stealth,
    every picture/.style={thick},
    graphs/every graph/.style={empty nodes},
}

\tikzstyle{vertex}=[
    draw,
    circle,
    fill=black,
    inner sep=1pt,
    minimum width=5pt,
]
\usepackage[position=top]{subfig}
\usepackage{amssymb}
\usepackage{color}

\usepackage{xurl}

\calclayout

\usetikzlibrary{decorations.pathmorphing}
\tikzstyle{printersafe}=[decoration={snake,amplitude=0pt}]

\newcommand{\Aut}{\operatorname{Aut}}

\newcommand{\codim}{\operatorname{codim}}

\newcommand{\pp}{\mathbb{P}}

\newcommand{\qq}{\mathbb{Q}}
\newcommand{\zz}{\mathbb{Z}}
\newcommand{\nn}{\mathbb{N}}
\newcommand{\rr}{\mathbb{R}}
\newcommand{\cc}{\mathbb{C}}

\usepackage{tikz}
\usetikzlibrary{matrix,arrows,decorations.pathmorphing}
\usetikzlibrary{arrows}

\definecolor{uuuuuu}{rgb}{0.26666666666666666,0.26666666666666666,0.26666666666666666}

\newtheorem{introcon}{Conjecture}

  \newtheorem{theorem}{Theorem}[section]
  \newtheorem{lemma}[theorem]{Lemma}
  \newtheorem{proposition}[theorem]{Proposition}
  \newtheorem{corollary}[theorem]{Corollary}

\theoremstyle{definition}
  
  \newtheorem{definition}[theorem]{Definition}
  \newtheorem{example}[theorem]{Example}

  \newtheorem{question}[theorem]{Question}
  \newtheorem{construction}[theorem]{Construction}

\newtheorem{alphtheorem}{Theorem}

\newtheorem{alphcor}[alphtheorem]{Corollary}

\newtheorem{remark}[theorem]{Remark}

\theoremstyle{remark}

\numberwithin{equation}{section}

\newcommand{\cecile}[1]{\textcolor{red}{#1}}

\keywords{Fundamental groups, Calabi--Yau surfaces, toric surfaces, toric fibrations}

\subjclass[2020]{Primary 14E30, 14F35; Secondary 90C57, 14M25, 20F34}

\begin{document}

\title[Fundamental groups of log Calabi--Yau surfaces]{Fundamental groups of log Calabi--Yau surfaces}

\author[C. Gachet]{C\'ecile Gachet}
\address{Ruhr Universit\"at Bochum, Universit\"atsstr. 150, 
44801 Bochum, Germany
}
\email{cecile.gachet@rub.de}

\author[Z.~Liu]{Zhining Liu}
\address{Laboratoire de Mathématiques et Applications,
Université de Poitiers,
86073 Poitiers CEDEX 9,
France}
\email{zhining.liu@univ-poitiers.fribs.re.kr}

\author[J.~Moraga]{Joaqu\'in Moraga}
\address{UCLA Mathematics Department, Box 951555, Los Angeles, CA 90095-1555, USA
}
\email{jmoraga@math.ucla.edu}

\begin{abstract}
In this article, 
we study the orbifold fundamental group
$\pi_1^{\rm orb}(X,\Delta)$
of a Calabi--Yau pair $(X,\Delta)$ with log canonical singularities.
We conjecture that the orbifold fundamental group
$\pi_1^{\rm orb}(X,\Delta)$ of a $n$-dimensional log Calabi--Yau pair
admits a normal solvable subgroup
of rank at most $2n$
and index at most $c(n)$.
We prove this conjecture in the case that $n=2$. 
More precisely, for a log Calabi--Yau surface pair $(X,\Delta)$ we show that $\pi_1^{\rm orb}(X,\Delta)$ is the extension of a nilpotent group of length at most $2$ and rank at most $4$ by a finite group of order at most $7200$. We also show that the bounds on the nilpotency length, rank, and order of the finite group quotient in this result are sharp.
Finally, we provide some necessary criteria for a log Calabi--Yau surface $(X,\Delta)$ to have an infinite, or a non virtually abelian orbifold fundamental group.
\end{abstract}

\maketitle

\setcounter{tocdepth}{1}
\tableofcontents

\section{Introduction}

The topology of complex algebraic varieties is a long-standing topic with numerous ramifications, from the classical Riemann uniformization theorem (see \cite{dSG16} for a historical survey) to a plethora of more recent results and open questions. There are some recent notable results about the most central notions of fundamental groups and universal covers of varieties. To just cite a few striking results in varied situations, in chronological order of publication, see \cite{Ara91,Cam93,Kol93,Tol93,Ara95,Cam95,CT95,Kol95,CT97,ABCKT96,AN99,CKO03,CF03,Cam04,DPS09,Cam11bis,BCGP12,CHK13,CC14,ADH16,CC16,GKP16,Ara17,Cat17,Bra20,AFPRW22,CGGN22,CGG23}.
One of the simplest ideas governing that realm of questions is that the fundamental group of a complex algebraic variety should be easiest to control when the variety's curvature is the most positive.
For instance, any smooth Fano variety is simply connected~\cite{Kob61,Cam91,KMM92}, the fundamental group of a smooth Calabi--Yau variety\footnote{Here, we refer to a variety as Calabi--Yau if it has a numerically trivial canonical divisor.} is virtually abelian~\cite{Gro78,Bea83}, and the fundamental groups of smooth canonically polarized varieties are still not quite fully understood (see, e.g., \cite{Ara95} for partial results).

Allowing singularities in the picture unleashes many more unruly topological phenomena (see, for example, \cite[Theorem 12.1]{Sim11} and \cite[Theorems 1 and 2]{KK14}). In birational geometry, certain types of singularities naturally arise running the minimal model program~\cite[Definition 2.34]{KM98}. Their local topology has been recently studied: For example, Braun shows in \cite{Bra20} that Kawamata log terminal (klt) singularities have finite local fundamental groups, whereas log canonical (lc) singularities have virtually solvable local fundamental groups in dimension $2$, can have larger local fundamental groups in dimension $3$, and even have any free group as local fundamental group in dimension 4 by \cite{FM23}.

There are global counterparts to these local results: In \cite{Bra20}, it also is proven that the fundamental group of the smooth locus of a klt Fano variety is finite, whereas \cite{CC14} shows that the fundamental group of the smooth locus of a klt Calabi--Yau surface is virtually abelian. In higher dimension, the fundamental group of the smooth locus of a klt Calabi--Yau variety is conjectured to be virtually abelian as well. This conjecture is yet to be proven (for some partial results, see \cite{GGK19,Dru18,HP19,Cam21}, \cite{CC14} in dimension $2$, \cite{BF24} in dimension $3$).

The global topology of Fano and Calabi--Yau pairs with singularities wilder than klt (such as purely log terminal (plt), divisorial log terminal (dlt), or log canonical (lc) singularities) has been less studied in the past. Importantly, the orbifold fundamental group of such a pair can well be infinite, as it often mimics the fundamental group of a quasiprojective variety. For example, the pair $(\mathbb{P}^2,\ell_1+\ell_2)$, where $\ell_1,\ell_2$ are two distinct lines, is a dlt Fano pair, and has an infinite cyclic orbifold fundamental group. Some results on certain dlt Fano pairs are given by \cite[Proposition 7.10]{KM99}.

\subsection{Main results}
In this paper, we provide structural results on the orbifold fundamental groups of log canonical Calabi--Yau pairs of dimension 1 and 2. Along the way, we also prove results on the orbifold fundamental groups of log canonical Fano pairs of dimension 1 and 2.

We then establish our main result.

\begin{alphtheorem}\label{introthm:fun-group-lcy}
Let $(X,\Delta)$ be a log canonical Calabi--Yau pair with $\dim X =2$. Then the orbifold fundamental group $\pi_1^{\rm orb}(X,\Delta)$ admits a normal subgroup of index at most $7200$, that is abelian of rank at most $4$, or nilpotent of length at most $2$ and a quotient, for some $k\ge 1$, of the Heisenberg group
$$H_k:=\langle a,b,c\mid [a,b]=[a,c]=a^k\, [b,c] = 1\rangle.$$
\end{alphtheorem}

It is worth noting that the orbifold fundamental group of any log canonical Fano pair arises as the orbifold fundamental group of a log canonical Calabi--Yau pair too. The converse is not true, and we obtain stronger results for the orbifold fundamental groups of log canonical Fano pairs, see Proposition \ref{prop:lc-Fano-case}.

We later provide examples showing that Theorem \ref{introthm:fun-group-lcy} is optimal: Example \ref{ex:p1-over-elliptic} provides, for each $k\ge 1$, a log canonical Calabi--Yau pair whose orbifold fundamental group is isomorphic to the Heisenberg group $H_k$ (see also \cite[5.45]{HZ87}, \cite[Example 24]{AC23} for similar constructions; we also believe \cite[Question 26]{AC23} to be related and of interest; this also disproves \cite[Conjecture 4.46]{Mor22}); Example \ref{ex:quotient-p1xp1} provides a klt Fano pair whose orbifold fundamental group is isomorphic to the finite group $(\mathfrak{A}_5\times\mathfrak{A}_5)\rtimes\zz/2\zz$ of order $7200$, which has no non-trivial normal abelian subgroup; finally, any abelian fourfold has fundamental group isomorphic to $\zz^4$.

For klt Calabi-Yau surface pairs, the virtual abelianity was already known by Campana and Claudon~\cite{CC14}. To make their result effective, we provide a sharp bound on the index of the normal abelian subgroup. The proof goes by combining the known virtual abelianity with classical facts on group actions on abelian and K3 surfaces. We show the following result.

\begin{alphtheorem}\label{introthm:fun-group-klt}
Let $(X,\Delta)$ be a klt Calabi--Yau pair with $\dim X =2$. Assume that $\Delta =\Delta^{\rm st}$.
Then, the group $\pi_1^{\rm orb}(X,\Delta)$ 
is finite of order at most $3840$, or admits a normal subgroup isomorphic to $\zz^4$ of index at most $96$.
\end{alphtheorem} 

This theorem fails if we drop the klt assumption (see Example~\ref{ex:p1-over-elliptic}) or the standard coefficient assumption (see Example~\ref{ex:p1-over-elliptic-standard}). Again, the bound on the rank is clearly sharp. In the finite case, the bound on the order relates to work by Kondo (see Example \ref{ex:kondo}).

We finally prove an array of necessary criteria for a log canonical Calabi--Yau surface pair to have particularly large orbifold fundamental group.

\begin{alphtheorem}\label{introthm:fun-group-not-virt-ab}
Let $(X,\Delta)$ be a log canonical Calabi-Yau pair with $\dim X = 2$. If the group $\pi_1^{\rm orb}(X,\Delta)$ is not virtually abelian, then there exists a compatible finite Galois cover $p: (\tilde{X},\tilde{\Delta})\to (X,\Delta)$ of degree at most 96 such that the pair $(\tilde{X},\tilde{\Delta})$ is birationally equivalent to a pair $(\mathbb{P}(\mathcal{O}_E\oplus L),s_0+s_{\infty})$, where $E$ is an elliptic curve, $L$ is an ample line bundle on $E$, and $s_0,s_{\infty}$ are the two sections corresponding to the two factors.
\end{alphtheorem}

\begin{alphtheorem}\label{introthm:fun-group-Z4}
Let $(X,\Delta)$ be a log canonical Calabi-Yau pair with $\dim X = 2$. The group $\pi_1^{\rm orb}(X,\Delta)$ admits a subgroup isomorphic to $\mathbb{Z}^4$ {if and only if} there exists an abelian variety $A$ and a finite group $G<\Aut(A)$ such that the pairs $(X,\Delta)$ and $(A/G,\,{\rm Branch}(p))$ are isomorphic, where $p:A\to A/G$ denotes the quotient map.
\end{alphtheorem}

\begin{alphtheorem}\label{introthm:fun-group-Z3}
Let $(X,\Delta)$ be a log canonical Calabi-Yau pair with $\dim X = 2$. If the group $\pi_1^{\rm orb}(X,\Delta)$ admits a subgroup isomorphic to $\mathbb{Z}^3$, but no subgroup isomorphic to $\mathbb{Z}^4$, then there is a compatible finite Galois cover $p:(\tilde{X},\tilde{\Delta})\to (X,\Delta)$, and a birational equivalence of $(\tilde{X},\tilde{\Delta})$ with a pair $(\mathbb{P}(\mathcal{O}_E\oplus M),s_0+s_{\infty})$, where $E$ is an elliptic curve, $M$ is a numerically trivial line bundle on $E$, and $s_0,s_{\infty}$ are two sections of this $\pp^1$-bundle.
\end{alphtheorem}

\begin{alphtheorem}\label{introthm:fun-group-inf}
Let $(X,\Delta)$ be a log canonical Calabi--Yau pair with $\dim X = 2$. 
If the group $\pi_1^{\rm orb}(X,\Delta)$ contains a subgroup isomorphic to $\zz^r$ for $r=1$ or $2$, then at least one of the following conditions holds:
\begin{itemize}
    \item[(i)] there is a compatible finite Galois $p:(\tilde{X},\tilde{\Delta})\to (X,\Delta)$, and a birational equivalence of $(\tilde{X},\tilde{\Delta})$ with a pair $(Y,\Delta_Y)$ such that the surface $Y$ admits a fibration onto an elliptic curve, or
    \item[(ii)] there is a compatible finite Galois cover $p:(\tilde{X},\tilde{\Delta})\to (X,\Delta)$, and a birational equivalence of $(\tilde{X},\tilde{\Delta})$ with a pair $(Y,\Delta_Y)$ such that $Y$ admits a $(\mathbb{C}^*)^r$-action preserving $\lfloor\Delta_Y\rfloor$.
\end{itemize}
\end{alphtheorem}

\subsection{Jordan property and residual finiteness for fundamental groups}
Along the way to establishing our main results, we prove several results on orbifold fundamental groups of dlt Fano pairs. The following theorem is key, and holds in any dimension $n$. For $n\ge 1$, we denote by $J(n)$ the {\it universal Jordan constant} for birational self-maps of rationally connected klt projective varieties, which exists by \cite[Theorem 1.8]{PS14} and by the proof of the BAB conjecture \cite[Theorem 1.1]{Bir21}. Note that the value $J(2)=7200$ was computed by Yasinsky in \cite{Yas19}.

\begin{alphtheorem}\label{introthm:Jordan-dlt-Fano}
Let $(X,\Delta)$ be an $n$-dimensional dlt Fano pair.
If the group $\pi_1^{\rm orb}(X,\Delta)$ is residually finite, then it admits an normal abelian subgroup of index at most $J(n)$.
\end{alphtheorem}

The proof of Theorem~\ref{introthm:Jordan-dlt-Fano} is given in Section~\ref{sec:proofs}. It is not so technical, and can be read independently of the previous sections and results of this article.

Theorem~\ref{introthm:Jordan-dlt-Fano} is very important to us, as it reduces the task 
of proving that the orbifold fundamental group of a dlt Fano surface pair has a normal abelian subgroup of index at most $7200$ to showing that it is a residually finite group. In that way, it translates an ineffective result (residual finiteness) into an effective one (virtual abelianity, with an explicit bound on the index of a birational geometric nature).

\medskip

Residual finiteness is an important property for fundamental groups. In general, a group is residually finite if and only if it embeds into its {\it profinite completion}. Since that the profinite completion of the fundamental group is the algebraic fundamental group, it makes sense that algebro-geometric methods of proofs deal best with fundamental groups which are proven, {\it a priori} or {\it a posteriori}, to be residually finite. By Theorem \ref{introthm:fun-group-lcy}, this is the case of all fundamental groups considered in this paper. We state this in the following corollary.

\begin{alphcor}\label{introcor1}
Let $(X,\Delta)$ be a log canonical Calabi--Yau pair with $\dim X \le 2$.
The group $\pi_1^{\rm orb}(X,\Delta)$ is residually finite.
\end{alphcor}

Some examples of non--residually finite fundamental groups of smooth projective varieties are given in \cite{Tol93}.
Note that Corollary \ref{introcor1} furthers the folklore expectation that positivity of the canonical class relates to a ``large'' fundamental group $\pi_1(X)$, whereas positivity of the anticanonical class relates to a ``small'' fundamental group $\pi_1(X)$.

\subsection{Some conjectures}

Finally, we introduce two conjectures that encompass our expectations for higher-dimensional pairs. These are essentially effective versions of the ``Nilpotency conjecture'' introduced and studied in \cite[Conjecture 1.14]{CDY22} for $h$-special quasi-projective varieties, and which generalizes the well-known ``Abelianity conjecture'' due to Campana \cite[Conjecture 13.10.(2)]{Cam11}, \cite[11.2]{Cam11bis}. We also refer to \cite[Conjecture 23, Example 24]{AC23} to motivate this conjecture, its name, and its expected strengthening in the setting of Fano pairs.

The first conjecture predicts that the orbifold fundamental groups of $n$-dimensional lc Fano pairs are virtually abelian, with expected rank and index bounded universally in terms of the dimension $n$.

\begin{introcon}
\label{conj:lc-Fano}
For every $n\ge 1$, there is a constant $c(n)$ such that: For any $n$-dimensional log canonical Fano pair $(X,\Delta)$, the orbifold fundamental group $\pi_1^{\rm orb}(X,\Delta)$ contains a normal abelian subgroup of rank at most $n$ and index at most $c(n)$.
\end{introcon}

The second conjecture predicts that the orbifold fundamental groups of $n$-dimensional log canonical Calabi--Yau pairs are virtually nilpotent, with expected rank and index bounded in terms of the dimension $n$.

\begin{introcon}\label{conj:lc-CY}
For every $n\ge 1$, there is a constant $k(n)$ such that: For any $n$-dimensional log canonical Calabi--Yau pair $(X,\Delta)$, the orbifold fundamental group $\pi_1^{\rm orb}(X,\Delta)$ contains a normal subgroup that is nilpotent of rank at most $2n$ and has index at most $k(n)$.
\end{introcon}


\subsection{Sketch of the proof of Theorem \ref{introthm:fun-group-lcy}} We represent the diagram of dependence between the different results of this article in Figure \ref{fig:dependence}. There, we also credit three especially influential papers, which we thoroughly use in our proofs. As Theorem \ref{introthm:fun-group-lcy} is the most central of our results, we will explain our strategy to prove it. 

\begin{figure}
\[\begin{tikzcd}[cells={nodes={draw}}, column sep = 0.3cm, row sep = 1cm]
\textup{Campana--Claudon \cite{CC14}} \arrow[d] & & & 
    & \textup{Prokhorov--Shramov \cite{PS14}} \arrow[d] \\
\textup{Theorem \ref{introthm:fun-group-klt}} \arrow[drr] & 
    & \textup{Braun \cite{Bra20}} \arrow[d] & 
    & \textup{Theorem \ref{introthm:Jordan-dlt-Fano}} \arrow[dll]\\
& & \textup{Theorem \ref{introthm:fun-group-lcy}}
    \arrow[dr] \arrow[drr] \arrow[dl] \arrow[dll]& &\\
\textup{Theorem \ref{introthm:fun-group-not-virt-ab}} 
    & \textup{Theorem \ref{introthm:fun-group-Z4}} 
    & 
    & \textup{Theorem \ref{introthm:fun-group-Z3}}
    & \textup{Theorem \ref{introthm:fun-group-inf}}
\end{tikzcd}\]
    \caption{Diagram of dependences between our main results}
    \label{fig:dependence}
\end{figure}

Our starting point is twofold. First, we use birational geometry to reduce to case analysis. The idea is as follows. The orbifold fundamental group of a log canonical pair is unchanged by mild log resolutions (see Lemma \ref{lem:dlt-mod-vs-fun}), and it only increases through regular birational contractions (see Lemma \ref{lem:fun-under-surj}). Hence, to prove Theorem \ref{introthm:fun-group-lcy} for a particular pair $(X,\Delta)$, it suffices to prove it for the end-product $(X_0,\Delta_0)$ of a minimal model program run from a dlt modification $(X',\Delta')$ of $(X,\Delta)$. Let us make the following key remark: Of course, there is no point in running a $(K_{X'}+\Delta')$--MMP if the pair $(X',\Delta')$ is Calabi--Yau. What makes more sense is to pick an auxiliary divisor $\Delta_{{\rm aux}}$ on $X'$ so that the pair $(X',\Delta_{{\rm aux}})$ is log canonical, and run a $(K_{X'}+\Delta_{{\rm aux}})$-MMP on the surface $X'$. Denoting by $X_0$ the last surface obtained by that MMP, and by $\Delta_0$ the pushforward to $\Delta'$, we have our pair $(X_0,\Delta_0)$. We will apply this procedure several times in the course of the proof and choose various auxiliary divisors when running MMPs.

Second, we establish in Lemma \ref{subsec:galois} a Galois correspondence between normal subgroups of finite index in orbifold fundamental groups, and so-called {\it compatible} finite Galois covers of pairs. This allows to understand properties of the orbifold fundamental group of a given pair in terms of the geometry of some of its finite Galois covers.

\medskip

From now on, we fix a log canonical Calabi--Yau surface pair $(X,\Delta)$, and we show that its orbifold fundamental group is described by Theorem \ref{introthm:fun-group-lcy}.

We take a dlt modification $(X',\Delta')$ of $(X,\Delta)$, then run a $K_{X'}$-MMP from it. It terminates with a log canonical Calabi--Yau surface pair $(X_0,\Delta_0)$, and the surface $X_0$ has klt singularities. We want to show that $\pi_1^{\rm orb}(X_0,\Delta_0)$ is described by Theorem \ref{introthm:fun-group-lcy}. We distinguish three cases depending on the reason of termination of this MMP.

\medskip

\noindent\underline{\textit{Case 1:}} The MMP terminates because $K_{X_0}$ is nef.

Since $(X_0,\Delta_0)$ is a Calabi--Yau pair, the canonical $K_{X_0}$ is then anti-effective and nef, hence numerically trivial. Thus, $\Delta_0=0$ too. This case is covered by Theorem \ref{introthm:fun-group-klt}, which is proved in (the independent) Section~\ref{sec:klt-cy-stand} using Campana--Claudon \cite{CC14}, the Beauville--Bogomolov decomposition for smooth surfaces, and descriptions of automorphism groups of abelian and K3 surfaces due to Fujiki~\cite{Fujiki88} and Mukai~\cite{Mukai88}.

\medskip

\noindent\underline{\textit{Case 2:}}
The MMP terminates because $X_0$ is a Mori fiber space on a curve.

This case is completely studied in Section~\ref{sec:mfs}.
Using the canonical bundle formula and an variant on Nori's trick (see Lemma~\ref{lem:nori}), we have an exact sequence:
$$\pi_1^{\rm orb}(F,\Delta|_F)\to \pi_1^{\rm orb}(X_0,\Delta_0)\to \pi_1^{\rm orb}(C,\Delta_C)\to 1,$$ 
where $F$ is the general fiber of the Mori fiber space $f:X_0\to C$, and the pair $(C,\Delta_C)$ is induced by $f$ on the base of the Mori fiber space (see Definition~\ref{def:mult-divisor}).

The preliminary Section \ref{section:trichotomy} provides a good understanding of what the groups $\pi_1^{\rm orb}(F,\Delta|_F)$ and $\pi_1^{\rm orb}(C,\Delta_C)$ can be: Each of them is finite of order at most $60$, infinite cyclic, infinite dihedral, or containing $\mathbb{Z}^2$ as a normal subgroup of index at most $6$ (this third case is what we call the {\it elliptic type} of curve pairs). 
To paste both pieces of information together and describe the extension $\pi_1^{\rm orb}(X_0,\Delta_0)$ takes more work.

\medskip

\hfill \begin{minipage}{0.95\textwidth}
\noindent\underline{\textit{Case 2.1:}} The easiest case is when neither $(F,\Delta|_F)$, nor $(C,\Delta_C)$ is of elliptic type. That assumption resembles the ``slope rationally connected'' assumption outlined in \cite[Conjecture 23]{AC23}, and is handled by a variation on Theorem \ref{introthm:Jordan-dlt-Fano}.

\medskip

\noindent\underline{\textit{Case 2.2:}} The next case is when the base $(C,\Delta_C)$ is of elliptic type. By a base change, we reduce to the case when $X_0\simeq \mathbb{P}_E(V)$ is the projectivization of a rank two vector bundle $V$ on an elliptic curve $E$, and the divisor $\Delta_0$ is purely horizontal. This case can be settled by writing the divisor $\Delta_0$ explicitly in terms of sections and multi-sections of the $\mathbb{P}^1$-bundle.

\medskip

\noindent\underline{\textit{Case 2.3:}} The last case is when the base $(C,\Delta_C)$ is not of elliptic type, but the fiber $(F,\Delta|_F)$ is. This case comprises several examples with distinct geometric properties, notably a product case $$(X_0,\Delta_0)\simeq \left(\mathbb{P}^1\times \mathbb{P}^1,\, f_0+f_{\infty}+\frac{1}{2}(s_0+s_1+s_2+s_{\infty})\right),$$ where $s_p := \{p\}\times \mathbb{P}^1, f_p := \mathbb{P}^1\times \{p\}$, with the second projection as the preferred Mori fiber space, and a non-klt case 
$$(X_0,\Delta_0)\simeq \left(\mathbb{F}_2,\, \frac{1}{2}(s^+_0+s^+_1+s^+_2+s^+_3)\right),$$ 
where $\mathbb{F}_2$ is the Hirzebruch surface of degree $2$ with its unique fibration to $\mathbb{P}^1$, the curves $s^+_0,s^+_1$ are two distinct sections of self-intersection $2$, and the curves $s^+_2,s^+_3$ are any other two members of the pencil generated by $s^+_0$ and $s^+_1$. This case is studied using the Galois correspondence for orbifold fundamental groups, MRC fibrations, a variation on Theorem \ref{introthm:Jordan-dlt-Fano} as in {\it Case 2.1}, and Theorem \ref{introthm:fun-group-klt} as in {\it Case 1}. In one particularly singular case, we have to perform an elementary transformation (also known as a Sarkisov link) on the Mori fiber space to be able to apply these methods.
\end{minipage}

\bigskip

\noindent\underline{\textit{Case 3:}} The MMP terminates because $X_0$ is a klt Fano surface of Picard rank $1$.

This case further splits into subcases. Denoting by $\Delta_0^{\rm st}\le \Delta_0$ the standard approximation of $\Delta_0$, the main two cases are when $\Delta_0^{\rm st} = \Delta_0$, i.e., $(X_0,\Delta_0)$ is an lc Calabi--Yau pair with standard coefficients, and when $\Delta_0^{\rm st} < \Delta_0$, i.e., the pair $(X_0,\Delta_0)$ has the same orbifold fundamental group as the lc Fano pair $(X_0,\Delta_0^{\rm st})$. 


\medskip 

\hfill \begin{minipage}{0.95\textwidth}
\noindent\underline{\textit{Case 3.1:}} The pair $(X_0,\Delta_0)$ is a Calabi--Yau pair with standard coefficients.

\medskip

If the pair is klt, then Theorem \ref{introthm:fun-group-klt} describes its orbifold fundamental group. If the pair is lc but not klt, there is an index one cover of $(X_0,\Delta_0)$ of degree at most $6$. We take a dlt modification $(Y,\Delta_Y)$ of that cover. Note that the surface $Y$ is Gorenstein and has canonical singularities, and the divisor $\Delta_Y$ has integral coefficients.

\medskip

If every compatible finite Galois cover of $(Y,\Delta_Y)$ is rationally connected, then a variant of Theorem \ref{introthm:Jordan-dlt-Fano} applies, and we are left proving that the orbifold fundamental group $\pi_1^{\rm orb}(Y,\Delta_Y)$ is residually finite. We prove that by comparing that group with the orbifold fundamental group of a pair $(Y',\Delta_{Y'})$, that is obtained from an appropriate compatible finite Galois cover of $(Y,\Delta_Y)$ by a running an MMP, and that has a Mori fiber space structure.

\medskip

If there is a compatible finite Galois cover of $(Y,\Delta_Y)$ that is not rationally connected, then there also is a non rationally connected compatible finite Galois cover $(X',\Delta')$ of the initial pair $(X_0,\Delta_0)$. Modifying it to be dlt and running an equivariant $K_{X'}$-MMP, we end up with a pair $(X_1,\Delta_1)$ where the surface $X_1$ is klt and not rationally connected, and has an equivariant Mori fiber space structure. Since klt Fano surfaces are rationally connected, the surface $X_1$ must be an equivariant Mori fiber space onto a curve. Quotienting by the group action, this shows that a dlt modification of $(X_0,\Delta_0)$ has an MMP that terminates with a Mori fiber space onto a curve too, and reduces to {\it Case 2}.
\end{minipage}

\medskip

\hfill \begin{minipage}{0.95\textwidth}
\noindent\underline{\textit{Case 3.2:}} The pair $(X_0,\Delta_0^{\rm st})$ is a Fano pair.

\medskip

If the pair is klt, then Braun \cite{Bra20} shows that its orbifold fundamental group is finite, and Theorem \ref{introthm:Jordan-dlt-Fano} gives the appropriate explicit bounds on the index and rank of a maximal abelian subgroup.

\medskip

If the pair is not klt, we need to distinguish between two subcases, based on the coregularity of the pair $(X_0,\Delta_0^{\rm st})$. Loosely speaking, the coregularity of a pair is a birational invariant that measures how singular it is. The coregularity of the pair $(X_0,\Delta_0^{\rm st})$ ranges in the value set $\{0,1,2\}$. Importantly, if the pair $(X_0,\Delta_0^{\rm st})$ has coregularity {\it two}, then it is a klt pair, a case which we have already treated.
\end{minipage}

\medskip

\hfill \begin{minipage}{0.9\textwidth}
\noindent\underline{\textit{Case 3.2.a:}} The pair $(X_0,\Delta_0^{\rm st})$ has coregularity zero.

\medskip

We can then increase the coefficients of components in $\Delta_0^{\rm st}$ to obtain an $2$-complement, that is an effective divisor $0\le \Delta_0^{\rm st}\le \Gamma$ on $X_0$ such that the pair $(X_0,\Gamma)$ is an lc Calabi--Yau pair, and moreover $2(K_{X_0}+\Gamma)\sim 0$, see \cite{FFMP22}. We can then focus on the orbifold fundamental group of the new pair $(X_0,\Gamma)$. This new pair is in fact an lc Calabi--Yau pair with standard coefficients, so it is covered by {\it Case 3.1}.

\medskip

\noindent\underline{\textit{Case 3.2.b:}} The pair $(X_0,\Delta_0^{\rm st})$ has coregularity one.

\medskip

We take a dlt modification $(X',\Delta')$ of $(X_0,\Delta_0^{\rm st})$, and run a $K_{X'}+\Delta'$-MMP. It terminates with a plt pair $(X_1,\Delta_1)$ that has standard coefficients.
If that MMP terminates because $K_{X_1}+\Delta_1$ is nef, then $(X_1,\Delta_1)$ is in fact a Calabi--Yau pair with standard coefficients, and {\it Case 3.1} concludes. If that MMP terminates because $(X_1,\Delta_1)$ is a Mori fiber space onto a curve, then {\it Case 2} concludes.

\medskip

Otherwise, that MMP terminates because $(X_1,\Delta_1)$ is a plt Fano pair and the surface $X_1$ has Picard rank one.
By a connectedness principle for the non-klt locus \cite{FS20} and since $\rho(X_1) = 1$, there is exactly one irreducible component for $\lfloor \Delta_1\rfloor$, which we denote by $S$. It is a smooth curve by \cite[Lemma 3.6]{Sho92}. We can write $\Delta_1 = S + L$, where $L$ is effective, with standard coefficients strictly smaller than $1$. The adjunction formula provides an lc pair $(S,\Delta_S)$ such that $$(K_{X_1}+\Delta_1)|_S\sim K_S+\Delta_S,$$
which is antiample. Using a lemma that relates $N$-complements of pairs before and after adjunction (Lemma \ref{lem:lifting-complements}), we deal with several cases, and we are left with three cases to study:
The cases when $(S,\Delta_S)\simeq \left(\mathbb{P}^1,\frac{1}{2}\{0\}+\frac{2}{3}\{1\}+\frac{n-1}{n}\{\infty\}\right)$ for $n=3,4,$ and $5$. Note that contributions to $\Delta_S$ come from singular points of $X_1$ lying on $S$, and from intersection points of $L$ and $S$.

\medskip

From here on, our strategy is to prove that the orbifold fundamental group $\pi_1^{\rm orb}(X_1,S+L)$ is residually finite, and to apply Theorem \ref{introthm:Jordan-dlt-Fano} to conclude. If $L=0$, this is proven by Keel and M\textsuperscript{c}Kernan~\cite[Theorem 1.4]{KM92}. Otherwise, we can list various cases depending on the way singular points of $X_1$ and intersection points with $L$ are arranged on the curve $S\simeq\mathbb{P}^1$, accounting for the pair obtained by adjunction $(S,\Delta_S)\simeq\left(\mathbb{P}^1,\frac{1}{2}\{0\}+\frac{2}{3}\{1\}+\frac{n-1}{n}\{\infty\}\right)$ for $n=3,4,$ or $5$. In each case, we use our lemma that relates $N$-complements of pairs before and after adjunction (Lemma \ref{lem:lifting-complements}) to provide a particular and explicit complement of the pair $(X_1,S+L)$. In most cases, that reduces us to the toric setting, using the well-known criterion by Brown, M\textsuperscript{c}Kernan, Svaldi, Zhong \cite{BMSZ18}, and some {\it ad hoc} toric arguments conclude. In a few remaining cases, we have to dwelve further into the explicit geometry of the pair and find a related Mori fiber space setting, using {\it Case 2} to ultimately conclude.
\end{minipage}

\subsection*{Acknowledgements}
The authors would like to thank 
Rodolfo Aguilar Aguilar for an interesting discussion on virtual solvability {\it versus} virtual nilpotency of fundamental groups in conjectures in the literature. We are also grateful for
Lukas Braun, 
Fr\'ed\'eric Campana,
Fabrizio Catanese,
Beno\^it Claudon, 
Gavril Farkas,
Fernando Figueroa,
Pedro Nu\~nez,
Mirko Mauri,
and De-Qi Zhang
for many useful comments and questions.

Gachet was supported by the ERC Advanced Grant SYZYGY during part of this project. This project has received funding from the European Research Council (ERC) under the European Union Horizon 2020 research and innovation program (grant agreement No. 834172).

\section{Preliminaries}

We work over the field $\cc$ of complex numbers. A {\it pair} $(X,\Delta)$ is the data of a normal projective variety $X$, and of a formal sum $\Delta$ of irreducible divisors with positive rational coefficients. For basic concepts regarding pairs and their singularities in birational geometry
we refer the reader to the book~\cite{Kol13}. 

We define the {\it rank} of a finitely generated group $G$ as the least number of generators of $G$. We define the {\it torsion-free rank} of an abelian group $A$ as its dimension as a $\zz$-module. For a group $G$, and for some elements $g_1,\dots,g_r$ of it, 
we denote by $\langle g_1,\dots,g_r\rangle_n$ the smallest normal subgroup of $G$ containing the elements $g_1,\dots,g_r$.

In Subsection~\ref{subsec:orb-fun}, we introduce and prove several results related to orbifold fundamental groups; In Subsection~\ref{subsec-rf}, we prove elementary facts about residually finite groups. In Subsection~\ref{subsec:ft-cy}, we review useful results regarding the geometry of Fano and Calabi--Yau pairs. 

\subsection{The orbifold fundamental group}\label{subsec:orb-fun}
We first recall the definition of the {\it orbifold fundamental group} of a pair. To the authors' knowledge, it appears first in \cite[14.10.1]{DM93} \cite[Definition 4.4]{Cat00}. It is a global notion. A local analogue is the {\it regional fundamental group} introduced by \cite{Bra20}, which we briefly define at the end of the subsection.

\begin{definition}\label{def:reg-fund}
For any real number $a\in [0,1]$, we define its standard round-down
\[
a^{\rm st}:=\max\left\{1-\frac{1}{m}\leq a\, \biggr\rvert\, m\in \mathbb{N}\cup \infty\right\}.
\]
For any effective divisor $\Delta$, with coefficients smaller or equal to one, on a quasiprojective variety $X$, we define its {\em standard approximation} as $$\Delta^{\rm st}:=\sum_{i=1}^k {a_i}^{\rm st}\, \Delta_i,$$
where the $\{\Delta_i\}_{1\le i\le k}$ are the prime components of the support of $\Delta$, with their respective coefficients $a_i$.
We say that a pair $(X,\Delta)$ {\em has standard coefficients} if it holds ${a_i}^{\rm st}=a_i$ for all $1\le i\le k$.
\end{definition}

\begin{definition}
    Let $(X,\Delta)$ be a pair, and let $P$ be a prime divisor on $X$. We define the {\it orbifold index} of $\Delta$ at $P$ as the largest value $m\in\mathbb{N}\cup\{\infty\}$ such that the coefficient of the prime component $P$ in $\Delta$ is at least $1 -\frac{1}{m}$.
\end{definition}

\begin{definition}
Let $(X,\Delta)$ be a pair.
The {\em orbifold fundamental group} of the pair is denoted by $\pi_1^{\rm orb}(X,\Delta)$, and defined as the quotient group
\[
\pi_1(X_{\rm reg}\setminus {\rm Supp}\,\Delta^{\rm st})\,\biggr/\,\left\langle \gamma_i^{m_i} \,\biggr\vert\, 1\le i\le k \mbox{ such that } {a_i}^{\rm st} = 1-\frac{1}{m_i} < 1\right\rangle_n, 
\]
where $\gamma_i$ denotes a loop around the prime component $\Delta_i$.
\end{definition}

In Appendix \ref{appA}, we justify what is meant by ``a loop around'' a prime divisor, and explain how this definition of the orbifold fundamental group is independent of any choice of loops.

The following lemma follows from the definition.

\begin{lemma}\label{lem:fun-group-increasing-coeff}
Let $(X,\Delta)$ and $(X,\Gamma)$ be two pairs with standard coefficients, such that
for every prime divisor $P$ on $X$, the orbifold index of $\Delta$ at $P$
divides the orbifold index of $\Gamma$ at $P$.
Then, there is a surjective group homomorphism $\pi_1^{\rm orb}(X,\Gamma)\twoheadrightarrow \pi_1^{\rm orb}(X,\Delta)$.
\end{lemma}

\begin{proof}
For every prime divisor $P$ in $X$, we denote by $d_P$ and $n_P$ its orbifold index in $\Delta$, respectively $\Gamma$: By assumption, $d_P$ divides $n_P$ for all $P$. In particular, any divisor $P$ in the support of the standard approximation $\Delta^{\rm st}$ has $d_P \ge 2$, hence $n_P\ge 2$, hence is contained in the support of $\Gamma^{\rm st}$.
Hence, there are two surjective group homomorphisms
\[\xymatrixcolsep{0mm}
\xymatrix{
& \pi_1(X_{\rm reg}\setminus {\rm Supp}\,\Gamma^{\rm st}) \ar@{->>}[ld]^{g} \ar@{->>}[rd]_{h}\ &\\
\pi_1^{\rm orb}(X,\Gamma) & & \pi_1^{\rm orb}(X,\Delta)\\
}\]
It suffices to factor $h$ through $g$ to conclude, i.e., show that $\ker g\subset \ker h$. The kernel $\ker g$ is generated by elements of the form ${\gamma_P}^{n_P}$, for every prime component $P$ in the support of $\Gamma^{\rm st}$. Fix one such prime component $P$, and note that ${\gamma_P}^{d_P}$ is in $\ker h$. The divisibility $d_P\mid n_P$ concludes.
\end{proof}

\begin{corollary}\label{cor:2cplt-sur}
Let $(X,\Delta)$ be a pair with standard coefficients and let $\Gamma \ge \Delta$ be a divisor with coefficients in $\{0,\frac{1}{2},1\}$.
Then, there is a surjective group homomorphism $\pi_1^{\rm orb}(X,\Gamma)\twoheadrightarrow \pi_1^{\rm orb}(X,\Delta)$.
\end{corollary}

\begin{proof}
For every prime divisor $P$ in $X$, we denote by $d_P$ and $n_P$ its orbifold index in $\Delta$, respectively $\Gamma$. If we can show that $d_P\mid n_P$ for every $P$, then Lemma \ref{lem:fun-group-increasing-coeff} concludes.
Fix a prime divisor $P$ in $X$. Since $\Gamma$ is a $2$-complement, we have that $n_P\in\{1,2,\infty\}$ and $d_P\le n_P$. So $(d_P,n_P)$ must be one of the following pairs:
    $(1,1),(1,2),(2,2),(d,\infty).$
In all of them, the left coordinate divides the right one, as wished.
\end{proof}

We now state some useful results on the behavior of the orbifold fundamental group {\it via} certain finite maps. We also state a Galois correspondence. The proofs are relegated to Appendix \ref{appA}.
We start by a definition.

\begin{definition}\label{def:compatible} 
Let $(X_0,\Delta_0)$ and $(X_1,\Delta_1)$ be two pairs. Let $p:X_1\to X_0$ be a proper morphism that is flat in codimension zero and one in $X_1$.
We say that $p$ is \emph{compatible} with the two pairs if it holds
    $$p^*\Delta_0^{\rm st}={\rm Ram}(p)+{\Delta}_1^{\rm st} \mbox{ and }p^*\Delta_0={\rm Ram}(p)+\Delta_1.$$
\end{definition}

\begin{remark}
    The following proper morphisms are flat in codimension zero and one in the source:
    \begin{enumerate}
        \item a finite surjective morphism $p:X'\to X$, with $X'$ and $X$ normal projective varieties;
        \item an equidimensional fibration $f:X\to B$,  with $X$ and $B$ normal projective varieties.
    \end{enumerate}
\end{remark}

We now state the existence of a functorial pushforward.

\begin{proposition}[\protect{=Propositions \ref{prop:can_compose} and~\ref{prop:functorialpf}}]
   Let $(X_0,\Delta_0), (X_1,\Delta_1)$ be two pairs. Any compatible morphism $p:(X_1,\Delta_1) \to (X_0,\Delta_0)$ induces a group homomorphism
$p_{\bullet}:\pi_1^{\rm orb}(X_1,\Delta_1)\to\pi_1^{\rm orb}(X,\Delta)$ by pushforward. Moreover, for any two compatible morphisms
$$(X_2,\Delta_2)\overset{q}{\longrightarrow} (X_1,\Delta_1) \overset{p}{\longrightarrow} (X_0,\Delta_0)$$
whose composition $p\circ q$ is also compatible, it holds $(p\circ q)_{\bullet} = p_{\bullet}\circ q_{\bullet}$.
 \end{proposition}

The next two results constitute a Galois correspondence for orbifold fundamental groups.
 
\begin{proposition}[\protect{=Proposition~\ref{prop:pushforward}}]
\label{prop:galois1}
    Let $p:(X',\Delta') \to (X,\Delta)$ be a compatible morphism between two pairs. If it also is a finite Galois cover, then we have an exact sequence of groups
    $$1\longrightarrow \pi_1^{\rm orb}(X',\Delta') \overset{p_{\bullet}}{\longrightarrow} \pi_1^{\rm orb}(X,\Delta) \longrightarrow {\rm Gal}(p) \longrightarrow 1.$$
\end{proposition}

\begin{proposition}[\protect{=Lemma~\ref{lem:galoiscorr}}]
\label{prop:galois2}
    Let $(X,\Delta)$ be a pair. Let $H$ be a normal subgroup of finite index in $\pi_1^{\rm orb}(X,\Delta)$. Then there exist a pair $(X_H,\Delta_H)$, and a finite Galois cover $p_H:X_H\to X$ compatible with the two pairs, such that the subgroups $H$ and
    ${p_H}_{\bullet}[\pi_1^{\rm orb}(X_H,\Delta_H)]$ coincide in $\pi_1^{\rm orb}(X,\Delta)$.
\end{proposition}

The next few lemmas describe the behavior of the orbifold fundamental group with respect to certain birational morphisms.

\begin{lemma}\label{lem:Qfact-mod-vs-fun}
Let $(X,\Delta)$ be an lc pair, and let $\varepsilon:(Y,\Delta_Y)\to (X,\Delta)$ be a small proper birational morphism.
Then the groups
$\pi_1^{\rm orb}(X,\Delta)$ and $\pi_1^{\rm orb}(Y,\Delta_Y)$ are isomorphic.
\end{lemma}

\begin{proof}
Let $E$ be the reduced exceptional locus of $\varepsilon$. Since $\varepsilon$ is small, the Zariski-closed set $E$ has codimension at least two in $Y$. Hence, excising $E\cap Y_{\rm reg}$ does not affect the fundamental group, and therefore
$$\pi_1^{\rm orb}(Y,\Delta_Y) 
= \pi_1^{\rm orb}(Y\setminus E, \Delta_Y)
\cong \pi_1^{\rm orb}(X\setminus\varepsilon(E), \Delta).$$
Since $\varepsilon(E)$ has codimension at least two in $X$,
we also have
$$\pi_1^{\rm orb}(X\setminus\varepsilon(E), \Delta) = \pi_1^{\rm orb}(X, \Delta),$$ 
which concludes this proof.
\end{proof}

\begin{lemma}\label{lem:dlt-mod-vs-fun}
Let $(X,\Delta)$ be a log canonical pair, and let $\varepsilon:(Y,\Delta_Y)\to (X,\Delta)$ be a dlt modification.
Then the groups
$\pi_1^{\rm orb}(X,\Delta)$ and $\pi_1^{\rm orb}(Y,\Delta_Y)$ are isomorphic.
\end{lemma}

\begin{proof}
Let $E$ be the reduced exceptional divisor of $\varepsilon$. Since $\varepsilon$ is a dlt modification, the whole divisor $E$ appears with coefficient $1$ in $\Delta_Y$. Hence,
$$\pi_1^{\rm orb}(Y,\Delta_Y) = \pi_1^{\rm orb}(Y\setminus E, \Delta_Y - E)\cong \pi_1^{\rm orb}(X\setminus\varepsilon(E), \Delta).$$
Since $\varepsilon(E)\cap X_{\rm reg}$ has codimension at least two in $X_{\rm reg}$,
hence removing it does not affect the fundamental group, in the sense that $\pi_1(X_{\rm reg}\setminus \varepsilon(E)\cup {\rm Supp}\,\Delta)=
\pi_1(X_{\rm reg}\setminus {\rm Supp}\,\Delta))$.
Hence, we have 
$$\pi_1^{\rm orb}(X\setminus\varepsilon(E), \Delta) = \pi_1^{\rm orb}(X, \Delta),$$ 
which concludes this proof.
\end{proof}

\begin{lemma}\label{lem:fun-under-surj}
Let $(X,\Delta)$ be a log canonical pair.
Let $\varphi:X\rightarrow Y$ be a proper birational morphism,
and let $\Delta_Y:=\phi_*\Delta$.
Then there is a surjective homomorphism
$\pi_1^{\rm orb}(Y,\Delta_Y)\twoheadrightarrow 
\pi_1^{\rm orb}(X,\Delta). $
\end{lemma}

\begin{proof}
Let $E$ be the reduced exceptional divisor of $\varphi$, and let $B$ be the strict transform of $\Delta_Y$ by $\varphi$. Clearly, $\Delta^{\rm st} - B^{\rm st}$ is effective and $\varphi$-exceptional. Hence, by Lemma \ref{lem:fun-group-increasing-coeff}, there is a surjective homomorphism
$$\pi_1^{\rm orb}(X, B + E)=\pi_1^{\rm orb}(X, B^{\rm st} + E)
\twoheadrightarrow\pi_1^{\rm orb}(X,\Delta^{\rm st})=\pi_1^{\rm orb}(X,\Delta).$$
Since $\varphi(E)\cap Y_{\rm reg}$ has codimension at least two in $Y_{\rm reg}$, removing it does not affect fundamental groups. Hence, we have
$$\pi_1^{\rm orb}(Y, \Delta_Y) = \pi_1^{\rm orb}(Y\setminus \varphi(E), \Delta_Y) \cong  \pi_1^{\rm orb}(X\setminus E, B) = \pi_1^{\rm orb}(X, B + E),$$
which concludes this proof.
\end{proof}


\begin{lemma}\label{lem:fun-under-surj-smooth}
Let $X$ and $Y$ be normal projective surfaces, and let $\varphi\colon X\rightarrow Y$ be a proper birational morphism. 
If the image of the reduced exceptional divisor $E$ of $\varphi$ lies in $Y_{\rm reg}$, then the fundamental groups $\pi_1(X_{\rm reg})$ and $\pi_1(Y_{\rm reg})$ are isomorphic.
\end{lemma}

\begin{proof}
    Blowing-up a smooth point leaves fundamental groups invariant, and since we have strong factorization for proper birational morphisms between surfaces, we see that the quasiprojective variety $X\setminus\varphi^{-1}(Y_{\rm sing})$ is smooth, and that
    $$\pi_1(Y_{\rm reg}) \cong \pi_1(X\setminus\varphi^{-1}(Y_{\rm sing})).$$
    Since $\varphi(E)$ is contained in $Y_{\rm reg}$, the preimage $\varphi^{-1}(Y_{\rm sing})$ consists in finitely many points on $X$ and is contained in $X_{\rm sing}$. But since we know that $X\setminus\varphi^{-1}(Y_{\rm sing})$ is smooth, we must have $\varphi^{-1}(Y_{\rm sing})=X_{\rm sing}$.
\end{proof}

We now present a type of fibrations that naturally induce exact sequences at the level of orbifold fundamental groups, relating the orbifold fundamental groups on the fiber, the base, and on the total space.

\begin{definition}\label{def:mult-divisor}
    Let $(X,\Delta)$ be a pair, and $f: X\to C$ be an equidimensional fibration. Denote by $\Delta_{\rm vert}$ the $f$-vertical part of the divisor $\Delta$. We say that $f$ is an {\it equimultiple fibration} if there exists a divisor $\Delta_C$ on $C$ with standard coefficients such that $f$ is compatible with the pairs $(X,{\Delta_{\mathrm{vert}}}^{\rm st})$ and $(C,\Delta_C)$.
    The (unique) such divisor $\Delta_C$ on $C$ is called the divisor {\it induced by $f$}.
\end{definition}

\begin{lemma}\label{lem:mfscompatible}
    Let $(X,\Delta)$ be a pair with $X$ a surface, and let $f:X\to C$ be a Mori fiber space onto a curve. Then $f$ is equimultiple.
\end{lemma}

\begin{proof}
    Since $f$ is a Mori fiber space from a surface to a curve, it has irreducible (not necessarily reduced) fibers. In particular, we can write the divisor
    ${\Delta_{\rm vert}}^{\rm st} + {\rm Ram}(f)$ as a pullback of an effective divisor $\Delta_C$ on $C$. We are left checking that $\Delta_C$ has standard coefficients.
    Let $p$ be a point on $C$. Let $m$ be the multiplicity of the fiber of $f$ at $p$. Let $n$ be the orbifold index of $D$ in ${\Delta_{\rm vert}}^{\rm st}$. It is then easy to check that the point $p$ has coefficient $1-\frac{1}{nm}$ in $\Delta_C$, a standard coefficient indeed.
\end{proof}

The following remark is an easy consequence of the canonical bundle formula.

\begin{remark}\label{rem:mfscompatible}
   If $f:(X,\Delta)\to (C,\Delta_C)$ is an equimultiple fibration from a surface to a curve, and the pair $(X,\Delta)$ is 
    \begin{itemize}
        \item log canonical Calabi-Yau, then $(C,\Delta_C)$ is a log canonical Fano or log canonical Calabi--Yau pair;
        \item log canonical Fano, then $(C,\Delta_C)$ is a log canonical Fano pair.
    \end{itemize}
\end{remark}

The next lemma may be called {\it Nori's trick}, as a tribute to a result by Nori \cite[Lemma 1.5.C]{Nor83}. The version we use here is closer to \cite[Lemma 3.13]{FM23}.

\begin{lemma}\label{lem:nori}
Let $f:(X,\Delta)\to (C,\Delta_C)$ be an equimultiple fibration. Let $i:F\hookrightarrow X$ be the inclusion of a general fiber of $f$. Then the following sequence is exact:
$$\pi_1^{\rm orb}(F,\Delta|_F)\overset{i_{\bullet}}{\longrightarrow} \pi_1^{\rm orb}(X,\Delta)\overset{f_{\bullet}}{\longrightarrow} \pi_1^{\rm orb}(C,\Delta_C)\to 1.$$
\end{lemma}

We finally recall a definition due to Braun \cite{Bra20}.

\begin{definition}
The {\em regional fundamental group} 
of a pair $(X,\Delta)$ at a closed point $x\in X$ is
denoted by $\pi_1^{\rm reg}(X,\Delta;x)$, and is defined as the inverse limit of all groups $\pi_1^{\rm orb}(U,\Delta|_U)$, where $U$ ranges over all analytic neighborhoods of $x$ in $X$.
\end{definition}

\subsection{Basic facts on residually finite groups}\label{subsec-rf} 
We recall a definition

\begin{definition}
A group $G$ is {\it residually finite} if it admits a set $(N_i)_{i\in I}$ of normal subgroups of finite index such that $$\bigcap_{i\in I} N_i=\{1\}.$$
\end{definition}

\begin{remark}
We make note of the following fact: For any subgroup $H$ of finite index in a group $G$, there is a normal subgroup of finite index $N \triangleleft G$ such that $N\subset H$.
Therefore, a group $G$ is residually finite if and only if it admits a set $(H_i)_{i\in I}$ of subgroups of finite index such that $$\bigcap_{i\in I} H_i=\{1\}.$$

 A finitely presented group only admits finitely many normal subgroups of a given index $N\in\nn$, thus countably many normal subgroups of finite, unspecified index. Therefore, a finitely presented group $G$ is residually finite if and only if it admits a decreasing sequence $(H_n)_{n\in \nn}$ of subgroups of finite index such that
    $$\bigcap_{n\in \nn} H_n=\{1\}.$$
\end{remark}

\begin{lemma}\label{lem:resfinimage}
    The image of a residually finite group by a group homomorphism is residually finite.
\end{lemma}

\begin{proof}
    It suffices to show that residual finiteness descends by surjective group homomorphisms. The fact that surjective group homomorphisms send normal subgroups to normal subgroups concludes.
\end{proof}

\begin{lemma}\label{lem:resfinsubgroup}
    If a group is residually finite, then every subgroup of finite index in it is residually finite. If a group has a subgroup of finite index that is residually finite, then the group itself is residually finite.
\end{lemma}

\begin{proof}
    Both assertions are easy to check, using the following fact: For any subgroup $H$ of finite index in a group $G$, there is a normal subgroup of finite index $N \triangleleft G$ such that $N\subset H$.
\end{proof}

\begin{lemma}\label{lem:res-finite-from-exact}
Let $G$ be a finitely generated group. Assume that there is an exact sequence
$$K\to G\to Q\to 1,$$
where $K$ is residually finite and $Q$ is virtually abelian.
Then, the group $G$ is residually finite.
\end{lemma}

\begin{proof}
By Lemma \ref{lem:resfinimage}, we can replace $K$ by the image of $K\to G$ in $G$, which is still residually finite. By Lemma \ref{lem:resfinsubgroup}, we can replace $Q$ by any subgroup of finite index of itself, and replace $G$ by a subgroup of finite index of itself accordingly, which is residually finite if and only if $G$ is.

Using the classification of finitely generated abelian groups, we thus reduce to the following situation
\begin{equation}\label{eq:2.22exseq}
   1\to K\to G\to \mathbb{Z}^r\to 1, 
\end{equation}
where $K$ is residually finite, $r\ge 0$, and we want to prove that $G$ is residually finite.

If $r=0$, this is obvious.
Let us prove the lemma for $r=1$.
Let $a\in G$ be a pre-image of $1\in\mathbb{Z}$. Let $(K_n)_{n\in \nn}$ be a decreasing sequence of normal subgroups of finite index in $K$, with trivial intersection. Consider the subgroups $H_n:=\langle a^{n!},K_n\rangle$ in $G$. They form a decreasing sequence of subgroups of $G$, and clearly have trivial intersection. Now, $H_n$ is a subgroup of finite index in $\langle a^{n!},K\rangle$, which is itself a normal subgroup of finite index in $G$, so $H_n$ is a subgroup of finite index in $G$. This concludes the case $r=1$.

We now proceed by induction. Assume that $r\ge 2$, and that the lemma already holds for any exact sequence of the form
$$1\to K'\to G'\to \zz^{s}\to 1,$$
with $K'$ residually finite and $0\le s < r$.
In the exact sequence \eqref{eq:2.22exseq}, we fix an element $b\in\zz^r$ such that $\zz^r/\langle b\rangle\simeq \zz^{r-1}$, and choose a preimage $a\in G$ of it. We obtain two exact sequence
$$1\to K\to \langle a, K\rangle \to\zz\to 1 \mbox{, and}$$
$$1\to \langle a, K\rangle \to G\to\zz^{r-1}\to 1.$$
We use the induction hypothesis for $s=1$ and $s=r-1$ to conclude.
\end{proof}

We prove a lemma about the rank of a finitely generated abelian group.

\begin{lemma}\label{lem:ab-group}
Let $G$ be a finitely generated abelian group, and let $r\in\nn$.
If there is a decreasing sequence of normal subgroups $(H_n)_{n\in \nn}$ with trivial intersection such that each quotient $G/H_n$ has rank at most $r$, then $G$ has rank at most $r$.
\end{lemma}

\begin{proof}
By the structure theorem for finitely generated abelian groups, we have $G\simeq \zz/{m_1}\zz \oplus \dots \oplus \zz/{m_s}\zz \oplus \zz^k $ where $m_1\mid \dots \mid m_s$. The rank of $G$ is then $s+k$.
Consider the subgroup 
$H = (m_{s}\zz)^k\subset \zz^k\subset G$.
Since $G/H$ is finite and the sequence $(H_n)_{n\in\nn}$ is decreasing, the sequence $(H_n/H\cap H_n)_{n\in\nn}$ is stationary. Since the sequence $(H_n)_{n\in\nn}$ has trivial intersection, the sequence of the quotients $H_n/H\cap H_n$ stations at the trivial group, i.e., $H_n \subset H$ for all $n$ large enough.
Hence, for $n$ large enough, there is a surjective homomorphism
$$G/H_n \twoheadrightarrow G/H\simeq \zz/{m_1}\zz \oplus \dots \oplus (\zz/{m_{s}}\zz)^{k+1}.$$
Comparing the ranks on each side, we get $s+k\le r$, as wished.
\end{proof}

We conclude this section with a few constructions of normal subgroups of reasonable index in a group that satisfy a certain type of exact sequence.

\begin{lemma}\label{lem:exact-seq}
 Let $G$ be a group sitting in a short exact sequence
    $$1\to N\to G\overset{s}{\longrightarrow} Q\to 1,$$
where the group $N$ is abelian and satisfies $|{\rm Aut}(N)| = \lambda\in \nn$, and the group $Q$ contains a normal abelian subgroup of finite index $\kappa$ isomorphic to $\mathbb{Z}^r$ for some $r\ge 1$. Then $G$ contains a normal subgroup $H$ of index $\kappa\lambda^r$ which sits in a central extension
$$1\to N\to H\to \zz^r\to 1.$$
\end{lemma}

\begin{proof}
    We consider the normal subgroup $(\lambda\zz)^r\triangleleft Q$ and its preimage $H$ by $s$ in $G$. It is normal and of index $\kappa\lambda^r$. We clearly still have an exact sequence
    $$1\to N\to H\to \zz^r\to 1.$$
    
    Fix $h\in H$. There exist $g\in G$ and $z\in N$ such that $h = zg^{\lambda}$. Since $N$ is abelian, conjugation by $z\in N$ is trivial on it. Since $N$ is normal, conjugation by $g$ is an automorphism of $N$, thus has order dividing the group cardinal $|{\rm Aut}(N)|=\lambda$. Hence, conjugation by $h$ acts trivially on $N$, as wished.
\end{proof}

\subsection{Adjunction, coregularity, complements, and complexity as a characterization of toric pairs}\label{subsec:ft-cy}
We first recall the adjunction formula for plt pairs, as it can be found in~\cite[Proposition 3.9]{Sho92}.

\begin{lemma}\label{lem:coeff-under-adj}
Let $(X,\Delta)$ be a plt pair. Assume that $\Delta=S+B$ with $S$ a prime divisor and $B=\sum_{j=1}^k b_jB_j$ an effective $\qq$-divisor. We define an effective $\mathbb{Q}$-divisor on $S$
\[
\Delta_S := \sum_{P\subset S}\left( 1-\frac{1}{m_P}+\sum_{j=1}^k \frac{m_{j,P}b_j}{m_P}\right)P,
\]
summing over all prime divisors $P$ in $S$, denoting by $m_P$ the orbifold index of the variety $X$ along $P$ and by
$m_{j,P}$ the multiplicity of $B_j$ along $P$.
Then $(K_X+\Delta)|_S \sim K_S+ \Delta_S.$
\end{lemma}

\begin{remark}\label{rem:plt}
In the notations of Lemma \ref{lem:coeff-under-adj}, the component $S$ is normal due to the plt assumption, see \cite[Lemma 3.6]{Sho92}.
\end{remark}

We now turn to dual complexes, used to define the important notion of coregularity of an lc pair.

\begin{definition}
Let $(X,\Delta)$ be a log canonical pair. 
Consider a log resolution of it $\varepsilon\colon (Y,\Delta_Y)\rightarrow (X,\Delta)$ so that it holds $K_Y+\Delta_Y = \varepsilon^*(K_X+\Delta)$.

The {\it dual complex} of $(X,\Delta)$, denoted by $\mathcal{D}(X,\Delta)$, is defined as the abstract simplicial complex given by the dual incidence relations of the irreducible components of $\lfloor\Delta_Y\rfloor$. (Its vertices correspond to irreducible components of $\lfloor\Delta_Y\rfloor$, and $k\ge 2$ of its vertices $E_1,\ldots E_k$ form a $k-1$ dimensional face iff the intersection $E_1\cap\ldots\cap E_k$ is non-empty.)
\end{definition}

The dual complex associated to a germ of a singularity has been present in the literature for a while, see \cite{Gor80}. Stepanov proved in \cite{Ste06} that its homotopy type was independent of the resolution chosen. Over the next few years, these notions were introduced and studied in an increasing degree of generality, see notably \cite{FS20} and references therein. We recall the following results, which will be used thoroughly.

\begin{remark}\label{rem:dual-complex}
    Let $(X,\Delta)$ be a log canonical Calabi-Yau surface pair. Its dual complex clearly has dimension at most $1$. Therefore, by \cite[Theorem 1.6]{FS20}, the dual complex $\mathcal{D}(X,\Delta)$ is piecewise linearly homeomorphic to one of the following manifolds with boundary:
    $$\emptyset,\{0\},\{0,1\},[0,1],S^1.$$
\end{remark}

\begin{definition}\label{def:coreg}
Let $(X,\Delta)$ be a log canonical pair. If it is a Calabi--Yau pair, we define its {\it coregularity}  as the integer
$${\rm coreg}(X,\Delta):= \dim X - \dim \mathcal{D}(X,\Delta)-1,$$
using the convention that an empty complex has dimension $-1$.

If it is a Fano pair, we define its {\em coregularity} as the integer 
$${\rm coreg}(X,\Delta):= \min \{{\rm coreg}(X,\Gamma)\mid \Gamma\ge\Delta\mbox{ and }(X,\Gamma)\mbox{ lc Calabi--Yau pair}\}.$$
\end{definition}

We now recall the definition and basic results on complement pairs.

\begin{definition}
Let $(X,\Delta)$ be an lc pair. Let $N\ge 1$ be an integer. We say that a $\qq$-divisor $\Gamma\ge \Delta$ is an {\it $N$-complement} of $(X,\Delta)$ if:
\begin{itemize}
    \item $(X,\Gamma)$ is an lc pair,
    \item $N\Gamma$ is an integral Weil divisor, and $N(K_X+\Gamma)\sim 0$. 
\end{itemize}
\end{definition}

The following result is an application of the cyclic covering trick. It is often refered to as the construction of the {\it index one cover}.

\begin{lemma}\label{lem:index-one-cover}
    Let $(X,\Gamma)$ be an $N$-complement for some $N\ge 1$ chosen to be minimal for this property. Then there exists a unique compatible finite cyclic cover $p:(Y,\Gamma_Y)\to (X,\Gamma)$ of degree $N$ such that $(Y,\Gamma_Y)$ is a $1$-complement.
\end{lemma}

\begin{proof}
    Consider the Zariski open subset 
    $U := X_{\rm reg}$ in $X$. By \cite[Proposition 4.1.6]{Laz04a}, there exists a finite cyclic cover $p:V\to U$ of degree $N$ such that the Cartier divisor $p^*(K_X+\Gamma)|_U$ is trivial.
    By \cite[Theorem 3.5]{DG94}, it extends to a unique finite cyclic cover $p:Y\to X$ for a normal projective variety $Y$. Let $\Gamma_Y$ be the divisor $p^*\Gamma- {\rm Ram}(p)$, which is effective with coefficients in $\{0,1\}$ and supported entirely above $\lfloor\Gamma\rfloor$. 
    Then $p$ is clearly compatible with the pairs $(Y,\Gamma_Y)$ and $(X,\Gamma)$. The pair $(Y,\Gamma_Y)$ is lc by \cite[Proposition 5.20]{KM98} and satisfies $K_Y+\Gamma_Y\sim p^*(K_X+\Gamma)\sim 0$ as wished.
\end{proof}

Certain Calabi--Yau pairs are automatically complements. 
The following lemma follows from~\cite[Corollary 3]{FMM22} and~\cite[Theorem 3]{FFMP22}.

\begin{lemma}\label{lem:coreg-vs-complementCY}
Let $(X,\Delta)$ be an lc Calabi--Yau pair with standard coefficients.
\begin{itemize}
\item If ${\rm coreg}(X,\Delta)=0$, then
$(X,\Delta)$ is a $2$-complement.
\item If ${\rm coreg}(X,\Delta)=1$, then $(X,\Delta)$ is an $N$-complement for some $N\in \{1,2,3,4,6\}$.
\end{itemize}
\end{lemma}

Certain Fano pairs automatically have reasonable complements too. The following lemma follows from~\cite[Theorems 4 and 5]{FFMP22}.

\begin{lemma}\label{lem:coreg-vs-complementFano}
Let $(X,\Delta)$ be an lc Fano pair with standard coefficients.
\begin{itemize}
\item If ${\rm coreg}(X,\Delta)=0$, then
$(X,\Delta)$ admits a $2$-complement of coregularity zero.
\item If ${\rm coreg}(X,\Delta)=1$, then $(X,\Delta)$ admits an $N$-complement of coregularity one, for some $N\in \{1,2,3,4,6\}$.
\end{itemize}
\end{lemma}

The following lemma allows to extend complements of a curve pair obtained by adjunction from a surface pair. It is a consequence of Kawamata-Viehweg vanishing, see~\cite[Proposition 6.7]{Bir19}.

\begin{lemma}\label{lem:lifting-complements} 
Let $(X,\Delta)$ be a plt Fano surface pair with standard coefficients.
Let $S$ be a component of $\Delta$ with coefficient one, and let $(S,\Delta_S)$ be the pair obtained by adjunction of $(X,\Delta)$ to $S$. 
If $(S,\Delta_S)$ admits an $N$-complement $\Gamma_S$, 
then $(X,\Delta)$ admits an $N$-complement $\Gamma$ such that $(S,\Gamma_S)$ is also obtained by adjunction of $(X,\Gamma)$ to $S$.
\end{lemma}

We finally recall a characterization 
of toric varieties conjectured and proven for surfaces by Shokurov in \cite[Theorem 6.4]{Sho00} (it is also known to hold more generally by the work of Brown, McKernan, Svaldi, Zong \cite[Theorem 1.2]{BMSZ18}).

\begin{definition}
\label{def:complexity}
Let $(X,\Delta)$ be an lc surface pair such that $-K_X-\Delta$ is nef. Assume that $X$ is $\qq$-factorial.
The {\em (absolute) complexity} of $(X,\Delta)$ is defined as
\[
c(X,\Delta):=\rho(X)+\dim X - |\Delta|,
\]
where $|\Delta|$ denotes the sum of the coefficients of $\Delta$. 
\end{definition}

\begin{proposition}\label{prop:toric-characterization}
Let $(X,\Delta)$ be an lc surface pair such that $-K_X-\Delta$ is nef.
Then $c(X,\Delta)\geq 0$, and if moreover $c(X,\Delta)<1$, then $X$ is a toric surface and $\lfloor\Delta \rfloor
$ is contained in its toric boundary divisor. 
\end{proposition}

\subsection{Pot-pourri of useful lemmas}

\begin{lemma}\label{lem:Gor-MMP}
Let $X$ be a klt surface that admits a $1$-complement
$\Gamma$.
Let $X\rightarrow X_0$ be a step of the $K_X$-MMP contracting an irreducible curve $C$ to a point $x_0$.
Then either $C$ is a component of $\Gamma$, or $\Gamma$ and $C$ have exactly one intersection point $p$ and at most two components of $\Gamma$ contain $p$.

Furthermore, if the variety $X$ is smooth at a point $q\in C$, then at most one component of $\Gamma$ (besides possibly $C$ itself) contains $q$.
\end{lemma}

\begin{proof}
Assume that $C$ is not a component of $\Gamma$. Since $-K_X\cdot C > 0$, the curve $C$ intersects $\Gamma$ in at least one point.
Let $\Gamma_0$ be the push-forward of $\Gamma$ in $X_0$.
The pair $(X_0,\Gamma_0)$ is log canonical.

Since $X_0$ is a klt surface, it is $\mathbb{Q}$-factorial, and its local class group at any point is thus finite, see \cite[Remark 3.29]{MS21}.
By~\cite[Theorem 1.2]{MS21}, at most two components of $\Gamma_0$ contain the point $x_0$, i.e., either one or two components of $\Gamma$ intersect $C$.

By~\cite[Theorem 1.1]{FS20}, and since the contraction $X\to X_0$ is birational, the intersection of $C$ with the non-klt locus of $(X,\Gamma)$ is connected. It consists of finitely many points, hence of exactly one point $p$. In particular, $\Gamma\cap C=\{p\}$, as wished.

Now, assume that the variety $X$ is smooth at $q\in C$
and that there are two components $\Gamma_1$ and $\Gamma_2$ of $\Gamma$ that are different from $C$ and contain $q$.
Let $\Gamma_{1,0}$ and $\Gamma_{2,0}$ be the images of $\Gamma_1$ and $\Gamma_2$ in $X_0$.
By \cite[Theorem 1 and 2]{MS21}, the pairs $(X,\Gamma_1+\Gamma_2;q)$ and $(X_0,\Gamma_{1,0}+\Gamma_{2,0};x_0)$ are formally toric surface singularities.
Let $Y_0\rightarrow X_0$ be a toric resolution of singularities.
Write $(Y_0,\Gamma_{Y_0,0}+\Gamma_{Y_0,1}+E_{0,1}+\dots+E_{0,r})$ for the log pullback of $(X_0,\Gamma_{1,0}+\Gamma_{2,0})$ to $Y_0$.
By performing further toric blow-ups on $Y_0$, we may assume that the center of $C$ on $Y_0$ is not a toric strata.
Thus, we may extract $C$ by blowing-up $E_{0,j}$ at a smooth point not contained in another $E_{0,i}$. We write $Y\rightarrow Y_0$ for the smooth blow-up extracting $C$. 
We write $(Y,\Gamma_{Y,0}+\Gamma_{Y,1}+E_1+\dots+E_r)$ for the log pullback of $(X,\Gamma_1+\Gamma_2)$ to $Y$.
By construction, we have $E_j^2=E_{j,0}^2-1$ and $E_i^2=E_{i,0}^2$ for every $i\neq j$.
As $Y$ and $X$ are smooth
and $Y\rightarrow X$ is a projective birational morphism, the variety $X$ can be obtained from $Y$ by performing a sequence of smooth blow-downs, i.e., 
contracting a sequence of $(-1)$-curves.
By blowing down $(-1)$-curves simultaneously on $Y$ and $Y_0$ over $X$ and $X_0$ respectively, we arise to a model on which the strict transform of $E_{j,0}$ has self-intersection zero. This leads to a contradiction.
\end{proof}

The following lemma is an important consequence of Lemma \ref{lem:nori}. Under particular assumptions, it describes a {\em large} subgroup of the orbifold fundamental group as a quotient of a local orbifold fundamental group. We will use it to prove the residual finiteness of certain orbifold fundamental groups.

\begin{lemma}\label{lem:Nori-reg-fun}
Let $f:(X,\Delta)\to (C,\Delta_C)$ be an equimultiple fibration.
Let $\Delta_{\rm hor}$ be the horizontal part of $\Delta$ with respect to the fibration $f$.
Let $x\in X$ a point. Assume that for any small enough analytic open ball $V$ containing $f(x)\in C$, the point $x$ is contained in every connected component of $\Delta_{\rm hor}\cap f^{-1}(V)$, except for possibly one.
Moreover, denote by $F_x$ the fiber of $f$ containing $x$. Assume that  $f_{\bullet}(\gamma_{F_x})$ generates a subgroup of finite index $N$ in $\pi_1(C,\Delta_C)$.
Then, the image of the group homomorphism (induced by inclusion)
\[
\pi_1^{\rm reg}(X,\Delta;x)\rightarrow
\pi_1^{\rm orb}(X,\Delta)
\]
is a subgroup of finite index at most $N$ in $\pi_1^{\rm orb}(X,\Delta)$.
\end{lemma}

\begin{proof}
Let $U$ be a small enough connected simply-connected analytic neighborhood of $x\in X$ such that $\pi_1^{\rm reg}(X,\Delta;x)=\pi_1^{\rm orb}(U,\Delta|_U)$.
Let $p\in U$ be a general point
and $F_p$ be the fiber of $f$ containing $p$.
Let $\Delta|_{F_p}$ be the restriction of $\Delta$ to $F_p$.
By Lemma~\ref{lem:nori}, we have the following exact sequence:
\[
\xymatrix{
\pi_1^{\rm orb}(F_p,\Delta|_{F_p},p)\ar[r]^-{p_1}&
\pi_1^{\rm orb}\left(X,\Delta,p\right)\ar[r]^-{f_{\bullet}}&
\pi_1^{\rm orb}\left( 
C,\Delta_C, f(p) \right)\ar[r] & 1.
}
\]
where the points $p$ and $f(p)$ are precised as chosen base points for these fundamental groups.
Let $U_p$ be the analytic subset $F_p\cap U$ in $F_p$. 
Considering the inclusions 
$U_p \hookrightarrow U \hookrightarrow X$ and $U_p \hookrightarrow F_p \hookrightarrow X$, we get 
a commutative diagram:
\[
\xymatrix{ 
\pi_1^{\rm orb}(U_p,\Delta|_{U_p},p)\ar[r]\ar[d]^-{\pi_1} & \pi_1^{\rm orb}(U,\Delta|_{U},p)\ar[d]^-{\pi_2} & & \\ 
\pi_1^{\rm orb}(F_p,\Delta|_{F_p},p)\ar[r]^-{p_1}&
\pi_1^{\rm orb}\left(X,\Delta,p\right)\ar[r]^-{f_{\bullet}}&
\pi_1^{\rm orb}\left( 
C,\Delta_C,f(p) 
\right)\ar[r] & 1 .
}
\]
We claim that $\pi_1$ is surjective: Indeed, the intersection $F_p\cap \Delta_{\rm hor}$ contains finitely many, say $k$, distinct points, each corresponding to a distinct local branch of $\Delta_{\rm hor}$. Since $F_p$ is a smooth rational curve, the loops $\gamma_1,\ldots,\gamma_{k-1},\gamma_k$ around the points of $F_p\cap \Delta_{\rm hor}$ satisfy $$\langle \gamma_1,\ldots,\gamma_k\mid \gamma_1\cdots\gamma_k = 1 \rangle\twoheadrightarrow\pi_1(F_p,\Delta|_{F_p}).$$ 
Since the point $x$ is contained in all local branches of $\Delta_{\rm hor}$ except possibly one, the open subset $U_p$ in $F_p$ contains all of the $k$ points $F_p\cap \Delta_{\rm hor}$ except possibly one. Hence, all the loops $\gamma_1,\ldots,\gamma_k$, except possibly one, are in the image of $\pi_1$. But any $k-1$ of these loops already generate the whole group $\pi_1(F_p,\Delta|_{F_p})$. So the group homomorphism $\pi_1$ is surjective.

Since $\pi_1$ is surjective, the image of $\pi_2$ contains the image of $p_1$, that is the kernel of $f_{\bullet}$. Moreover, the loop $\gamma_{F_x}$ appears in the image of $\pi_2$, thus and by our finite index assumption, the subgroup $f_{\bullet}{\rm Im}(\pi_2)$ has finite index at most $N$ in $\pi_1^{\rm orb}(C,\Delta_C)$.
The set of left ${\rm Im}(\pi_2)$-cosets in $\pi_1^{\rm orb}(X,\Delta)$ is thus in bijection by $f_{\bullet}$ with the set of left $f_{\bullet}{\rm Im}(\pi_2)$-cosets in $\pi_1^{\rm orb}(C,\Delta_C)$, which is finite of cardinal at most $N$.
\end{proof}

\section{Proof of Theorem \ref{introthm:fun-group-klt}}\label{sec:klt-cy-stand}

\begin{proof}[Proof of Theorem \ref{introthm:fun-group-klt}]
Let $(X,\Delta)$ be a klt Calabi--Yau surface pair with standard coefficients. By~\cite[Lemma~2.7]{CGG23}, there exists a Zariski open set $U\subset X$ such that $\codim_{X\setminus U}(X)\geq 3$ and $(U,\Delta|_{U})$ is an orbifold. Since $X$ is a surface, we have $U=X$ and thus~\cite[Theorem~4.2]{CC14} provides the following exact sequence
 \begin{center}
     $1\rightarrow A_0\rightarrow \pi_1^{\rm orb}(X,\Delta)\rightarrow G_0\rightarrow 1$,
 \end{center}
 where $A_0$ is an abelian group of rank at most $4$ and $G_0$ is a finite group.

By Proposition~\ref{prop:galois2}, there is a finite orbifold cover $\pi:(X_0,\Delta_0)\rightarrow (X, \Delta)$ with Galois group $G_0$ and with an equality of subgroups $\pi_{\bullet}\pi_1^{\rm orb}(X_0,\Delta_0)=A_0$. 
By the Riemann--Hurwitz formula and \cite[Proposition 5.20]{KM98}, the pair $(X_0,\Delta_0)$ is klt and Calabi--Yau.

By the non-vanishing result \cite[V.4.9 Corollary]{Nak04}, there is a positive integer $M$ such that $M(K_{X_0}+\Delta_0)\sim 0$, hence the index one cover of $(X_0,\Delta_0)$ given by Lemma \ref{lem:index-one-cover} is a klt Gorenstein surface $X_1$ with $K_{X_1}\sim 0$. In particular, it has canonical singularities. Take the minimal resolution of singularities $r:Y\rightarrow X_1$. Then $Y$ is a K3 surface or an abelian surface.

By uniqueness of the index one cover, the composition of finite covers $(X_1,0)\rightarrow (X,\Delta)$ remains a compatible Galois cover. We denote its Galois group by $G_1$. Then we have a short exact sequence
\begin{center}
    $1\rightarrow \mathbb{Z}/M\mathbb{Z}\rightarrow G_1\rightarrow G_0\rightarrow 1$.
\end{center}
By uniqueness of the minimal resolution, the action of $G_1$ on $X_1$ lifts to an action by automorphisms on $Y$.

If $Y$ is an abelian surface, then it contains no rational curves, thus $Y\simeq X_1$.
Let $T\triangleleft G_1$ be the normal subgroup of $G_1$ acting on $X_1$ by translations, and consider the abelian surface $S:=X_1/T$ with the naturally induced compatible Galois cover $(S,0)\rightarrow (X,\Delta)$ of Galois group isomorphic to $Q:=G_1/T$. The Galois action of $Q$ on $S$ fixes the origin of $S$, hence $Q$ is contained in the group of Lie group automorphisms of $S$. Now \cite[Lemma~3.2 and 3.3]{Fujiki88} imply that $|Q|\leq 96$. We reach the following short exact sequence
\begin{center}
    $1\rightarrow \mathbb{Z}^4 \rightarrow \pi_1^{\rm orb}(X,\Delta)\rightarrow Q\rightarrow 1$.
\end{center}

When $Y$ is a $K3$ surface, \cite[Theorem~1]{Kondo99} implies that $|G_0|\le |G_1|\leq 3840$. In that case, note that $\pi_1^{\rm orb}(X,\Delta)$ is in fact isomorphic to $G_1$.
\end{proof}

\section{Proof of Theorem \ref{introthm:Jordan-dlt-Fano} and more}\label{sec:Jordan}

We adopt the following definitions: A normal variety $X$ is called {\it rationally connected} if for any two general points $x,y$ in $X$, there is a rational curve on $X$ passing through $x$ and $y$. This property is invariant by birational equivalence, and equivalent to rationality in dimension $1$ and $2$.
A normal variety $X$ is called {\it rationally chain connected} if for any two general points $x,y$ in $X$, there is a finite chain of rational curves $C_1,\ldots, C_n$ on $X$ such that $x\in C_1,y\in C_n,$ and $C_i\cap C_{i+1}\neq\emptyset$ for $1\le i\le n-1$. Note that the notions of rational connectedness, and of rational chain connectedness coincide for complex projective varieties that are smooth, or that have klt singularities as well (see, e.g.,~\cite[Theorem IV.3.10]{Kol96}).

By~\cite[Theorem 1.8]{PS14}, the following constants are well-defined.
\begin{definition}
    Fof $d\in\nn$, we define the {\it rationally connected Jordan constant} $J(d)$ as the minimal constant $J$ such that, for any rationally connected variety $X$ of dimension $d$ and for any finite group $Q\subset {\rm Aut}(X)$, there exists a normal abelian subgroup of index at most $J$ in $Q$.
\end{definition}

\begin{remark}\label{rem:jordan-values}
    It is well-known that $J(1)=60$, as it is the Jordan constant of the projective linear group ${\rm PGL}_2(\cc)$. The value $J(2)=7200$ was computed by Yasinsky in \cite[Theorem 1.9]{Yas19}. The next values are unknown.
\end{remark}

The next {\it ad hoc} definition is by no means standard, but should simplify the next few statements.

\begin{definition}
    A log canonical pair $(X,\Delta)$ is said to be {\it hereditarily rationally connected} if, for any compatible finite Galois cover $p:(X',\Delta')\to (X,\Delta)$, the normal projective variety $X'$ is rationally connected. 
\end{definition}

The next lemma is profusely used in our proofs.

\begin{lemma}\label{lem:res-finite}
Let $(X,\Delta)$ be a log canonical pair of dimension $d$. Suppose that
\begin{enumerate}
    \item $\pi_1^{\rm orb}(X,\Delta)$ is residually finite;
    \item $(X,\Delta)$ is hereditarily rationally connected. 
\end{enumerate}
Then the group $\pi_1^{\rm orb}(X,\Delta)$ admits a normal abelian subgroup of index at most $J(d)$.
\end{lemma}

\begin{proof}
Denote by $G$ the group $\pi_1^{\rm orb}(X,\Delta)$. We have $(H_n)_{n\in \nn}$ a decreasing sequence of normal subgroups of finite index in $G$ such that $\bigcap_{n\in \nn} H_n =\{1\}$. By Proposition \ref{prop:galois2} (Galois correspondence), we have corresponding compatible finite Galois covers $p_n:(Y_n,\Delta_n)\to (X,\Delta)$ such that ${\rm Gal}(p_n)\simeq G/H_n$ acts on $Y_n$.
Since $Y_n$ is rationally connected and by definition of $J(d)$, there exists a normal abelian subgroup
$A_n \triangleleft G/H_n$ that has index at most $J(d)$ in $G/H_n$
Let $G_n$ be the preimage of $A_n$ in $G$: It is a normal subgroup of $G$ of index at most $J(d)$.
Since $G$ is a finitely presented group, there are only finitely many subgroups of $G$ of index at most $J(d)$.
Hence, there exists $G_N\triangleleft G$ a normal subgroup of index at most $J(d)$
that agrees with $G_n$ for all $n\ge N$.

We argue that $G_N$ is an abelian group. 
Let $x,y\in G_N$, and write $z:=xyx^{-1}y^{-1}$. Note that for each $n\ge N$, we have an exact sequence
\begin{equation}\label{eq:amses}
1\rightarrow H_n \rightarrow G_N \rightarrow A_n \rightarrow 1,     
\end{equation}
where we recall that $A_n$ is a finite abelian group.
The image of $z$ in $A_n$ is trivial, so $z\in H_n$ for every $n\ge N$.
But they have trivial intersection, so $x$ and $y$ commute.
\end{proof}

\begin{corollary}\label{cor:res-finite-for-surf}
Let $(X,\Delta)$ be a log canonical pair of dimension 2. Suppose that
\begin{enumerate}
\item $\pi_1^{\rm orb}(X,\Delta)$ is residually finite, 
\item $(X,\Delta)$ is hereditarily rationally connected.
\end{enumerate} 
Then the group $\pi_1^{\rm orb}(X,\Delta)$ admits a normal abelian subgroup of rank at most $3$, of torsion-free rank at most $2$, and of index at most $7200$.
\end{corollary}

\begin{proof}
We follow the proof of Lemma~\ref{lem:res-finite} for $d=2$ until the end. Using also Remark \ref{rem:jordan-values}, it yields a normal abelian subgroup $G_N\triangleleft G$ of index at most $7200$, which satisfies exact sequences as in Equation \eqref{eq:amses}, for each $n\ge N$, that is
$$1\to H_n\to G_N\to A_n\to 1,$$
where we recall that $A_n$ is a finite abelian subgroup of the quotient $G/H_n$ acting on a rationally connected surface $Y_n$. As a finite abelian subgroup of the plane Cremona group ${\rm Bir}(\pp^2_{\cc})$, the group $A_n$ has one of the isomorphism types classified by Blanc in \cite[Theorem 6]{Bla07}. In particular, either it has rank at most $2$, or it is isomorphic to one of the following groups of rank 3 and 4:
$\zz/2m\zz\times (\zz/2\zz)^2$ for $m\ge 1$, $(\zz/4\zz)^2\times \zz/2\zz,$ $(\zz/3\zz)^3,$ or $(\zz/2\zz)^4$. 

Let us first assume that there are infinitely many values of $n$ such that the group $A_n$ has rank at most $3$.
Then, by Lemma \ref{lem:ab-group}, the group $G_N$ has rank at most $3$. We now prove that the torsion-free rank of $G_N$ is at most $2$ by contradiction: Assume that $\zz^3$ is a subgroup of $G_N$. Since $G_N$ is infinite and the decreasing sequence $(H_n)$ has trivial intersection, the sequence of indices of the $H_n$ in $G$ must grow unbounded. Thus, the sequence of the orders of the groups $A_n$ grow unbounded too. Hence, up to extracting an appropriate subsequence, each $A_n$ is of rank at most 2, or each $A_n$ is isomorphic to $\zz/2f(n)\zz\times (\zz/2\zz)^2$ for some increasing function $f:\nn\to\nn$. In the first case, Lemma \ref{lem:ab-group} shows that $G_N$ has rank at most $2$, hence torsion-free rank at most $2$. In the second case, Lemma \ref{lem:ab-group} shows that $G_N$ has rank $3$, hence $G_N\cong \zz^3$. In particular, the kernel of the quotient map $G_n\to (\zz/2\zz)^2$ induced by all the $A_n$ is a non-trivial subgroup of $G_N$ that is contained in all the $H_n$, a contradiction. 

\medskip

Let us now assume the converse, i.e., that for every $n$ large enough, the group $A_n$ is isomorphic to $(\zz/2\zz)^4$. Fix a large enough $n$. Then, by the classification of finite abelian subgroups of the plane Cremona group \cite{Bla07}, we can embed the group $G/H_n$, which contains $A_n\simeq (\zz/2\zz)^4$, into the automorphism group of either $\pp^1\times \pp^1$, or a smooth del Pezzo surface of degree $4$. 

In the second case, it turns out that $G/H_n$ is of order at most $160$ by \cite[Theorem 1.1]{Hos96}, hence since $\bigcap H_n =\{1\}$, the group $G$ contains the trivial subgroup, that is abelian of rank $0$, with index at most $160$.

In the first case, we can follow the proof of \cite[Propositions 3.2 and 3.3 (Case $d=8$)]{Yas19}, use that $(\zz/2\zz)^4$ is a normal subgroup in $G/H_n$ and that $(\zz/2\zz)^2$ cannot be normally embedded into $\mathfrak{A}_5$ to show that either $G/H_n$ is of order at most $1440= 24\cdot 60$, or $G/H_n$ contains a normal abelian subgroup $B_n$ of index at most $120$. If $G/H_n$ has order at most $1440$ for infinitely many $n$, then $G$ contains the trivial subgroup, that is abelian of rank $0$, with index at most $1440$. Otherwise, for every $n$ large enough, $G/H_n$ contains a normal abelian subgroup $B_n$ of index at most $120$. Either $B_n$ has rank at most $3$ for infinitely many $n$, in which case we just reset $A_n := B_n$ and conclude as above when we assumed that $A_n$ had rank at most $3$, or $B_n$ is isomorphic to $(\zz/2\zz)^4$ for every $n$ large enough, and in that case $G$ contains the trivial subgroup, that is abelian of rank $0$, with index at most $120\cdot 2^4 = 1920$.
\end{proof}

We now prove a simple lemma.

\begin{lemma}\label{lem:fano-herratcon}
    Let $(X,\Delta)$ be a dlt Fano pair of dimension $n$. Suppose that $\pi_1^{\rm orb}(X,\Delta)$ is residually finite.
    Then, $(X,\Delta)$ is hereditarily rationally connected.
\end{lemma}

\begin{proof}
Using the fact that dlt pairs have dlt small $\qq$-factorializations \cite[Corollary 1.37]{Kol13}, and that these do not alter orbifold fundamental groups by Lemma \ref{lem:Qfact-mod-vs-fun}, we may assume that $X$ itself is $\qq$-factorial.

Let $p:(Y,\Delta_Y)\to (X,\Delta)$ be a compatible finite Galois cover. Let $m$ denote the degree of $p$. Then, for any small enough positive $\varepsilon$, the pair $(X,\Delta - \varepsilon\lfloor \Delta\rfloor)$ is a klt Fano pair, and since $m\varepsilon < 1$, the divisor
$$B_Y := p^*(\Delta - \varepsilon\lfloor \Delta\rfloor) - {\rm Ram}(p)$$
is effective. By \cite[Proposition 5.20]{KM98}, the pair $(Y,B_{Y})$ is thus klt; It also is a Fano pair.
By \cite{Zha06}, \cite{HM07}, $Y$ is rationally connected, thus $(X,\Delta)$ is indeed hereditarily rationally connected.
\end{proof}

We conclude this section with the proof of Theorem~\ref{introthm:Jordan-dlt-Fano}.

\begin{proof}[Proof of Theorem~\ref{introthm:Jordan-dlt-Fano}]
It follows from Lemmas \ref{lem:fano-herratcon} and \ref{lem:res-finite}.
\end{proof}

\section{Classifying Fano and Calabi--Yau curve pairs}\label{section:trichotomy}

In this section, we classify pairs $(C,\Delta)$, where $C$ is a smooth curve, $\Gamma$ has standard coefficients, and the pair $(C,\Delta)$ is Fano or Calabi--Yau.

\begin{definition}\label{def:trichotomy}
    Let $(C,\Delta)$ be a Fano or Calabi--Yau curve pair. We say that $(C,\Delta)$ is of:
    \begin{enumerate}
    \item {\it elliptic type} if $(C,\Delta^{\rm st})$ is a klt Calabi--Yau pair;
    \item {\it toric type} if there is a non-klt $2$-complement $(C,\Gamma)$ of $(C,\Delta^{\rm st})$;
    \item {\it sporadic type} if $(C,\Delta^{\rm st})$ is of the form $$\left(\pp^1,\frac{1}{2}\{0\}+\frac{2}{3}\{1\}+\frac{n-1}{n}\{\infty\}\right),$$
    for some $n\in\{3,4,5\}$.
    \end{enumerate}
\end{definition}

\begin{lemma}
    Let $(C,\Delta)$ be a Fano or Calabi--Yau curve pair. Then $(C,\Delta)$ is of exactly one of the three types listed in Definition \ref{def:trichotomy}.
\end{lemma}

\begin{proof}
    These three types are clearly mutually exclusive. Let $(C,\Delta)$ be a Fano or Calabi--Yau curve pair. Assume that it is neither of elliptic, nor of toric type. In particular, $C$ is then isomorphic to $\pp^1$. 

    Note that any pair $(\pp^1,D)$ with standard coefficients such that $D$ is supported on two points or less has a non-klt $1$-complement, thus is of toric type. So $\Delta^{\rm st}$ is supported on at least three points.

    If $\Delta^{\rm st}$ is supported on four or more points, then, as the smallest non-zero standard number is $\frac{1}{2}$, we have 
    $$(C,\Delta^{\rm st})\cong\left(\pp^1,\frac{1}{2}\{0\}+\frac{1}{2}\{1\}+\frac{1}{2}\{t\}+\frac{1}{2}\{\infty\}\right),$$
    for some $t\in\cc^*$. This pair is of elliptic type, a contradiction.

    So $\Delta^{\rm st}$ is supported on exactly three points, say $0,1,\infty\in\pp^1$ (in increasing coefficients order). 
    
    If both $0$ and $1$ have the smallest non-zero standard coefficient $\frac{1}{2}$, the pair $(C,\Delta^{\rm st})$ has a non-klt $2$-complement $(C,\Gamma)$ with
    $$\Gamma := \frac{1}{2}\{0\}+\frac{1}{2}\{1\}+\{\infty\},$$ hence is of toric type, a contradiction.

    If $0$ has coefficient $\frac{1}{2}$ and $1$ has a larger coefficient $\frac{k-1}{k}$ with $k\ge 3$, then $\infty$ must have coefficient $\frac{n-1}{n}$ for some $k\le n\le 6$. Since $(C,\Delta^{\rm st})$ is klt and not of elliptic type, it must be a Fano pair. Hence, $k=3$ and $3\le n\le 5$, i.e. this pair is of sporadic type, which concludes the proof.
\end{proof}

We now describe the orbifold fundamental groups of these three types of curve pairs.

\begin{proposition}\label{prop:curvepair}
    Let $(C,\Delta)$ be a Fano or Calabi--Yau curve pair. Then the orbifold fundamental group $\pi_1^{\rm orb}(C,\Delta)$ admits a characteristic abelian subgroup of rank at most $2$ and of index at most $60$.
\end{proposition}

The proof of this proposition relies on three lemmas.

\begin{lemma}\label{lem:ellpair}
    Let $(C,\Delta)$ be a pair of elliptic type. Then $\pi_1^{\rm orb}(C,\Delta)$ admits an abelian characteristic subgroup of index at most $6$ that is isomorphic to $\mathbb{Z}^2$.
\end{lemma}

\begin{lemma}\label{lem:toricpair}
    Let $(C,\Delta)$ be a pair of toric type. Then $\pi_1^{\rm orb}(C,\Delta)$ is isomorphic to $(\zz/2\zz)^2$, or it admits a cyclic characteristic subgroup of index at most $2$. 
    Moreover, the group is infinite if and only if the pair is Calabi--Yau.
\end{lemma}

\begin{lemma}\label{lem:sporpair}
    Let $(C,\Delta)$ be a pair of sporadic type. Then $\pi_1^{\rm orb}(C,\Delta)$ is one of the finite groups $\mathfrak{A}_4,\mathfrak{S}_4$, $\mathfrak{A}_5$.
\end{lemma}

We first prove these three lemmas.

\begin{proof}[Proof of Lemma \ref{lem:ellpair}]
By definition of the elliptic type, the pair $(C,\Delta)$ is klt, Calabi-Yau, and has standard coefficients. Listing the possibilities by hand, it is easy to see that $(C,\Delta)$ is an $N$-complement for some $N\in\{2,3,4,6\}$. In particular, by Lemma \ref{lem:index-one-cover}, it admits a compatible finite cyclic (index one) cover $p:E\to (C,\Delta)$, where $E$ is an elliptic curve and the degree of $p$ equals $N$.
By Proposition \ref{prop:galois1}, this shows that $\pi_1^{\rm orb}(C,\Delta)$ contains a normal subgroup isomorphic to $\pi_1(E)\simeq\zz^2$ of index $N$.

To show that this subgroup is characteristic, we prove that it is the only normal subgroup isomorphic to $\zz^2$ and of index $N$ in $\pi_1^{\rm orb}(C,\Delta)$. By Proposition \ref{prop:galois2}, it suffices to show that any compatible finite cyclic cover $q:(B,\Delta_B)\to (C,\Delta_C)$ of degree $N$ with $\pi_1^{\rm orb}(B,\Delta_B)\simeq\zz^2$ is exactly $p$. Fix such a cover $q$, and consider the klt Calabi--Yau pair $(B,\Delta_B)$ with standard coefficients. Note that $\zz^2$ has no element of finite order, thus $\Delta_B=0$ and $B$ is an elliptic curve. Thus, $q$ is the index one cover of $(C,\Delta_C)$. It therefore coincides with $p$.
\end{proof}

\begin{proof}[Proof of Lemma \ref{lem:toricpair}]
Let $(C,\Gamma)$ be a non-klt $2$-complement of $(C,\Delta^{\rm st})$. 
We first show that $\pi_1^{\rm orb}(C,\Gamma)$ is isomorphic to $\zz$ or to ${\rm Dih}:=\zz\rtimes\zz/2\zz$. 
Indeed, $(C,\Gamma)$ is either 
\begin{itemize}
    \item the pair $(\pp^1,\{0\}+\{\infty\})$, with orbifold fundamental group
$$\pi_1^{\rm orb}(\pp^1,\{0\}+\{\infty\}) = \pi_1(\cc^*)\simeq \mathbb{Z},$$
\item or the pair $(\mathbb{P}^1,\frac{1}{2}\{0\}+\frac{1}{2}\{1\}+\{\infty\})$, with orbifold fundamental group
$$\pi_1\left(\mathbb{P}^1,\frac{1}{2}\{0\}+\frac{1}{2}\{1\}+\{\infty\}\right)=\langle \gamma_0,\gamma_1\mid (\gamma_0)^2=(\gamma_1)^2=1\rangle\simeq {\rm Dih}.$$
\end{itemize}

Hence, by Lemma \ref{lem:fun-group-increasing-coeff}, the orbifold fundamental group $\pi_1^{\rm orb}(C,\Delta)$ is a quotient of $\zz$ or of $\zz\rtimes \zz/2\zz$. In the first case, it is cyclic, as wished. In the second case, it is isomorphic to ${\rm Dih}$, to a finite dihedral group ${\rm Dih}_m:=\zz/m\zz\rtimes \zz/2\zz$ for some $m\ge 3$, or to $(\zz/2\zz)^2$. The result is clear from here.
\end{proof}

\begin{proof}[Proof of Lemma \ref{lem:sporpair}]
We consider the unique faithful projective linear representations of
\begin{itemize}
    \item $\mathfrak{A}_4$ in ${\rm PGL}_2(\cc)$, inducing a quotient map $p:\pp^1\to \pp^1$. Its branching divisor is
    $${\rm Branch}(p)\cong\frac{1}{2}\{0\}+\frac{2}{3}\{1\}+\frac{2}{3}\{\infty\}.$$
    \item $\mathfrak{S}_4$ in ${\rm PGL}_2(\cc)$, with a quotient map $p:\pp^1\to \pp^1$. It has branching divisor
    $${\rm Branch}(p)\cong\frac{1}{2}\{0\}+\frac{2}{3}\{1\}+\frac{3}{4}\{\infty\}.$$
    \item $\mathfrak{A}_5$ in ${\rm PGL}_2(\cc)$, yielding a quotient map $p:\pp^1\to \pp^1$ with branching divisor
    $${\rm Branch}(p)\cong\frac{1}{2}\{0\}+\frac{2}{3}\{1\}+\frac{4}{5}\{\infty\}.$$
\end{itemize}
 (Here, we use the symbol $\cong$ to denote equality up to pushforward by an automorphism of $\pp^1$.) In each case, $p$ is compatible with the pairs $(\pp^1,0)$ and $(\pp^1,{\rm Branch}(p))$, and we use Proposition \ref{prop:galois1} (Galois correspondence) to conclude.
\end{proof}

\section{Log Calabi--Yau surfaces admitting Mori fiber spaces to curves}
\label{sec:mfs}

In this section, we study orbifold fundamental groups of lc Calabi--Yau surface pairs with the structure of a Mori fiber space onto a curve. Let us first fix some conventions and definitions.

\begin{definition}
    Let $X$ be a normal projective surface. We say that it admits a {\it Mori fiber space onto a curve} if the following requirements hold
    \begin{itemize}
        \item $X$ is $\qq$-factorial with lc singularities;
        \item there is a fibration $f:X\to C$, where $C$ is a smooth curve;
        \item the general fibre of $f$ is isomorphic to $\pp^1$, and $\rho(X)=2$.
    \end{itemize}
\end{definition}

\begin{definition}\label{def:heisenberg}
For $k\in\mathbb{Z}_{\ge 0}$, we define the $k$-th Heisenberg group as
    \[
    H_k:=\langle a,b,c\mid [a,b]=[a,c]=1,[b,c]=a^k\rangle.
    \]
    It is a nilpotent group of length 2 and rank 3, and virtually abelian if and only if $k=0$.
\end{definition}

\begin{remark}
    The set of $k$-th Heisenberg groups for $k\in\mathbb{Z}_{\ge 0}$ coincides with the set of groups $H$ realized as a central extension of the form
    $$1\to \zz\to H\to \zz^2.$$
\end{remark}

The main result of this section is the following.

\begin{proposition}\label{prop:mfs}
Let $(X,\Delta)$ be a log canonical Calabi--Yau pair where $X$ is a surface and a Mori fiber space onto a curve. Then the group $\pi_1^{\rm orb}(X,\Delta)$ admits a normal subgroup of index at most $7200$ that is a quotient of the nilpotent group $H_k$ for some $k\ge 1$, or an abelian group of rank $r\le 4$. Moreover, we can ensure that $r\le 3$, unless possibly when $(X,\Delta)$ is a klt Calabi--Yau pair with standard coefficients. Furthermore, if the group $\pi_1^{\rm orb}(X,\Delta)$ is not virtually abelian, then there is a compatible finite cover $(Y,\Delta_Y)$ of $(X,\Delta)$ of degree at most $96$ such that 
$$(Y,\Delta_Y)\simeq(\mathbb{P}(\mathcal{O}_E\oplus L),s_0+s_{\infty}),$$
where $E$ is an elliptic curve, $L$ is an ample line bundle on $E$, and $s_0,s_{\infty}$ are the two sections corresponding to the two factors.
\end{proposition}

In Example~\ref{ex:circle-over-elliptic}, it is explained how the group $H_k$ naturally appears in this set-up.

\medskip

Throughout this section, we denote by $(X,\Delta)$ a log Calabi-Yau surface and by $f:X\to C$ a Mori fiber space onto a smooth curve. We denote by $F$ the general fiber of $f$. By Lemma \ref{lem:mfscompatible} and Remark \ref{rem:mfscompatible}, there is an lc Fano or Calabi--Yau curve pair $(C,\Delta_C)$ that makes $f:(X,\Delta)\to(C,\Delta_C)$ a compatible map.
We make the following observation, which should be considered a corollary of Lemma~\ref{lem:res-finite-from-exact} and Proposition \ref{prop:curvepair}.

\begin{proposition}\label{prop:res-finite-mfs}
Let $(X,\Delta)$ be an lc Calabi-Yau pair where $X$ is a surface and a Mori fiber space onto a curve. Then the group $\pi_1^{\rm orb}(X,\Delta)$ is residually finite.
\end{proposition}

\begin{proof}
By Lemma~\ref{lem:nori}, 
we have an exact sequence
$$\pi_1^{\rm orb}(F,\Delta|_F) \rightarrow 
\pi_1^{\rm orb}(X,\Delta) \rightarrow 
\pi_1^{\rm orb}(C,\Delta_C)\rightarrow 1,$$
which only features finitely generated groups.
We want to conclude by Lemma~\ref{lem:res-finite-from-exact}. For that, it suffices to show that both groups $\pi_1^{\rm orb}(F,\Delta|_F)$ and $\pi_1^{\rm orb}(C,\Delta_C)$ are virtually abelian. Since the pairs $(F,\Delta|_F)$ and $(C,\Delta_C)$ are log canonical Fano or Calabi--Yau curve pairs, this just follows from Proposition~\ref{prop:curvepair}.
\end{proof}

We use the following birational trick.

\begin{lemma}\label{lem:sarkisovlink}
Let $(X,\Delta)$ be an lc Calabi-Yau pair where $X$ is a surface and a Mori fiber space onto a curve $C$. Then there exists a birational modification $(X',\Delta')$ of $(X,\Delta)$ relative to the base curve $C$ such that
\begin{itemize}
    \item All log canonical centers of $(X',\Delta')$ are contained in $\lfloor\Delta'\rfloor$;
    \item The fibration $f':X'\to C$ induced by $f:X\to C$ remains a Mori fiber space;
    \item The pair $(X',\Delta')$ remains an lc Calabi--Yau pair;
    \item We have a surjective homomorphism $\pi_1^{\rm orb}(X',\Delta')\twoheadrightarrow\pi_1^{\rm orb}(X,\Delta)$.
\end{itemize}
\end{lemma}

\begin{proof}
We consider the finite set $\{p_1,\ldots,p_k\}$ of points $p\in C$ such that $f^{-1}(p)\setminus f^{-1}(p)\cap \lfloor\Delta\rfloor$ contains a log canonical center of the pair $(X,\Delta)$.
We now perform the elementary transformations of the ruled surface $X$ centered at the points $p_1,\ldots,p_k$, i.e., apply a composition of $k$ Sarkisov links of type II, i.e., blow-up each point $p_i$ to obtain a surface pair $(Z,\Delta_Z)$, and then contract the strict transform $F'_i$ of each fiber $f^{-1}(p_i)$ to get a surface pair $(X',\Delta')$. We have a commutative diagram

$$\xymatrix{
& (Z,\Delta_Z)\ar[dl]_{\varepsilon} \ar[dr]^{\varepsilon'} &\\
(X,\Delta) \ar[d]^{f} & &(X',\Delta')\ar[d]^{f'} \\
(C,\Delta_{C}) & & (C,\Delta'_C)
}$$
where $f':X'\to C$ clearly remains a Mori fiber space and $(X',\Delta')$ remains an lc Calabi--Yau pair. By Lemmas \ref{lem:dlt-mod-vs-fun} and \ref{lem:fun-under-surj}, we have a surjective homomorphism
$$\pi_1^{\rm orb}(X',\Delta')\twoheadrightarrow \pi_1^{\rm orb}(Z,\Delta_Z)\simeq\pi_1^{\rm orb}(X,\Delta).$$
We note that all log canonical centers of $(Z,\Delta_Z)$ are contained in $\lfloor\Delta_Z\rfloor\cup \bigcup_{i} F'_i.$
Since $\lfloor\Delta_Z\rfloor$ contains the exceptional divisors $E_i$ above the points $p_i$, and each $E_i$ intersects the corresponding $F'_i$, the log canonical centers of $(X',\Delta')$ are contained in $\lfloor \Delta'\rfloor\cup \bigcup_{i} \varepsilon'(F'_i)=\lfloor\Delta'\rfloor$.
\end{proof}

\subsection{Hereditarily uniruled pairs}
The following definition is rather {\it ad hoc}, but proves convenient. We recall that a normal projective variety $X$ is called {\it uniruled} if there exist a normal projective variety $Z$ and a dominant rational map $r:Z\times\pp^1\dashrightarrow X$.

\begin{definition}
    A log canonical pair $(X,\Delta)$ is said to be {\it hereditarily uniruled} if for any compatible finite Galois cover $p:(X',\Delta')\to (X,\Delta)$, the normal projective variety $X'$ is uniruled.
\end{definition}

Let us now focus on the geometry of certain Mori fiber space pairs that are hereditarily uniruled.

The following lemma is a classical result. It summarizes some of the statements made throughout \cite[Part V.2]{Har77}.

\begin{lemma}\label{lem:p1bdle-antican-system}
   Let $\phi: Y\to E$ be a smooth $\pp^1$-bundle over an elliptic curve with general fiber $F$. Then, $(-K_Y)^2 = 0$ and
the Mori cone $\overline{NE}(Y)$ is 
\begin{itemize}
\item either spanned by $-K_Y$ and $F$;
\item  or spanned by $B$ and $F$, where $B$ is a section of $\phi$ with negative self-intersection.\\ Moreover in this case, any Calabi--Yau pair $(Y,\Delta_Y)$ satisfies $\Delta_Y\ge B$.
\end{itemize}
\end{lemma}

We are now well equipped to prove the following statement.

\begin{proposition}\label{prop:ellipticbase}
Let $(X,\Delta)$ be an lc Calabi-Yau surface pair whose log canonical centers are all contained in $\lfloor\Delta\rfloor$, which is hereditarily uniruled, and which admits an equimultiple Mori fiber space $f:(X,\Delta)\to (C,\Delta_C)$ where the curve pair $(C,\Delta_C)$ is of elliptic type.
Then the orbifold fundamental group $\pi_1^{\rm orb}(X,\Delta)$ admits a normal subgroup of index at most $3456$ that is abelian of rank at most $2$, or a normal subgroup of index at most $96$ that is a quotient of the nilpotent group $H_k$ for some $k\ge 0$. 
Moreover, if the group $\pi_1^{\rm orb}(X,\Delta)$ is not virtually abelian, then there is a compatible finite cover $(Y,\Delta_Y)$ of $(X,\Delta)$ of index at most 96 such that 
$$(Y,\Delta_Y)\simeq(\mathbb{P}(\mathcal{O}_E\oplus L),s_0+s_{\infty}),$$
where $E$ is an elliptic curve, $L$ is an ample line bundle on $E$, and $s_0,s_{\infty}$ are the two sections corresponding to the two factors.
Furthermore, if the group $\pi_1^{\rm orb}(X,\Delta)$ contains a copy of $\zz^3$, then there is a compatible finite cover $(Y',\Delta_{Y'})$ of $(X,\Delta)$ such that 
$$(Y',\Delta_{Y'})\simeq(\mathbb{P}(\mathcal{O}_E\oplus M),s_0+s_{\infty}),$$
where $E$ is an elliptic curve, $M$ is a line bundle of degree zero on $E$, and $s_0,s_{\infty}$ are two sections of this $\pp^1$-bundle.
\end{proposition}

\begin{proof}
 Note that, since $(C,\Delta_C)$ is of elliptic type, we have $\lfloor \Delta_C\rfloor = 0$ and $K_C+\Delta_C\equiv 0$. The vertical part of our divisor thus satifies $\lfloor \Delta_{\rm vert}\rfloor = 0$. Hence and by our assumption on log canonical centers, $(X,\Delta_{\rm vert})$ is a klt pair. 
 By Remark \ref{rem:mfscompatible}, the divisor $\Delta_{\rm vert}^{\rm st}$ is linearly equivalent to $L:= -f^*K_C-{\rm Ram}(f)$, which is an integral Weil divisor, thus Cartier on $X_{\rm reg}$. 
By Lemma \ref{lem:ellpair} and its proof, we have an integer $N\in\{1,2,3,4,6\}$ such that the pair $(C,\Delta_C)$ is an $N$-complement. Using Lemma \ref{lem:index-one-cover}, we build a compatible finite cyclic cover $q:E\to (C,\Delta_C)$ of degree $N$ where $E$ is a smooth elliptic curve. Let $Y$ denote the normalization of the surface obtained by fibre product as follows
$$\xymatrix{
Y \ar[r]^{p} \ar[d]_{\phi} & X\ar[d]^{f} \\
(E,0) \ar[r]_{q} & (C,\Delta_C).
}$$
Clearly, $p$ remains a finite cyclic cover of degree $N$ and $\phi$ remains a Mori fiber space. By the compatibility of $q$, we note that $\phi$ has no multiple fibers, thus by \cite[Theorem II.2.8]{Kol96}, $\phi: Y\to E$ is a smooth $\pp^1$-bundle. We then easily check that 
$${\rm Ram}(p) =\phi^*{\rm Ram}(q)=\phi^*q^*\Delta_C=p^*\Delta_{\rm vert}^{\rm st}.$$
Hence, the finite cover $p$ is compatible with the pairs $(Y,\Delta_Y)$ and $(X,\Delta)$ where $\Delta_Y$ denotes the effective divisor $p^*(\Delta - \Delta_{\rm vert}^{\rm st})$.
The facts that $p$ is compatible and that the higher homotopy group $\pi_2(E)$ is trivial allow to refine Lemma \ref{lem:nori}. We have a short exact sequence

\begin{equation}\label{eq:specialnori}
    1\to \pi_1^{\rm orb}(F,\Delta|_F)\to \pi_1^{\rm orb}(X,\Delta)\to \pi_1^{\rm orb}(C,\Delta_C)\to 1
\end{equation}

\noindent Writing down the canonical bundle formula for $\phi$, we note that the discriminant part $B_E$ is trivial. Thus, the divisor $\Delta_Y$ has no vertical components and the components of $\Delta_Y^{\rm st}$ are disjoint.

\bigskip

One case easily ruled out is the following: Assume that the induced pair $(F,\Delta|_F)$ on the general fiber of $f$ is of elliptic type. Then the coefficients of the divisor $\Delta_Y$, which has no vertical components, are contained in $\{\frac{1}{2},\frac{2}{3},\frac{3}{4},\frac{5}{6}\}$, i.e., $\Delta_Y$ has standard coefficients smaller than one. Since the log canonical centers of $(Y,\Delta_Y)$ are all contained in $\lfloor\Delta_Y\rfloor = 0$, this pair is thus klt. By the non-vanishing result \cite[V.4.9 Corollary]{Nak04}, there is a positive integer $M$ such that $M(K_Y+\Delta_Y)\sim 0$, hence the index one cover of $(Y,\Delta_Y)$ is a klt Gorenstein surface $Z$ with $K_Z\sim 0$. In particular, $Z$ is not uniruled, and that contradicts the hereditary uniruledness of $(X,\Delta)$.

\bigskip

We consider the next case: Assume that the pair $(F,\Delta|_F)$ is of sporadic type, i.e., 
$$(F,\Delta^{\rm st}|_F)\cong \left(\pp^1,\frac{1}{2}\{0\}+\frac{2}{3}\{1\}+\frac{n-1}{n}\{\infty\}\right)$$ 
with $n\in\{3,4,5\}$. One of the following situations then occurs
\begin{enumerate}[label=\alph*)]
    \item either $\Delta^{\rm st} = \frac{1}{2}s_0+\frac{2}{3}s_1+\frac{n-1}{n}s_{\infty}$ where $s_0, s_1, s_{\infty}$ are sections of $f$;
    \item or $\Delta^{\rm st} = \frac{1}{2}s + \frac{2}{3}b$ where $s_0$ is a section and $b$ a bisection of $f$.
\end{enumerate}

In Case a), we note that the Mori fiber space $f$ has no multiple fiber, hence is a smooth $\pp^1$-bundle over $C$. It has three disjoint sections, thus is trivial
$$(X,\Delta^{\rm st})\cong (F,\Delta^{\rm st}|_F)\times (C,\Delta_C),$$
and therefore $\pi_1^{\rm orb}(X,\Delta)$ has a normal abelian subgroup of index at most $360$ isomorphic to $\zz^2$, as wished.

In Case b), we let $s_{Y}$ and $b_Y$ be the preimages of $s$ and $b$ by the finite cyclic cover $p:Y\to X$. By Lemma \ref{lem:p1bdle-antican-system} and since $\Delta_Y$ has no component of coefficient one, the sections $2s_Y$ and $b_Y$ belong to the basepoint free pencil $|-K_Y|$. Hence, we have an elliptic fibration
$$\psi:(Y,\Delta_Y)\longrightarrow \left(\pp^1,\frac{3}{4}\{\psi(s_Y)\}+\frac{2}{3}\{\psi(b_Y)\}\right)$$
induced by $|-K_Y|$ which is equimultiple and makes the divisor $\Delta_Y^{\rm st}$ be purely vertical. Since the base curve pair has trivial orbifold fundamental group, by Lemma \ref{lem:nori}, $\pi_1^{\rm orb}(Y,\Delta_Y)$ is an abelian group of rank at most $2$. It is a normal subgroup of index at most $6$ in $\pi_1^{\rm orb}(X,\Delta)$, as wished.

\bigskip

We consider the last case: Assume that the pair $(F,\Delta|_F)$ is of toric type. We break it further down.

\medskip

If the group $\pi_1^{\rm orb}(F,\Delta|_F)$ is trivial, then Exact Sequence \eqref{eq:specialnori} concludes. 

If the group $\pi_1^{\rm orb}(F,\Delta|_F)$ is isomorphic to $(\zz/2\zz)^2$, then by Exact Sequence \eqref{eq:specialnori} and by Lemma \ref{lem:exact-seq}, we have a normal subgroup $H$ of index at most $216$ in $\pi_1^{\rm orb}(X,\Delta)$ that sits in a central extension
$$1\to (\zz/2\zz)^2\to H\to \zz^2\to 1.$$
Let $b,c$ be the preimages in $H$ of a set of two generators of $\zz^2$. If they commute, then $H\simeq(\zz/2\zz)^2\times\zz^2$ has a characteristic subgroup of index $4$ isomorphic to $\zz^2$. If they do not, it is easy to check that their squares $b^2$ and $c^2$ commute. Thus, the subgroup spanned by $b^2$ and $c^2$ is isomorphic to $\zz^2$ and characteristic of index $16$ in $H$. Either way, we constructed a normal subgroup isomorphic to $\zz^2$ in $G$ of index at most $3456$.

\medskip

If the group $\pi_1^{\rm orb}(F,\Delta|_F)$ is neither trivial nor $(\zz/2\zz)^2$, one of the following holds
\begin{enumerate}[label=\alph*)]
    \item either $\Delta_{\rm hor}^{\rm st}= a_0 s_0 + a_{\infty} s_{\infty}$ with $a_0,a_{\infty}$ positive standard coefficients and $s_0,s_{\infty}$ sections of $f$;
    \item or $\Delta_{\rm hor}^{\rm st}= a b$ with $a$ a positive standard coefficient and $b$ a bisection of $f$;
    \item or $\Delta_{\rm hor}^{\rm st} = \frac{1}{2}s_0+\frac{1}{2}s_1+a_{\infty} s_{\infty}$ with $a_{\infty}$ a positive standard coefficient and $s_0,s_1,s_{\infty}$ sections of $f$;
    \item or $\Delta_{\rm hor}^{\rm st} = \frac{1}{2}b+a_{\infty} s_{\infty}$ with $a_{\infty}$ a positive standard coefficients, $s_{\infty}$ a section and $b$ a bisection of $f$.
\end{enumerate}
In any case, we can increase coefficients to provide a divisor $\Gamma_{\rm hor}\ge \Delta_{\rm hor}^{\rm st}$ with coefficients in $\{\frac{1}{2},1\}$ such that the pair $(F,\Gamma_{\rm hor}|_F)$ has orbifold fundamental group $\zz$ or $\zz\rtimes \zz/2\zz$. 
We set $\Gamma:=\Gamma_{\rm hor}+\Delta_{\rm vert}$.
By Corollary \ref{cor:2cplt-sur}, it suffices to show the conclusion of Proposition \ref{prop:ellipticbase} for $\pi_1^{\rm orb}(X,\Gamma)$ to derive it for $\pi_1^{\rm orb}(X,\Delta)$.
We also have an exact sequence
\begin{equation}\label{eq:specialnorigamma}
1\to \pi_1^{\rm orb}(F,\Gamma|_F)\to \pi_1^{\rm orb}(X,\Gamma)\to \pi_1^{\rm orb}(C,\Delta_C)\to 1,
\end{equation}
where injectivity follows as for Exact Sequence \eqref{eq:specialnori} and we use the fact that $\Gamma_C=\Delta_C$. 

\medskip

Therefore, in Cases a) and b), we have $\pi_1^{\rm orb}(F,\Gamma|_F)\simeq \zz$. Lemma \ref{lem:exact-seq} applies and provides a normal subgroup $H$ of index at most $24$ in $\pi_1^{\rm orb}(X,\Gamma)$ that is a central extension of $\zz^2$ by $\zz$, i.e., a Heisenberg group $H_k$ for some $k\ge 0$, as wished.

\medskip

Finally, in Cases c) and d), we have $\pi_1^{\rm orb}(F,\Gamma|_F)\simeq \zz\rtimes\zz/2\zz$. This group thus has a characteristic subgroup isomorphic to $\zz$, which is normal in $\pi_1^{\rm orb}(X,\Gamma)$ and provides two exact sequences
\begin{gather}
    \begin{aligned}
     1\to \zz\to \pi_1^{\rm orb}(X,\Gamma)\to Q\to 1 \quad \mbox{and}\\
     1\to \zz/2\zz \to Q \to \pi_1^{\rm orb}(C,\Delta_C)\to 1.
    \end{aligned}
\end{gather}
We claim that $Q$ has a normal abelian subgroup of index at most $48$ isomorphic to $\zz^2$.
Indeed, by Lemma \ref{lem:exact-seq}, we have a normal subgroup $N$ of index at most $6$ in $Q$ that sits in a central extension
$$1\to \zz/2\zz\to N\to \zz^2\to 1.$$
Let $b,c$ be the preimages in $N$ of a set of two generators of $\zz^2$. If they commute, then $N$ is isomorphic to $\zz/2\zz\times \zz^2$. Its characteristic subgroup $\zz^2$ is thus a normal subgroup of index at most $12$ in $Q$, as wished. If $b$ and $c$ do not commute, it is easy to check that their squares $b^2$ and $c^2$ commute. Thus, the subgroup spanned by $b^2$ and $c^2$ is characteristic, isomorphic to $\zz^2$, and of index at most $8$ in $N$. This provides a normal subgroup of index at most $48$ in $Q$, and concludes the proof of our claim. 
We can now apply Lemma \ref{lem:exact-seq} to get that $\pi_1^{\rm orb}(X,\Gamma)$ has a normal subgroup of index at most $96$ that is a central extension of $\zz^2$ by $\zz$, as wished.

\medskip

To conclude this proof, we want to investigate the situations that lead to an orbifold fundamental group $\pi_1^{\rm orb}(X,\Delta)$ that is not virtually abelian, or that contains a copy of $\zz^3$. These are, by the previous discussion, cases where the induced curve pair on the general fiber $(F,\Delta|_F)$ is of toric type, and one of the Cases a), b), c), or d) above occurs. 

Assume for a moment that the surjection induced by the coefficient increase
$\pi_1^{\rm orb}(F,\Gamma|_F)\twoheadrightarrow \pi_1^{\rm orb}(F,\Delta|_F)$
has a non-trivial kernel. The Heisenberg group $H_k$ occuring as a  normal finite index subgroup of
$\pi_1^{\rm orb}(X,\Gamma)$ then surjects onto a normal finite index subgroup of $\pi_1^{\rm orb}(X,\Delta)$ with finite presentation
$$P = \langle a,b,c\mid a^n = [a,b]=[a,c]=1,[b,c]=a^k\rangle,$$
where $n\ge 2$ is the order, respectively half the order, of $\pi_1^{\rm orb}(F,\Delta|_F)$ if it is Then, note that $\langle a, b^n, c^n\rangle  < P$ is a normal subgroup of index $n^2$ in $P$, and is isomorphic to $\zz/n\zz\times \zz^2$: This yields that $\pi_1^{\rm orb}(X,\Delta)$ is in fact virtually abelian and contains no copy of $\zz^3$, a contradiction to our premises.

We can now assume that $\pi_1^{\rm orb}(F,\Gamma|_F)\simeq \pi_1^{\rm orb}(F,\Delta|_F)$, i.e., the coefficients appearing in the afore-mentioned Cases a), b), c) and d) above all equal $1$. The canonical bundle formula for the fibration $\phi:(Y,\Delta_Y)\to (E,0)$ further shows that in Case a), the two sections $s_0$ and $s_{\infty}$ are disjoint, which shows that $(Y,\Delta_Y)$ is of the form $(\mathbb{P}(\mathcal{O}_E\oplus L),s_0+s_{\infty})$ for some line bundle $L$ on $E$ and some sections $s_0, s_{\infty}$ of the $\pp^1$-bundle structure. In the non virtually abelian situation, the computation of Example \ref{ex:p1-over-elliptic} shows that $L$ has degree $k\neq 0$, thus $s_0$ and $s_{\infty}$ are indeed the two disjoint sections stemming from the factors. In the situation with a copy of $\zz^3$, the computation of Example \ref{ex:p1-over-elliptic} shows that $L$ has degree zero, as wished.

In Case c), the canonical bundle formula shows that $s_{\infty}$ is disjoint from the other two sections, and that if the sections $s_0$ and $s_1$ intersect, this must be a transversal intersection. If $s_0$ and $s_1$ intersect, then we can perform elementary transformations of the Mori fiber space $\phi: Y\to E$ at all of their intersection points: After this birational equivalence, we can assume that the three sections $s_0, s_1,s_{\infty}$ are disjoint. This implies that $(Y,\Delta_Y)$ is isomorphic to a product of two Calabi--Yau curve pairs. In particular, it has a virtually abelian fundamental group. It is easy to check that this orbifold fundamental group contains a copy of $\zz^3$, and that (possibly after an additional index one double cover), the pair is in fact of the form $(\pp^1\times E, s_0+s_{\infty})$ that we wished.

Finally, in Cases b) and d), the bisection $b$ admits an isogeny of degree $2$ to $E$, thus a base change by multiplication by two on the elliptic curve $E$ induces a compatible Galois cover $r:(\tilde{Y},\tilde{\Delta})\to (Y,\Delta_Y)$ of degree $4$, where $r_{\bullet}\pi_1^{\rm orb}(\tilde{Y},\tilde{\Delta})$ is a characteristic subgroup of $\pi_1^{\rm orb}(Y,\Delta_Y)$, thus a normal subgroup of finite index of $\pi_1^{\rm orb}(X,\Delta)$. However, the pre-image $r^*b$ has become two sections of the Mori fiber space $\tilde{\phi}:\tilde{Y}\to E$, which reduces us to the previously treated Cases a) (the compatible cover having index at most $96$ now) and c) respectively. This concludes the proof.
\end{proof}

\subsection{Proof of Proposition \ref{prop:mfs}}
We can now prove Proposition \ref{prop:mfs}.

\begin{proof}[Proof of Proposition \ref{prop:mfs}]
We have our lc Calabi--Yau surface pair $(X,\Delta)$ with its equimultiple Mori fiber space onto a smooth curve
$f\colon (X,\Delta) \rightarrow (C,\Delta_C)$. By Lemma \ref{lem:sarkisovlink}, we assume without loss of generality that all log canonical centers of the pair $(X,\Delta)$ are contained in $\lfloor\Delta\rfloor$. In particular, the surface $X$ is klt, as well as any surface $Y$ arising from a compatible finite Galois cover $p:(Y,\Delta_Y)\to (X,\Delta)$. Indeed, this follows from \cite[Proposition 5.20]{KM98} applied to the compatible cover $(Y,0)\to (X,{\rm Branch}(p))$ induced by $p$.

By Proposition \ref{prop:res-finite-mfs}, the group $\pi_1^{\rm orb}(X,\Delta)$ is residually finite. Thus, if $(X,\Delta)$ is hereditarily rationally connected, Corollary~\ref{cor:res-finite-for-surf} concludes this proof.

Assume now that $(X,\Delta)$ is not hereditarily rationally connected, but is hereditarily uniruled. In other words, there is a compatible finite Galois cover $p\colon (Y,\Delta_Y)\rightarrow (X,\Delta)$ such that the normal projective surface $Y$ is uniruled and not rationally connected. Since $Y$ is a surface, its MRC fibration is a morphism
$\phi: Y\to E$ to a smooth curve $E$ that is not rational. Since $(Y,\Delta_Y)$ is a Calabi--Yau pair and by the canonical bundle formula, this is an elliptic curve. By uniqueness of the MRC fibration, the action of $G:={\rm Gal}(p)$ on $Y$ descends to an action on $E$, inducing a quotient map $q:E\to E/G=:C'$ and an equimultiple fibration $\psi: (X,\Delta)\to (C',{\rm Branch}(q))$ with rational fibers. Since $X$ is a Mori fiber space over the curve $C$, we have $\rho(X)=2$, and thus $X$ is a Mori fiber space over $C'$ as well. By the Riemann--Hurwitz formula, this shows that $(X,\Delta)$ is in fact a Mori fiber space over a curve $C'$ with a base pair $(C',{\rm Branch}(q))$ of elliptic type. Therefore, Proposition \ref{prop:ellipticbase} applies and concludes in this case.

Assume finally that $(X,\Delta)$ is not hereditarily uniruled. In other words, there is a compatible finite Galois cover $p\colon (Y,\Delta_Y)\rightarrow (X,\Delta)$ such that the normal projective surface $Y$ is not uniruled. Thus by \cite[Corollary 0.3]{BDPP13}, the minimal resolution $\tilde{Y}$ of $Y$ has a pseudoeffective canonical class. Since $Y$ is klt, this shows that $K_Y$ itself is a pseudoeffective class. However, the pair $(Y,\Delta_Y)$ is also Calabi--Yau, so $K_Y$ must be numerically trivial, and $\Delta_Y = 0$. Hence, we have $\Delta={\rm Branch}(p)$. The pair $(X,\Delta)$ is thus a klt Calabi--Yau pair with standard coefficients, and Theorem~\ref{introthm:fun-group-klt} concludes in this case. 
\end{proof}

\section{Log canonical Calabi-Yau surface pairs with standard coefficients}

The main result of this section, worth comparing to the main result of Section~\ref{sec:klt-cy-stand}, is the following.

\begin{proposition}\label{prop:lc-cy}
Let $(X,\Delta)$ be a non-klt log canonical Calabi--Yau surface with $\Delta=\Delta^{\rm st}$.
Then the group $\pi_1^{\rm orb}(X,\Delta)$ admits a normal subgroup of index at most $7200$ that is abelian of rank at most $3$, or a quotient of the nilpotent group $H_k$ for some $k\ge 1$. Moreover, if $\pi_1^{\rm orb}(X,\Delta)$ is not virtually abelian, then $(X,\Delta)$ is birationally equivalent to a log canonical Calabi--Yau pair on surface that is a Mori fiber space to a curve.
\end{proposition}


Before proving this proposition, we prove a few lemmas. First, we propose a technical result about $K_X$-MMP for certain dlt pairs $(X,\Delta)$.

\begin{lemma}\label{lem:MMP-and-pi1reg}
    Let $(X,\Delta)$ be a dlt pair of dimension 2 such that $K_X+\Delta\sim 0$ and $\Delta\ne 0$.
    Then, if we run a $K_X$-MMP on $X$, it terminates with a surface $X_0$ that has canonical singularities, sits in an induced pair $(X_0,\Delta_0)$ which is dlt, satisfies $K_{X_0}+\Delta_0\sim 0$ and $\Delta_0\neq 0$, and 
    $$\pi_1({X}_{{\rm reg}})\simeq \pi_1(X_{0,{\rm reg }}).$$
\end{lemma}

\begin{proof}
First, note that $X$ is a klt surface. Moreover, since $\Delta$ is integral, its supported is contained in the non-klt locus of the pair $(X,\Delta)$, hence by definition of dlt singularities, $X$ is smooth along the $\Delta$. In particular, the Weil divisor $\Delta$ is Cartier and thus $X$ is Gorenstein. Hence, $X$ has canonical singularities. 

To prove the rest of our statement, it suffices to argue by induction and check it for one step in the $K_X$-MMP, say a proper birational map
$\varepsilon:X_{i+1}\to X_i$ contracting a unique curve $C$ of negative intersection with $K_{X_i}$ spanning an extremal ray of $\overline{{\rm NE}}(X_i)$. Clearly in this set-up, we have $C^2 < 0$.
Therefore,
$$\varepsilon^*K_{X_{i}}+aC = K_{X_{i+1}},$$
for some positive coefficient $a$. This shows that $X_i$ has canonical singularities. It also shows that the induced divisor $\Delta_i:=\varepsilon_*\Delta_{i+1}$ is non-zero (since $\Delta_{i+1}\sim -K_{X_{i+1}}$ cannot be proportional to $C$).

Let us show that $P\in X_{i,{\rm reg}}$. Indeed, let $\mu:\tilde{X}\to X_{i+1}$ be the minimal resolution of the finitely many (canonical) singular points of $X_{i+1}$ contained in the curve $C$. We see that the exceptional locus of $\varepsilon\circ\mu$ is exactly the support of $\mu^*C$. Since it also holds
$$\mu^*\varepsilon^*K_{X_{i}}+a\mu^*C_i = \mu^*K_{X_{i+1}} = K_{\tilde{X}},$$
and as we recall that $a>0$, this shows that $X_i$ is terminal at the point $P$. Since $X_i$ is a surface, we thus get $P\in X_{i,{\rm reg}}$, as wished.

We can now apply Lemma~\ref{lem:fun-under-surj-smooth} to obtain $
\pi_1(X_{i+1,{\rm reg}})=\pi_1(X_{i,{\rm reg}})$ and thus conclude the proof of this lemma. Moreover, note that the pair $(X_i,\Delta_i)$ automatically inherits being log canonical from the old pair $(X_{i+1},\Delta_{i+1})$. Using that $P\in X_{i,{\rm reg}}$ and that $\Delta_i$ has coefficients in $\{0,1\}$ makes it easy to check that the pair $(X_i,\Delta_i)$ is dlt at the point $P$, hence dlt.
\end{proof}

Second, we prove that a very particular type of dlt Calabi--Yau pairs always have residually finite orbifold fundamental groups.

\begin{lemma}\label{lem:rank-1-dP-Gor-with-simply-smooth-locus}
Let $(X,\Delta)$ be a dlt pair of dimension 2 such that $K_X+\Delta\sim 0$ and $\Delta\ne 0$.
Assume that $X$ is a klt Fano surface of Picard rank one and that the group $\pi_1(X_{\rm reg})$ is trivial.
Then the group $\pi_1^{\rm orb}(X,\Delta)$ is abelian of rank at most $3$.
\end{lemma}

\begin{proof}[Proof of Lemma \ref{lem:rank-1-dP-Gor-with-simply-smooth-locus}]
First note that if $X$ is smooth, then $X\simeq \pp^2$ and $\Delta$ is a sum of three lines $\ell_1+\ell_2+\ell_3$, a sum of a smooth conic $Q$ and a non-tangent line $\ell$, or a smooth or nodal cubic $C$. We compute these fundamental groups: In the first case, we have $\pi^{\rm orb}(\pp^2,\ell_1+\ell_2+\ell_3)\simeq \zz^3$. In the second case, we blow up the two points in $Q\cap \ell$ and blow down the strict transform of $\ell$. We obtain $\pi^{\rm orb}(\pp^2,Q + \ell)\simeq \pi^{\rm orb}(\pp^1\times\pp^1,F_0+F_1+\Delta_{\pp^1})$ with $F_0,F_1$ fibers of the two projections and $\Delta_{\pp^1}$ the diagonal. By Lemma \ref{lem:nori}, this group is abelian of rank at most $2$. 
In the case when $C$ is a nodal cubic, we have $\pi_1^{\rm orb}(\pp^2,C)\simeq \pi_1^{\rm orb}(\mathbb{F}_1,s_0+s_1)$ with $\mathbb{F}_1$ the first Hirzebruch surface, $s_0$ its minimal section (of self-intersection $-1$) and $s_1$ another section intersecting $s_0$ at two points, transversally. By Lemma \ref{lem:nori}, this group is abelian of rank at most $2$. 
Finally, if $C$ is a smooth cubic, we blow-up one point on it and obtain $\pi_1^{\rm orb}(\pp^2,C)\simeq \pi_1^{\rm orb}(\mathbb{F}_1,s_0+b)$ with $\mathbb{F}_1$ the first Hirzebruch surface, $s_0$ its minimal section (of self-intersection $-1$) and $b$ a bisection intersecting $s_0$ at exactly one point $x$, transversally. Applying Lemma \ref{lem:Nori-reg-fun} at the point $x$, this group is abelian of rank at most $2$. This concludes the proof if $X\simeq \pp^2$.

From here on, we can assume that $X$ is singular. Since the support of $\Delta$ coincides with the non-klt locus of the dlt pair $(X,\Delta)$, it must be contained in the smooth locus of $X$.
Hence, $\Delta$ is a Cartier divisor, and thus $X$ is a normal Gorenstein Fano surface with quotient singularities, with
$\rho(X)=1$, and with $\pi_1(X_{{\rm reg}})=\{1\}$.
By the classification result due to Miyanishi-Zhang \cite[Lemma 6(1)]{MZ88}, either the Cartier divisor $-K_X$ generates ${\rm Pic}(X)\simeq\mathbb{Z}$, or $X$ is isomorphic to the weighted projective space $\mathbb{P}(1,1,2)$ and the divisor $-\frac{1}{2}K_X$ is Cartier and generates ${\rm Pic}(X)\simeq\mathbb{Z}$. Since $\Delta\sim -K_X$ is a sum of components which all are contained in the smooth locus of $X$, hence Cartier, either $\Delta$ is irreducible, or that $X\simeq \mathbb{P}(1,1,2)$ and $\Delta$ has exactly two components $d_0,d_{\infty}\in|\mathcal{O}_{\pp(1,1,2)}(2)|$. 

We first assume that $X\simeq \mathbb{P}(1,1,2)$ and $\Delta$ has exactly two components contained in $|\mathcal{O}_{\pp(1,1,2)}(2)|$. We blow-up the $A_1$-singular point $\varepsilon:\mathbb{F}_2\to X$ with exceptional divisor $s_0$. We can write the strict transform of $\Delta$ as $s_{1}+s_{2}$, a sum of two smooth sections of self-intersection $2$ of $\mathbb{F}_2$ (thus both disjoint from  $s_0$). Applying Lemma \ref{lem:nori} at an intersection point of $s_{1}$ and $s_{2}$ shows that $$\pi_1^{\rm orb}(X,\Delta)\simeq\pi_1^{\rm orb}(\mathbb{F}_2,s_0+s_{1}+s_{2})$$ is abelian of rank at most 2.

From now on, we can assume that $\Delta$ is irreducible. 
We take a minimal resolution $\varepsilon\colon \tilde{X}\rightarrow X$, denote its reduced exceptional divisor by $E$ and introduce the reduced irreducible divisor $\tilde{\Delta}:=\varepsilon^*\Delta$, which is disjoint from $E$. By \cite[Lemma 3]{MZ88}, we can run a $K_{\tilde{X}}$-MMP that terminates with the Mori fiber space structure of the second Hirzebruch surface $f\colon \mathbb{F}_2\rightarrow \pp^1$. The divisors obtained from $E$ and $\tilde{\Delta}$ on $\tilde{X}$ by pushforward through that MMP are denoted by $E_{\mathbb{F}}$ and $\Delta_{\mathbb{F}}$ respectively. By Lemma \ref{lem:MMP-and-pi1reg}, the pair $(\mathbb{F}_2,\Delta_{\mathbb{F}})$ is a dlt $1$-complement, and by Lemma~\ref{lem:fun-under-surj}, it suffices to prove that the group $\pi_1^{\rm orb}(\mathbb{F}_2,E_{\mathbb{F}}+\Delta_{\mathbb{F}})$ is abelian of rank at most $2$.

By~\cite[Appendix, Figure 1-6]{Ye02},  the divisor $E_{\mathbb{F}}$ in $\mathbb{F}_2$ is contained in $F_0\cup s_0$, where $F_0$ is a fiber of $f$ and $s_0$ is its minimal section. The curve $\Delta_{\mathbb{F}}\in |-K_{\mathbb{F}_2}|$ is a bisection of $f$, in particular it has arithmetic genus $1$. If it has geometric genus $1$, it is smooth. At a point where the double cover $f|_{\Delta_{\mathbb{F}}}$ is ramified, applying Lemma~\ref{lem:Nori-reg-fun} then shows that the group $\pi_1^{\rm orb}(\mathbb{F}_2,F_0+s_0+\Delta_{\mathbb{F}})$ is abelian of rank at most two, as wished. Otherwise, $\Delta_{\mathbb{F}}$ has geometric genus $0$ and a single node, and applying Lemma~\ref{lem:Nori-reg-fun} at the nodal point concludes.
\end{proof}

We can now prove Proposition~\ref{prop:lc-cy}.

\begin{proof}[Proof of Proposition~\ref{prop:lc-cy}]
We assume that the Calabi--Yau pair $(X,\Delta)$ is log canonical but not klt. In particular, it has coregularity $0$ or $1$. By Lemmas \ref{lem:coreg-vs-complementCY} and~\ref{lem:index-one-cover}, we can thus take the index one cover $p:(Y,\Delta_Y) \rightarrow (X,\Delta)$, which is cyclic and of degree at most $6$. Note that, by construction, $(Y,\Delta_Y)$ is a $1$-complement.
Using \cite[Proposition 2.16]{MS21}, we take a ${\rm Gal}(p)$-equivariant dlt modification $\varepsilon_Y:(Y',\Delta_{Y'})\to (Y,\Delta_Y)$.
We have a commutative diagram as follows 
\[
\xymatrix{
(Y',\Delta_{Y'}) \ar[r]^-{/{\rm Gal}(p)} \ar[d]_-{\varepsilon_Y} 
& (X',\Delta')  \ar[d]^-{\varepsilon} \\ 
(Y,\Delta_{Y}) \ar[r]_-{p}
& (X,\Delta) .
}
\]
By Proposition \ref{prop:galois1}, the group $\pi_1^{\rm orb}(Y',\Delta_{Y'})$ embeds as a normal subgroup of index at most $6$ in $\pi_1^{\rm orb}(X',\Delta')$.
The pair $(X',\Delta')$ remains a log canonical Calabi--Yau pair with coefficients in $\{0,1\}$, and since $\varepsilon_Y$ is a dlt modification, the reduced exceptional divisor $E$ of $\varepsilon$ satisfies $\Delta'\ge E$. So by the same argument as Lemma \ref{lem:dlt-mod-vs-fun}, we have an isomorphism 
$\pi_1^{\rm orb}(X',\Delta') \simeq \pi_1^{\rm orb}(X,\Delta)$.
We proceed in two cases.

\medskip 

\noindent\underline{\textit{Case 1:}} Assume that $(Y',\Delta_{Y'})$ is not hereditarily rationally connected\\

By Proposition \ref{prop:galois1}, let $H\triangleleft\pi_1^{\rm orb}(Y',\Delta_{Y'})\triangleleft \pi_1^{\rm orb}(X',\Delta')$
be a normal subgroup of finite index corresponding to a non rationally connected compatible finite Galois cover of $(Y',\Delta_{Y'})$. We take $N\triangleleft \pi_1^{\rm orb}(X',\Delta')$ to be the largest subgroup of $H$ that is normal in $\pi_1^{\rm orb}(X',\Delta')$ and consider the compatible finite Galois cover induced by Proposition \ref{prop:galois2}, that is $q:(Z,\Delta_Z)\to (X',\Delta_{X'})$. Note that $Z$ is not rationally connected. 

By \cite[Proposition 5.20(4)]{KM98}, $Z$ is a klt surface.
We run a ${\rm Gal}(q)$-equivariant $K_{Z}$-MMP. Since $Z$ is not rationally connected and since both klt Fano surfaces and klt conic bundles are rationally connected, it terminates with a ${\rm Gal}(q)$-equivariant Mori fiber space to a smooth elliptic curve. Quotienting by ${\rm Gal}(q)$ yields a $K_{X'}$-MMP that terminates with a Mori fiber space to a curve, and Lemma~\ref{lem:fun-under-surj} and Proposition~\ref{prop:mfs} conclude the proof in this case.

\medskip

\noindent\underline{\textit{Case 2:}} Assume that $(Y',\Delta_{Y'})$ is hereditarily rationally connected. 

\medskip

This means that $(X',\Delta')$ is hereditarily rationally connected too. By Corollary~\ref{cor:res-finite-for-surf}, it thus suffices to show that $\pi_1^{\rm orb}(X',\Delta')$ is residually finite to conclude. It clearly suffices to show that the group $\pi_1^{\rm orb}(Y',\Delta_{Y'})$ is residually finite.

\medskip

\hfill \begin{minipage}{0.95\textwidth}
\noindent\underline{\textit{Case 2.1:}} Assume that the group $\pi_1(Y'_{\rm reg})$ is infinite.

\medskip

We run a $K_{Y'}$-MMP and denote its end-product by $Y_0$ with the pushforward divisor $\Delta_0$. By Lemma \ref{lem:MMP-and-pi1reg}, the pair $(Y_0,\Delta_0)$ is dlt, we still have $K_{Y_0}+\Delta_0\sim 0$ and $\Delta_0\neq 0$, and the group $\pi_1(Y_{0,{\rm reg}})$ is infinite. Since $Y_0$ is rationally connected, the canonical divisor $K_{Y_0}$ is not nef, hence the surface $Y_0$ is a Mori fiber space. By \cite[Theorem 2]{Bra20}, it must be a Mori fiber space over a curve, thus Proposition \ref{prop:res-finite-mfs}, and Lemmas~\ref{lem:fun-under-surj} and \ref{lem:resfinimage} show that $\pi_1^{\rm orb}(Y',\Delta_{Y'})$ is residually finite, as wished.
\end{minipage}

\medskip 

\hfill \begin{minipage}{0.95\textwidth}
\noindent\underline{\textit{Case 2.2:}} Assume that the group $\pi_1(Y'_{\rm reg})$ is finite.

\medskip

Lemma \ref{lem:fun-under-surj} provides a surjective homomorphism
$\pi_1^{\rm orb}(Y',\Delta_{Y'})\twoheadrightarrow \pi_1(Y'_{\rm reg})$. Its kernel now is a normal subgroup of finite index of $\pi_1^{\rm orb}(Y',\Delta_{Y'})$ and corresponds by Proposition \ref{prop:galois2} to a compatible finite Galois cover $q\colon (Z,\Delta_{Z})\rightarrow (Y',\Delta_{Y'})$. 
Note that $(Z,\Delta_Z)$ is still a dlt pair with $K_{Z}+\Delta_{Z}\sim 0$ and $\Delta_Z\neq 0$. Moreover, by construction, the fundamental group of the quasiprojective variety $q^{-1}(Y_{0,{\rm reg}})$ is trivial. Since $\Delta_0$ is snc and contained in the smooth locus $Y_{0,{\rm reg}}$, the preimage $q^{-1}(Y'_{\rm reg})$ is a Zariski open subset of $Z_{\rm reg}$. So $\pi_1(Z_{\rm reg})=\{1\}.$
\end{minipage}

\hfill \begin{minipage}{0.95\textwidth}
Now, it suffices to show that $\pi_1^{\rm orb}(Z,\Delta_Z)$ is residually finite. We run a $K_Z$-MMP, which terminates with a pair $(Z_0,\Delta_0)$. This pair is dlt, we still have $K_{Z_0}+\Delta_0\sim 0$ and $\Delta_0\neq 0$, and the group $\pi_1(Z_{0,{\rm reg}})$ is trivial. Since $Z$ is rationally connected, the surface $Z_0$ is a Mori fiber space. If it is a Mori fiber space over a curve, Proposition \ref{prop:res-finite-mfs}, and Lemmas~\ref{lem:fun-under-surj} and \ref{lem:resfinimage} show that $\pi_1^{\rm orb}(Z,\Delta_{Z})$ is residually finite, as wished. Otherwise, the surface $Z_0$ is Fano and of Picard number one, thus Lemma \ref{lem:rank-1-dP-Gor-with-simply-smooth-locus} concludes.
\end{minipage}
\end{proof}

\section{Log canonical Fano surfaces} 
The main result of this section is the following proposition.

\begin{proposition}\label{prop:lc-Fano-case}
Let $(X,\Delta)$ be a log canonical Fano surface pair.
Then, the group $\pi_1^{\rm orb}(X,\Delta)$ admits a normal subgroup of index at most $7200$ that is abelian of rank at most $3$, or a quotient of the nilpotent group $H_k$ for some $k\ge 1$. Moreover, if $\pi_1^{\rm orb}(X,\Delta)$ is not virtually abelian, then $(X,\Delta)$ is birationally equivalent to a log canonical Calabi--Yau pair on surface that is a Mori fiber space to a curve.
\end{proposition}

We will first establish several lemmas which deal with particularly relevant special cases arising in the proof of Proposition \ref{prop:lc-Fano-case}, after an important reduction step. All of these lemmas are set-up with a plt Fano surface pair $(Z,S+L)$, where
$S$ is a curve in $Z$, $L$ is a divisor with standard coeficients strictly smaller than $1$, and the normal surface $Z$ has Picard rank $1$. They also all share the assumption that the curve pair obtained by adjunction of $(Z,S+L)$ to $S$ is of sporadic type. Among various additional assumptions, our lemmas describe the orbifold fundamental group of the pair $\pi_1^{\rm orb}(Z,S+L)$. Different assumptions yield different conclusions, but one thing remains: This orbifold fundamental group is always residually finite.

\subsection{Some toric preparatory results}

We start with the following lemma, which provides information on orbifold fundamental groups in a toric setting. 

\begin{lemma}\label{lem:toric-fund-group-1}
Let $Z$ be a normal toric surface of Picard rank $1$ and $S_1,S_2$ be two distinct reduced irreducible components of the toric boundary divisor of $Z$. Then, if there is a curve $C$ that intersects $S_1$ transversally at a unique point, smooth in $Z$ and distinct from the intersection point $S_1\cap S_2$, then the group $\pi_1^{\rm orb}(Z,S_1+S_2+C)$ is abelian of rank at most two.
\end{lemma}

\begin{proof}
Let $\varepsilon:Y\rightarrow Z$ be the toric blow-up of the torus fixed point $S_1\cap S_2$, and denote by $E$ its exceptional divisor. In terms of fans, if we denote by $\Sigma_Z(1)=\{v_1,v_2,v_3\}$ the rays of the toric fan of $Z$, where $v_1$ and $v_2$ correspond to $S_1$ and $S_2$ respectively, the toric fan of $Y$ has rays $\Sigma_Y=\{v_1,-v_3,v_2,v_3\}$.
In particular, if we denote by $S_{Y,i}$ the strict transform of $S_i$ and
by $C_Y$ the strict transform of $C$, we have $(S_{Y,i})^2 = 0$, $S_{Y,1}\cdot S_{Y,2} = 0$, and $-K_Y\cdot S_{Y,i}>0$. Since $\rho(Y)=2$, this provides a Mori fiber space structure $f\colon Y\rightarrow B$ onto a smooth toric curve $B$ ($\simeq\pp^1$) of which $S_{Y,1}$ and $S_{Y,2}$ are two distinct irreducible fibers.
Moreover, note that the general fiber $F$ of $f$ and the torus-invariant divisor $S_{Y,3}$ corresponding to the vector $v_3$ in the fan $\Sigma_Y$ satisfy $K_Y+S_{Y,1}+S_{Y,2}+S_{Y,3}+E\equiv 0$, hence $F\cdot E = F\cdot S_{Y,3} = 1$. So $E$ is a section of $f$.
Since multiple fibers of $f$ must be torus-invariant, they may only arise along the $S_{Y,i}$.

  The pair $(Y,S_{Y,1}+S_{Y,2}+C_Y+E)$ induces a curve pair on the base $B$ of orbifold fundamental group $\zz$, generated by the class $f_{\bullet}(\gamma_{S_{Y,1}})$. Let $y$ be the unique point of (transversal) intersection of $C_Y$ and $S_{Y,1}$. It satisfies the assumptions of Lemma \ref{lem:Nori-reg-fun}, thus yields a surjective group homomorphism
$$\pi_1^{\rm reg}(Y,S_{Y_1}+S_{Y,2}+C_Y+E;y)\twoheadrightarrow \pi_1^{\rm orb}(Y,S_{Y,1}+S_{Y,2}+C_Y+E).$$
Moreover, we have $\pi_1^{\rm reg}(Y,S_{Y_1}+S_{Y,2}+C_Y+E;y)\simeq \pi_1^{\rm reg}(Y,S_{Y_1}+C_Y;y)$. Since locally at $y$, the surface $Y$ is smooth and the divisor $S_{Y_1}+C_Y$ is snc, this group is isomorphic to $\zz^2$, as wished.
\end{proof}

The following definition is made {\it ad hoc}, and will be crucial to several arguments in a toric set-up as well.

\begin{definition}
    Let $Z$ be a normal toric surface with Picard rank $1$ and $S$ be an irreducible reduced component of the toric boundary divisor of $Z$. 
    We define the {\it standard $Z- S$ toric cover} as the unique toric finite Galois cover $p:X\to Z$ that is étale over $Z_{\rm reg}\setminus S$ and such that the pre-image
    $p^{-1}(Z\setminus S)$ is isomorphic to the smooth affine plane $\cc^2$.
\end{definition}

\begin{remark}\label{rem:standardtoric} We make a few remarks pertaining to this definition.
\begin{enumerate}
    \item Note that such a cover always exists: Indeed, let $\Sigma_Z$ be the toric fan of the surface $Z$ in the lattice $N_Z$; it is spanned by three vectors $v_1,v_2,v_3$, where the ray spanned by the vector $v_3$ corresponds to the torus-invariant prime divisor $S$. Considering the same fan $\Sigma_X := \Sigma_Z$ with respect to the sublattice $N_X := v_1\zz+v_2\zz\subseteq N_Z$, we obtain a toric morphism $p:X\to Z$ which is a finite Galois cover by \cite[Example 5.0.13]{CLS11}.    
    Clearly, the preimage $p^{-1}(Z\setminus S)$ is isomorphic to an affine toric variety of dimension $2$ encoded by the cone $v_1\rr_{+} +v_2\rr_{+}$ and the lattice $v_1\zz+v_2\zz$: It is clearly smooth and isomorphic to $\cc^2$, as wished. Clearly, the morphism $p$ is quasiétale above $Z\setminus S$, hence étale above $Z_{\rm reg}\setminus S$.
    \item The uniqueness of $p$ follows from toricity and the fact that $\cc^2$ is smooth and simply connected.
    \item By construction, the normal toric surface $X$ still has Picard rank $1$. The pre-image $S_X:= p^{-1}(S)$ remains a prime torus-invariant divisor in $X$; in the toric fan $\Sigma_X$, it corresponds to the ray spanned by the vector $v_3$.
    \item The singular points of $X$ and their respective orbifold indices are in bijection with the singular points of $Z$ contained in the curve $S$ with their respective orbifold indices in $Z$.
\end{enumerate}
    
\end{remark}

\subsection{When $L$ has two or more components}

We prove the next few lemmas using a toric characterization and the toric preparatory tools we developed.

\begin{proposition}\label{prop:toric-fund-group-3comps}
Let $(Z,S+L)$ be a plt Fano surface pair, where
$S$ is a curve in $Z$, $L$ is a divisor with standard coeficients strictly smaller than $1$, and $Z$ has Picard rank $1$. Assume moreover 
that the curve pair $(S,\Delta_S)$ obtained by adjunction of $(Z,S+L)$ to $S$ is of sporadic type.
If $L$ has three or more irreducible components, then the group $\pi_1^{\rm orb}(Z,S+L)$ is finite.
\end{proposition}

\begin{proof}
Let $d$ be the number of components of $L$. Since they each appear with coefficient at least $\frac{1}{2}$, the complexity of $(Z,S+L)$ as in Definition~\ref{def:complexity} satisfies 
$$c(Z,S+L+\varepsilon C)\le 3-1-\frac{d}{2}-\varepsilon.$$ 
By Proposition~\ref{prop:toric-characterization}, we derive that $d=3$, $Z$ is a normal toric surface, and $S+\lfloor L\rfloor$ is contained in the toric boundary divisor.


By assumption, we have $(S,\Delta_S)\cong (\pp^1,\frac{1}{2}\{0\}+\frac{2}{3}\{1\}+\frac{n-1}{n}\{\infty\})$ for some $n\in\{3,4,5\}$.
Since $Z$ has Picard rank $1$, each of the three components of $L$ intersects $S$, thereby contributing positively to the divisor $\Delta_S$ as prescribed by the adjunction formula (Lemma \ref{lem:coeff-under-adj}). Since $2,3,5$ are prime numbers and $4$ can only be written non-trivially as a product of $2$ with itself, that either $S$ is contained in $Z_{\rm reg}$, or $n=4$ and $S$ contained exactly one singular point of $Z$ of orbifold index $2$.

From there, we consider the standard $Z - S$ toric cover $p:X\to Z$. By Remark \ref{rem:standardtoric}(4), the normal toric surface $X$ has Picard rank one, and is either smooth, or has exactly one singular point of orbifold index two. Therefore, $X$ is isomorphic to $\pp^2$ or to the weighted projective space $\pp(1,1,2)$. The pre-image $S_X$ of $S$ in $X$ is a line.
Let $L_X:=p^*L$, and note that $p:(X,S_X+L_X)\to (Z,S+L)$ is a compatible finite Galois cover. By Proposition \ref{prop:galois1}, we are reduced to showing that the group $\pi_1^{\rm orb}(X,S_X+L_X)$ is finite.

Note that $L_X$ still has exactly three components, and that the pair $(S_X,\Delta_{S_X})$ obtained by adjunction of $(X,S_X+L_X)$ to $S_X$ remains of the same sporadic type as $(S,\Delta_S)$. In particular, each component of $L_X$ intersects $S_X$ at exactly one point. 

If $X\simeq \pp^2$, the components of $L_X$ are three lines $\ell_1,\ell_2,\ell_3$ in the order of decreasing standard coefficients. We blow-up the unique intersection point $\ell_1\cap S_X$ in $X$, then apply Lemma \ref{lem:Nori-reg-fun} at the unique intersection point $\ell_2\cap \ell_3$: It shows that $\pi_1^{\rm orb}(X,S_X+L_X)$ has a normal subgroup of index at most $5$ that is a quotient of 
$$\pi_1^{\rm reg}\left(\cc^2,\frac{1}{2}\cc\times \{0\} + \frac{2}{3}\{0\}\times \cc;(0,0)\right)\simeq \zz/6\zz,$$
thus it is indeed finite.

Otherwise, we have $X\simeq\pp(1,1,2)$, and one component of $L_X$, say $\ell_1$, has standard coefficient $\frac{1}{2}$ and intersects $S_X$ at the unique singular point, i.e., $\ell_1$ is also a line. The other two components of $L_X$, say $\ell_2$ and $\ell_3$, have coefficients $\frac{1}{2}$ and $\frac{2}{3}$, and intersect $S_X$ at one smooth point each, hence they must be reduced irreducible sections of $\mathcal{O}_{\pp(1,1,2)}(2)$. In particular, they intersect each other transversally at exactly two points. We blow-up the unique singular point of $X$ which is also $\ell_1\cap S_X$, then apply Lemma \ref{lem:Nori-reg-fun} at one of the intersection points in $\ell_2\cap \ell_3$: It shows that $\pi_1^{\rm orb}(X,S_X+L_X)$ has a normal subgroup of index $2$ that is a quotient of 
$$\pi_1^{\rm reg}\left(\cc^2,\frac{1}{2}\cc\times \{0\} + \frac{2}{3}\{0\}\times \cc;(0,0)\right)\simeq \zz/6\zz,$$
hence is indeed finite. This concludes this proof.
\end{proof}

\begin{proposition}\label{prop:toric-fund-group-2comps}
Let $(Z,S+L)$ be a plt Fano surface pair, where
$S$ is a curve in $Z$, $L$ is a divisor with standard coeficients strictly smaller than $1$, and $Z$ has Picard rank $1$. Assume moreover 
that the curve pair $(S,\Delta_S)$ obtained by adjunction of $(Z,S+L)$ to $S$ is of sporadic type.
If $L$ has exactly two irreducible components, then the group $\pi_1^{\rm orb}(Z,S+L)$ is virtually abelian of rank at most two.
\end{proposition}

\begin{proof}
Let $L_0$ and $L_1$ denote the two components of $L$. Since they each appear with coefficients at least $\frac{1}{2}$ in $L$, Proposition~\ref{prop:toric-characterization} applies to the lc Fano surface pair $(Z,S+(1+\varepsilon) L)$ for $\varepsilon>0$ small enough. It shows that $Z$ is a toric surface with the curve $S$ contained in its toric boundary divisor. We now investigate how the components of $L$ intersect $S$. 

By assumption, adjunction to $S$ provides a pair of sporadic type $(S,\Delta_S)\cong (\pp^1,\frac{1}{2}\{0\}+\frac{2}{3}\{1\}+\frac{n-1}{n}\{\infty\})$ for some $n\in\{3,4,5\}$.
Since $Z$ has Picard rank $1$, each component of $L$ intersects $S$, thereby contributing positively to the divisor $\Delta_S$ as prescribed by the adjunction formula (Lemma \ref{lem:coeff-under-adj}). Since $2,3,5$ are prime numbers and $4$ can only be written non-trivially as a product of $2$ with itself, any point where $S$ and the support of $L$ intersect must be both a point of transversal intersection, and either a smooth point of $X$ or a singular point of $X$ of type $A_1$ that appears with coefficient $\frac{3}{4}$ in the adjunction divisor $\Delta_S$.

\medskip

Let us first assume that the components $L_0$ and $L_1$ of $L$ intersect $S$ at exactly one point each. Then at least one of them, say $L_0$, contributes to a point with coefficient $\neq\frac{3}{4}$ in $\Delta_S$. Hence, the intersection point $Q\in S\cap L_0$ is a transversal intersection point, and a smooth point of $Z$.
Moreover, since $L_1$ intersects $S$ in exactly one point as well, we can take the log canonical Fano curve pair obtained by adjunction of $(Z,S+b_1L_1)$ to $S$: Its divisor is supported exactly on the two points of ${\rm Supp}\,\Delta_S \setminus \{Q\}$ hence, it has a $1$-complement, which by Lemma \ref{lem:lifting-complements} lifts to a $1$-complement of the pair $(Z,S+b_1L_1)$. It can be written $(Z,S+L_1+R)$ for some integral Weil divisor $R\ge 0$. Since $L_1$ only intersects $S$ at one point, the divisor $R$ has to intersect $S$ at exactly one point too, thus $R$ is effective and non-zero, i.e., ample (since $\rho(Z)=1$). Therefore, for $\varepsilon > 0$ small enough, the pair $(Z,S+\varepsilon L_0 + L_1)$ is a log canonical Fano pair. By Proposition \ref{prop:toric-characterization}, this shows that $L_1$ is contained in the toric boundary divisor of $Z$.
Hence, Lemma~\ref{lem:toric-fund-group-1} concludes.

\medskip

We now assume that $L_0$ intersects $S$ in exactly one point $Q$, and $L_1$ intersects $S$ exactly two points $P$ and $P'$.
This three points form the whole support of $\Delta_S$, which shows that either $S$ is contained in $Z_{\rm reg}$, or $n=4$ and $S$ contained exactly one singular point of $Z$ of orbifold index $2$.
From there, we consider the standard $Z - S$ toric cover $p:X\to Z$. By Remark \ref{rem:standardtoric}(4), the normal toric surface $X$ has Picard rank one, and is either smooth, or has exactly one singular point of orbifold index two. Therefore, $X$ is isomorphic to $\pp^2$ or to the weighted projective space $\pp(1,1,2)$. The pre-image $S_X$ of $S$ in $X$ is a line. Let $L_X:=p^*L$, and note that $p:(X,S_X+L_X)\to (Z,S+L)$ is a compatible finite Galois cover. By Proposition \ref{prop:galois1}, we are reduced to showing that the group $\pi_1^{\rm orb}(X,S_X+L_X)$ is virtually abelian of rank at most two.
The pair $(S_X,\Delta_{S_X})$ obtained by adjunction of $(X,S_X+L_X)$ to $S_X$ remains of the same sporadic type as $(S,\Delta_S)$, thus $L_X$ still has exactly two components $L_{0,X}$ and $L_{1,X}$.

If $X\simeq \pp^2$, the fact that $S_X$ is in the smooth locus of $X$ and the adjunction formula show that one component of $L_X$, say $L_{0,X}$, has coefficient $\frac{1}{2}$ and intersects $S_X$ at exactly one point, while the other component, namely $L_{1,X}$, has coefficient $\frac{2}{3}$ and intersects $S_X$ at exactly two points.
We blow-up one of the intersection points in $L_{1,X}\cap S_X$ in $X\simeq\pp^2$, yielding a Mori fiber space $f:\mathbb{F}_1\to \pp^1$ of which the exceptional divisor $E$ and the strict tranforms $L_{0,\mathbb{F}}$, $L_{1,\mathbb{F}}$ are sections, whilst the strict transform $S_{\mathbb{F}}$ is a smooth fiber. We apply Lemma \ref{lem:Nori-reg-fun} at the intersection point $E\cap L_{1,\mathbb{F}}$: It shows that $\pi_1^{\rm orb}(X,S_X+L_X)$ is a quotient of 
$$\pi_1^{\rm reg}\left(\cc^2,\frac{2}{3}\cc\times \{0\} + \{0\}\times \cc;(0,0)\right)\simeq \zz/3\zz\times\zz,$$
thus it is abelian of rank at most two.

Otherwise, $X\simeq \pp(1,1,2)$. We resolve its singular point by a single blow-up $\varepsilon: \mathbb{F}_2\to X$, denote by $E$ the exceptional divisor and by $S_{\mathbb{F}},L_{0,\mathbb{F}},L_{1,\mathbb{F}}$ the strict transforms of $S_X,L_{0,X},L_{1,X}$. We now have the classical Mori fiber space $f:\mathbb{F}_2\to \pp^1$, of which $S_{\mathbb{F}}$ is a fiber, which has no multiple fibers, and of which $E,L_{0,\mathbb{F}},$ and $L_{1,\mathbb{F}}$ are sections. By Lemma \ref{lem:Nori-reg-fun} applied at the intersection point $E\cap L_{1,\mathbb{F}}$, the orbifold fundamental group $\pi_1^{\rm orb}(\mathbb{F}_2,S_{\mathbb{F}}+E+\frac{2}{3}L_{0,\mathbb{F}}+\frac{1}{2}L_{1,\mathbb{F}})$ is a quotient of 
$$\pi_1^{\rm reg}\left(\cc^2,\frac{1}{2}\cc\times \{0\} + \{0\}\times \cc;(0,0)\right)\simeq \zz/2\zz\times\zz,$$
thus it is abelian of rank at most two.

\end{proof}

\subsection{When $L$ has one component and the sporadic pair is of type $(2,3,5)$}

This special case has a particularly simple combinatorics, hence we handle it first.

\begin{proposition}\label{prop:1comp-sporadic235}
Let $(Z,S+L)$ be a plt Fano surface pair, where
$S$ is a curve in $Z$, $L$ is a divisor with exactly one component and standard coeficients smaller than one, and $Z$ has Picard rank $1$. Assume moreover 
that the curve pair obtained by adjunction of $(Z,S+L)$ to $S$ is of type
$$(S,\Delta_S)\cong \left(\pp^1,\frac{1}{2}\{0\}+\frac{2}{3}\{1\}+\frac{4}{5}\{\infty\}\right).$$
Then the group $\pi_1^{\rm orb}(Z,S+L)$ is abelian of rank at most two.
\end{proposition}

\begin{proof} 
Let $L_0$ denote the unique component of $L$.
Spelling out the adjunction formula (Lemma \ref{lem:coeff-under-adj}), we can classify the configurations of singular points of $Z$ and intersection points with $L_0$ on the smooth rational curve $S$. Let us denote by $T_k$ any singular point of orbifold index $k$ in $Z$. For instance, a point denoted by $T_2$ is an $A_1$ singularity, and a point denoted by $T_3$ is either an $A_2$ singularity, or a non-Gorenstein quotient singularity modelled after $\cc^2/\langle{\rm diag}(\zeta_3,\zeta_3)\rangle$. With this terminology, there are three possible configurations:
    
\begin{center} 
\includegraphics[scale=0.3]{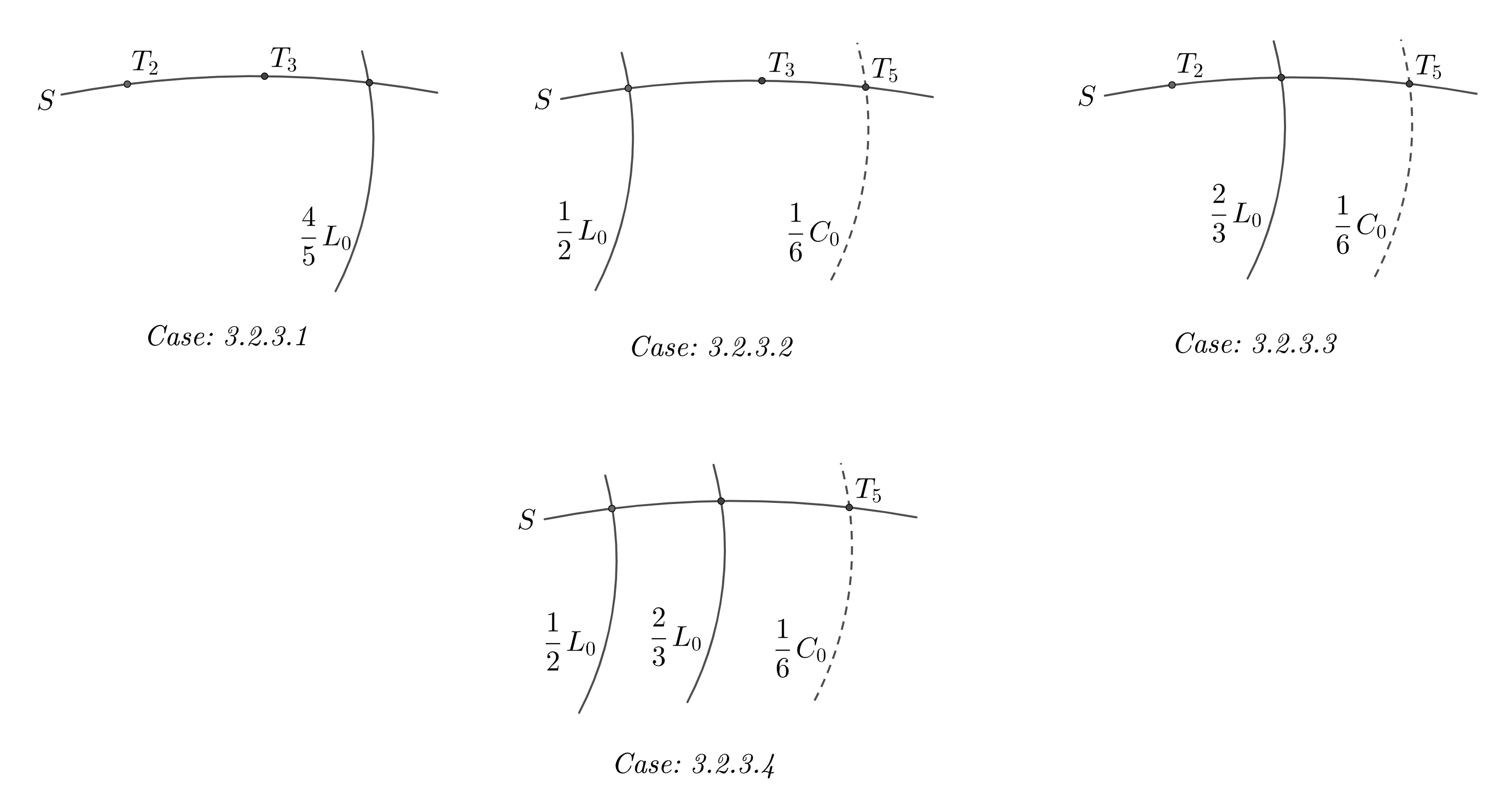}
\end{center} 

In {\it Cases 3.1} and {\it 3.2}, we note that the pair $(S,\Delta_S)$ admits a $6$-complement (by increasing the coefficient of the unique $T_5$-singular point of $Z$ that sits on $S$ from $\frac{4}{5}$ to $\frac{5}{6}$).
By Lemma \ref{lem:lifting-complements}, this complement lifts to $Z$, and we obtain a $6$-complement of the form $(Z,S+\frac{2}{3}L_0+\frac{1}{6}C_0)$ in {\it Case 3.1} and $(Z,S+\frac{1}{2}L_0+\frac{1}{6}C_0)$ in {\it Case 3.2}, where, in both cases; $C_0$ is a curve that intersects $S$ at its unique $T_5$-singular point transversally. Since $Z$ has Picard rank $1$, both $L_0$ and $L_0-\frac{5}{3}C_0$ are ample divisors. Hence we have, for small enough $\varepsilon> 0$, a log canonical Fano pair $(Z,S+\varepsilon L_0 + C_0)$. By Proposition~\ref{prop:toric-characterization}, this shows that $Z$ is a normal toric surface with the curves $S$ and $C_0$ contained in its toric boundary divisor of $Z$. Since $L_0$ intersects $S$ transversally at a unique smooth point, Lemma~\ref{lem:toric-fund-group-1} concludes.

In {\it Case 3.3}, let us first focus on the plt Fano pair $(Z,S+\frac{3}{4}L_0)$. Performing adjunction of this pair to $S$, we obtain a pair $(S,B_S)\cong (\pp^1,\frac{1}{2}\{0\}+\frac{2}{3}\{1\}+\frac{3}{4}\{\infty\})$, which admits a $4$-complement (by increasing $\frac{2}{3}$ to $\frac{3}{4}$). By Lemma \ref{lem:lifting-complements}, this complement lifts to $Z$, and yields a $4$-complement of the form $(Z,S+\frac{3}{4}L_0+\frac{1}{4}C_0)$ where $C_0$ is a curve intersecting $S$ transversally at the unique $T_3$-singular point that it contains. Since $L_0-C_0$ is ample on $Z$, the log canonical pair $(Z,S+C_0)$ is again Fano. By Proposition \ref{prop:toric-characterization}, this means that $Z$ is a toric surface, while $S$ and $C_0$ are components of its toric boundary divisor. Since $L_0$ intersects $S$ transversally at a unique smooth point, Lemma~\ref{lem:toric-fund-group-1} concludes.
\end{proof}

\subsection{When $L$ has one component and the sporadic pair obtained is of type $(2,3,3)$}

Although toric arguments play an important role in its study, the next special case cannot be fully understood from a toric perspective. The main result of this subsection is the following.

\begin{proposition}\label{prop:1comp-sporadic233}
Let $(Z,S+L)$ be a plt Fano surface pair, where
$S$ is a curve in $Z$, $L$ is a divisor with exactly one component and standard coefficients smaller than one, and $Z$ has Picard rank $1$. Assume moreover 
that the curve pair obtained by adjunction of $(Z,S+L)$ to $S$ is of type
$$(S,\Delta_S)\cong \left(\pp^1,\frac{1}{2}\{0\}+\frac{2}{3}\{1\}+\frac{2}{3}\{\infty\}\right).$$
Then the group $\pi_1^{\rm orb}(Z,S+L)$ is residually finite.
\end{proposition}

Prior to proving this proposition, we introduce and prove a key lemma to deal with a particularly tedious, non-toric situation.

\begin{lemma}\label{lem:t3-t3}
Let $(Z,S+\frac{1}{2}C)$ be a plt Fano surface pair, where
$S$ and $C$ are curves in $Z$ and $Z$ has Picard rank $1$. 
Assume that $S\cap Z_{\rm sing}$ consists in two points, $x$ and $y$, both of $T_3$-type, and that $S\cap C$ consists in exactly one point $z$, distinct from $x$ and $y$, at which $S$ and $C$ intersect transversally.
Then the group $\pi_1^{\rm orb}(Z,S+\frac{1}{2}C)$ is residually finite.
\end{lemma}

\begin{proof}
We apply the adjunction formula (that is, Lemma \ref{lem:coeff-under-adj}) for the pair $(Z,S)$ with respect to the curve $S$. It  provides the curve pair $(S,\frac{2}{3}\{x\}+\frac{2}{3}\{y\})$, which admits a $1$-complement $(\pp^1,\{0\}+\{\infty\})$. By Lemma~\ref{lem:lifting-complements}, it extends to a $1$-complement $(Z,S+\Gamma)$. 
If $\Gamma$ has two components $\Gamma_0$ and $\Gamma_1$, then Proposition~\ref{prop:toric-characterization} shows that $Z$ is a normal toric surface, and $S+\Gamma$ is exactly the toric boundary divisor. Then, applying Lemma~\ref{lem:toric-fund-group-1} to the pair $(Z,S+\Gamma_0+C)$ concludes.

We now assume that $\Gamma$ has exactly one reduced irreducible component.  Considering that the dual complex $\mathcal{D}(Z,S+\Gamma)$ contains a loop, by Remark \ref{rem:dual-complex}, it is piecewise linearly homeomorphic to the manifold $S^1$, thus $\Gamma$ contains no further log canonical center. Hence, the pair $(Z,S+\Gamma)$ is plt, and by adjunction to $\Gamma$, the surface $Z$ is smooth along $S\cup \Gamma \setminus\{x,y\}$. Since $K_Z+S+\Gamma\sim 0$, every point in $Z\setminus \{x,y\}$ is klt and Gorenstein, hence canonical. As for the points $x,y\in Z_{\rm sing}$, note that they are quotient surface singularities of orbifold index $3$: They are each locally isomorphic to the canonical singularity $A_2$, or to the singularity $\mathbb{C}^2/\langle{\rm diag}(\zeta_3,\zeta_3)\rangle$, which we denote by $C_3$.



\medskip
\noindent\underline{\textit{Case 1:}} The points $x,y$ are both of type $A_2$.

\noindent In that last case, $X$ is a Gorenstein del Pezzo surface of Picard rank $1$. By the classification result by Miyanishi-Zhang~\cite[Lemma 6]{MZ88}, there is a finite quasi-\'etale Galois cover $p:\pp^2 \rightarrow X$. Let $S_{\pp}:=p^*S$ and $C_{\mathbb P}:=p^*C$, so that $p$ is compatible with the plt Fano pairs $(\pp^2,S_{\mathbb P}+\frac{1}{2}C_{\mathbb P})$ and $(Z,S+\frac{1}{2}C)$. 
Note that $C_{\pp}\in|\mathcal{O}_{\pp^2}(c)|$ for some $c\in\{1,2,3\}$. We take $B$ to be the sum of $4-c$ lines in $\pp^2$, chosen general enough that the Calabi--Yau pair $(\pp^2,S_{\pp}+\frac{1}{2}C_{\pp}+\frac{1}{2}B_{\pp})$ is log canonical. Then, Proposition \ref{prop:lc-cy} and Lemma \ref{lem:fun-under-surj} conclude.

\medskip 

\noindent\underline{\textit{Case 2:}} The point $x$ is of type $A_2$, and the point $y$ is of type $C_3$.

\noindent We take the minimal resolution $\varepsilon\colon X\rightarrow Z$ of the singular point $y$, and denote its exceptional divisor by $E$. Let $S_X,\Gamma_X,C_X$ denote the strict transforms of $S,\Gamma,C$. Note that $X$ is a normal klt Gorenstein surface of Picard rank $2$, thus has canonical singularities, and that it is smooth along $E$. Since $K_X$ remains anti-effective, a single step of the $K_X$-MMP provides an extremal birational contraction or a Mori fiber space $f:X\to B$.

First, assume that this MMP provides an extremal birational contraction $\mu: X\to Y$. The surface $Y$ has canonical singularities and Picard rank $1$. Denote by $E'$ the curve that $\mu$ contracts. Since $E$ has square $-3$, is contained in the smooth locus of $X$, and $Y$ has canonical singularities, $E$ cannot be $E'$. We also check that $E'$ is neither $S_X$ not $\Gamma_X$: By contradiction, assume that $E'$ is one of these two curves, then $E'$ contains exactly one singular point of $X$, that is of type $A_2$. 
Since $E'$ is either $S_X$ or $\Gamma_X$, we also have 
$$K_X\cdot E' = K_Z\cdot\varepsilon_* E' - \frac{1}{3}E\cdot E' < 0,$$ 
so the discrepancy of $\mu$ along $E'$ is strictly positive, i.e., $\mu$ contracts $E'$ to a smooth point of the surface $Y$. Thus, $K_X=\mu^*K_Y + E'$, and $K_X\cdot E' = E'^2$ is a negative integer. By adjunction of the plt pair $(Y,E')$ to $E'$, we obtain a pair $(E',\Delta_{E'})$ with a negative, even integer as its degree. So $E'\simeq \pp^1$ and $\Delta_{E'}\sim 0$. Yet, $E'$ contains a singular point of $Y$ (of $A_2$-type), which must have a non-zero contribution to $\Delta_{E'}$, a contradiction to Lemma \ref{lem:coeff-under-adj}.

So the curve $E'$ contracted by $\mu$ is not contained in the support of $S_X+\Gamma_X+E$. Since $\rho(Y)=1$, the image of $E$ in $Y$ has positive self-intersection, hence $E'$ and $E$ intersect. Since $X$ is smooth along $E$, by Lemma~\ref{lem:Gor-MMP}, the curve $E'$ 
intersects $E$ at exactly one point, and is disjoint from both $S_X$ and $\Gamma_X$. Let $S_Y,\Gamma_Y,E_Y,$ and $C_Y$ be the images of $S_X,\Gamma_X,E,$ and $C_X$ by $\mu$.
By Proposition~\ref{prop:toric-characterization}, the log canonical Calabi--Yau pair $(Y,S_Y+\Gamma_Y+E_Y)$ is toric.
Since $E'$ and $S_X$ are disjoint, the curves $S_Y$ and $C_Y$ still intersect transversally at a single smooth point of $Y$, that is not contained in $E_Y$.
Hence, Lemma~\ref{lem:toric-fund-group-1} shows that $\pi_1^{\rm orb}(Y,S_Y+E_Y+\frac{1}{2}C_Y)$ is abelian of rank at most two, and so is $\pi_1^{\rm orb}(Z,S+\frac{1}{2}C)$.

Assume now that the MMP provides a Mori fiber space $f\colon X \rightarrow B$ instead. It has irreducible fibers, and since $C_X$ and $E$ have non-zero self-intersection, they are $f$-horizontal. Let $F$ be the general fiber of $f$.
Since $(X,S_X+\Gamma_X+E)$ is a $1$-complement and $F$ is Cartier, we notice two possibilities: Either $E\cdot F = 2$, and $S_X$ and $\Gamma_X$ are fibers of $f$, or $E$ and one of the two curves $S_X,\Gamma_X$ are sections of $f$, the remaining curve in $S_X,\Gamma_X$ being a fiber. Since $S_X$ and $\Gamma_X$ still intersect at an $A_2$ singular point, the latter is the only option, and in particular $E$ is a section of $f$. 
Since $\varepsilon^*(K_Z+S+\frac{1}{2}C) = K_X + S_X + \frac{2}{3} E + \frac{1}{2}C_X$
is anti-nef, we see that $S_X\cdot F<\frac{5}{6}$, thus $S_X\cdot F = 0$. Therefore, $S_X$ is a fiber of $f$ and $\Gamma_X$ a section. 

We want to apply Lemma \ref{lem:Nori-reg-fun} to the pair $(X,S_X+E+\frac{1}{2}C_X)$ at the intersection point $u$ of $S_X$ and $C_X$. Since it is a smooth point of transversal intersection, we have
$\pi_1^{\rm reg}(X,S_X+E+\frac{1}{2}C_X;u)\simeq\mathbb{Z}/2\mathbb{Z}\times\mathbb{Z}$. Note that $C_X$ and $E$ are sections of $f$, that the base curve $B$ is rational, and that, by the canonical bundle formula, $f: X\setminus S_X\to B\setminus\{\rm pt\}$ has zero, one, or two multiple fibers (and if two, then both with multiplicity two). Thus, we can apply Lemma \ref{lem:Nori-reg-fun} and see that $\pi_1^{\rm orb}(X,S_X+E+\frac{1}{2}C_X)$ admits a subgroup of index at most two (and thus, normal) isomorphic to a quotient of $\mathbb{Z}/2\mathbb{Z}\times\mathbb{Z}$.

\medskip 

\noindent\underline{\textit{Case 3:}} The points $x,y$ are both of type $C_3$.\\

Let $\varepsilon\colon X\rightarrow Z$ be the minimal resolution of the points $x$ and $y$, with corresponding exceptional divisors $E_1$ and $E_2$, and denote by $S_X,\Gamma_X,C_X$ the strict transforms of $S,\Gamma,C$. 
Note that $X$ is a klt Gorenstein (hence canonical) surface of Picard rank three, and smooth along the cycle of rational curves $S_X+\Gamma_X+E_1+E_2$.
As $K_X$ remains anti-effective, running a $K_X$-MMP terminates after one or two steps, with a Mori fiber space to a curve or a point. 
Since $-K_X = S_X+\Gamma_X+E_1+E_2$ is not nef (see, e.g., that $-K_X\cdot E_1 = -1$), a $K_X$-MMP does require two steps in order to terminate. The first step must be a divisorial contraction: We denote it by $\mu : X\to Y$, and let $E'$ denote its exceptional divisor.

Let us first prove that $E'$ is none of the four curves $S_X,\Gamma_X,E_1,E_2$. We argue by contradiction: If $E'$ is one of them, smoothness of $X$ along the cycle $S_X+\Gamma_X+E_1+E_2$ shows that $E'$ is a $(-1)$-curve, thus $S_X$ or $\Gamma_X$. Hence, the images $E_{1,Y},E_{2,Y}$ of $E_1,E_2$ in the surface $Y$ of Picard rank two are two curves of square $-2$: They span the Mori cone $\overline{NE}(Y)$. The next step of a $K_Y$-MMP must contract one of them, but they are $(-2)$-curves, thus have zero intersection with $K_Y$, a contradiction.

Let us discuss the geometry of $Y$ some more. By Lemma \ref{lem:Gor-MMP}, $E'$ intersects $S_X+\Gamma_X+E_{1}+E_{2}$ in exactly one point $p$. In particular, $E'$ is disjoint from at least one of the two (themselves disjoint) curves $E_1$ and $E_2$, say $E_i$. The discrepancy for the exceptional divisor $E'$ equals $\frac{K_X\cdot E'}{(E')^2}>0$, so $E'$ is contracted to a terminal singularity, i.e., a smooth point of the intermediate surface $Y$. Therefore, $Y$ is still smooth along the image $\mu_*(S_X+\Gamma_X+E_{1}+E_{2}) =: S_Y + \Gamma_Y + E_{1,Y}+E_{2,Y}$.

Arguing by contradiction, we show that $Y$ must be a Mori fiber space over a curve. Indeed, assume that a $K_Y$-MMP can yield a divisorial contraction. Its target is a Fano surface of Picard rank $1$.
\begin{itemize}
    \item If none of the four curves $S_Y,\Gamma_Y,E_{1,Y},E_{2,Y}$ is contracted by it, we can push-forward the pair on $Y$ to a log canonical pair on the final Fano surface. There, the presence of four components of coefficient one presents a contradiction with Proposition \ref{prop:toric-characterization}.
\item If one of the curves $S_Y,\Gamma_Y, E_{1,Y},E_{2,Y}$ is contracted by a divisorial contraction, by smoothness of $Y$ along the cycle, it has to be a $(-1)$-curve. Recall that $E_i$ is disjoint from $E'$, so $E_{i,Y}$ has square $-3$: We may not contract it. Neither can we contract the other $E_{j,Y}$: Indeed, the image of $E_{i,Y}$ on the final Fano surface of Picard rank $1$ would still have square $-3$. The only option left is to contract either $S_Y$ or $\Gamma_Y$. However, since $S_Y\cdot E_{i,Y} = \Gamma_Y\cdot E_{i,Y} = 1$, the image of $E_{i,Y}$ by the contraction of one of these two curves has square $-2$ on the final Fano surface of Picard number $1$. This is a contradiction.
\end{itemize}

This discussion shows that a $K_Y$-MMP's next step is a Mori fiber space $f\colon Y\rightarrow B$ to a rational curve. It further shows, in fact, that the Mori cone $\overline{\mathrm{NE}}(Y)$ is spanned by the general fiber $F$ of $f$ and the afore-mentioned curve $E_{i,Y}$ of square $-3$. Therefore, the other $E_{j,Y}$ cannot be of square $-3$: This means that the afore-mentioned unique point $p$ where $E'$ intersects $S_X+\Gamma_X+E_1+E_2$ belongs to $E_2$. By Lemma \ref{lem:Gor-MMP} again, we obtain that $E'$ is disjoint from $S_Y+\Gamma_Y$. Thus after contracting $E'$, the curves $S_Y$ and $C_Y$ still intersect transversally at exactly one smooth point in the surface $Y$.
Since $E_{i,Y}\cdot F\ge 1$ and $E_{j,Y}$ is disjoint from $E_{i,Y}$, we have $E_{j,Y}\cdot F\ge 1$ as well. Since it also holds
$(S_Y+\Gamma_Y+E_{1,Y}+E_{2,Y})\cdot F = 2$, we notice that $E_{1,Y},E_{2,Y}$ are sections of $f$, whilst $S_Y,\Gamma_Y$ are fibers. They are reduced irreducible fibers in fact, given that they both intersect our section $E_{1,Y}$ transversally in one smooth point. Finally, as $C_Y\cdot S_Y = 1$, it is a section of $f$ as well. 

We intend to apply Lemma~\ref{lem:Nori-reg-fun} to the pair $(Y,S_Y+E_{1,Y}+E_{2,Y}+\frac{1}{2}C_Y)$ at an intersection point $u$ of $C_Y$ and $E_{2,Y}$ (which exists since the two curves are numerically proportional to $E_{1,Y}+3F$, of positive square).
For that, the assumptions are clearly verified, and we are left to note that the group $$\pi_1^{\rm reg}\left(Y,S_Y+E_{1,Y}+E_{2,Y}+\frac{1}{2}C_Y;u\right)\simeq \pi_1^{\rm reg}\left(Y,E_{2,Y}+\frac{1}{2}C_Y;u\right)$$ 
is isomorphic to $\zz/2\zz\times\zz$, thus virtually abelian.
\end{proof}

We can now prove Proposition \ref{prop:1comp-sporadic233}.

\begin{proof}[Proof of Proposition \ref{prop:1comp-sporadic233}]
    Spelling out the adjunction formula (Lemma \ref{lem:coeff-under-adj}), we can classify the configurations of singular points of $Z$ and intersection points with $L_0$ on the smooth rational curve $S$. 
    There are three possible configurations:

\begin{center} 
\includegraphics[scale=0.3]{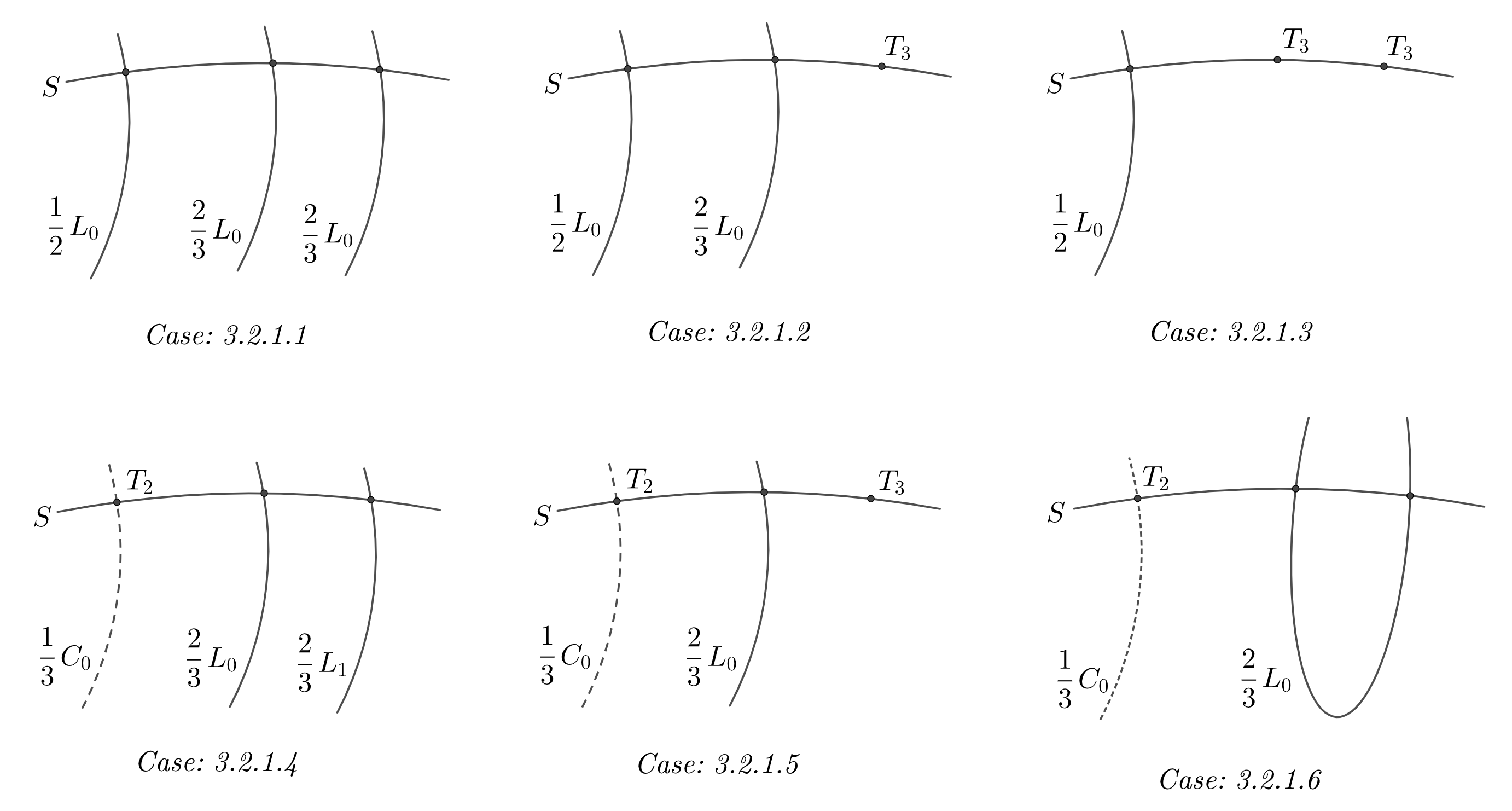}
\end{center}

In {\it Case 1.1}, Lemma~\ref{lem:t3-t3} concludes.

In {\it Cases 1.2 and 1.3}, we note that the pair $(S,\Delta_S)$ admits a $3$-complement (increase the coefficient $\frac{1}{2}$ to $\frac{2}{3}$), which by Lemma~\ref{lem:lifting-complements} lifts to a $3$-complement $(Z,S+\frac{m}{3}C_0+\frac{2}{3}L_0)$ with $m\in\{1,2,3\}$ and $C_0$ a curve that only intersects $S$ at the unique point of singular type $T_2$ on $S$. Checking the contribution of $C_0$ with the adjunction formula to $S$ (Lemma~\ref{lem:coeff-under-adj}), we see that $m=1$ and that $C_0$ intersects $S$ transversally. So the pair $(Z,S+C_0)$ is also log canonical. Since $L_0-C_0$ has positive intersection with $S$ and $\rho(Z)=1$, the divisor $L_0-C_0$ is ample, and so the pair $(Z,S+C_0)$ is also a Fano pair. Applying Proposition~\ref{prop:toric-characterization} to the log canonical Fano pair $(Z,S+C_0+\varepsilon L_0)$ for small enough $\varepsilon > 0$, we obtain that $Z$ is a toric surface, and $S$ and $C_0$ are contained in its toric boundary divisor. 

In {\it Case 1.2}, Lemmas~\ref{lem:toric-fund-group-1} and \ref{lem:fun-under-surj} yield the residual finiteness of $\pi_1^{\rm orb}(Z,S+L)$.

In {\it Case 1.3}, we take the standard $Z - S$ toric cover $p:X\to Z$, and use Remark \ref{rem:standardtoric} to identify $X$ as $\pp(1,1,2)$. We set $S_X:=p^*S$ and $L_{0,X} := p^*L_0$. Remark \ref{rem:standardtoric} further shows that $S_X$ and $L_{0,X}$ still intersect transversally at exactly two smooth points of $X$, and that $S_X$ contains the sole singular point of $X\simeq\pp(1,1,2)$, the coordinate point $[0:0:1]$. The pair $(X,S_X+\frac{2}{3}L_{0,X})$ remains a plt Fano pair, thus in the class group of $X$, we have two possibilities:
\begin{itemize}
    \item either $S_X\sim \mathcal{O}_{\pp}(1)$ and $L_{0,X}\sim \mathcal{O}_{\pp}(4)$;
    \item or $S_X\sim \mathcal{O}_{\pp}(2)$ and $L_{0,X}\sim \mathcal{O}_{\pp}(2)$.
\end{itemize}
(We also use the fact that $L_{0,X}$ does not contain $[0:0:1]$, thus has even degree.) 
We can rule out Case (ii): Since $S_X$ contains $[0:0:1]$, if it has degree two, its equation must be a linear combination of the monomials $x^2,xy,y^2$ in the variables of weight one. This contradicts the irreducibility of $S_X$.
Thus we are in Case (i). Blowing-up the point $[0:0:1]$, we obtain $\varepsilon:\mathbb{F}_2\to X$ with exceptional divisor $E$, and we denote by $S_{\mathbb{F}},L_{0;\mathbb{F}}$ the relevant strict transforms. The Hirzebruch surface $\mathbb{F}_2$ has a natural Mori fiber space structure $f:\mathbb{F}_2\to \pp^1$ with reduced fibers, of which $S_{\mathbb{F}}$ is a fiber, $E$ a section, and $L_{0;\mathbb{F}}$ a bisection. We apply Lemma \ref{lem:Nori-reg-fun} at a ramification point $u$ of the double cover of $L_{0,\mathbb{F}}$ to $\pp^1$ induced by $f$, and see that
$$\pi_1^{\rm reg}(\mathbb{F}_2, L_{0,\mathbb{F}} ; u)\simeq \pi_1^{\rm orb}(\mathbb{F}_2, S_{\mathbb{F}} + E + L_{0,\mathbb{F}}),$$
thus it is a cyclic group, and so $\pi_1^{\rm orb}(Z,S+L)$ is residually finite.
\end{proof}

\subsection{When $L$ has one component and the sporadic pair obtained is of type $(2,3,4)$}

The main result of this subsection is the following.

\begin{proposition}\label{prop:1comp-sporadic234}
Let $(Z,S+L)$ be a plt Fano surface pair, where
$S$ is a curve in $Z$, $L$ is a divisor with exactly one component and standard coefficients smaller than one, and $Z$ has Picard rank $1$. Assume moreover 
that the curve pair obtained by adjunction of $(Z,S+L)$ to $S$ is of type
$$(S,\Delta_S)\cong \left(\pp^1,\frac{1}{2}\{0\}+\frac{2}{3}\{1\}+\frac{3}{4}\{\infty\}\right).$$
Then the group $\pi_1^{\rm orb}(Z,S+L)$ is residually finite.
\end{proposition}

We prepare the proof of this proposition with some lemmas. The next lemma focuses on a very specific toric situation.

\begin{lemma}\label{lem:toric-fund-group-4}
Let $(Z,S+\frac{1}{2}C)$ be a plt Fano surface pair, where
$Z$ is a normal toric variety of Picard rank $1$, $S$ and $C$ are curves in $Z$, and $S$ is a component of the toric boundary. Assume that $S$ contains exactly two singular points of $Z$, $z_2$ and $z_3$, of respective orbifold indices $2$ and $3$, and that $C$ intersects $S$ at exactly two points, namely $z_2$ and a smooth point $z$, transversally at both.
Then the group  
$\pi_1^{\rm orb}(Z,S+\frac{1}{2}C)$ is virtually cyclic.
\end{lemma}

\begin{proof}
We take the standard $Z-S$ toric cover $p:X\to Z$ and note that by Remark \ref{rem:standardtoric}, $X$ is a normal toric surface of Picard rank $1$, with exactly two singular points $x_2$ and $x_3$ of respective orbifold indices $2$ and $3$, both lying on the toric boundary curve $S_X$. This identifies $X$ as the weighted projective space $\pp(1,2,3)$. 

We set $C_X:=p^*C$ and note that $p$ is compatible with the pairs $(X,S_X+\frac{1}{2}C_X)$ and $(Z,S+\frac{1}{2}C)$. We claim that $p$ is totally ramified along the curve $S_X$.
Indeed, we have a restriction $p|_{S_X}:S_X\to S$, which remains a toric finite Galois cover. We want to show that this restriction is an isomorphism. In terms of toric fans, this toric cover corresponds to the following operation: While $\Sigma_Z=\Sigma_X$ is the fan spanned by the vectors $\{v_1:=(1,0),v_2:=(0,1),v_3:=(-2,-3)\}$, we have a lattice $N_X = (1,0)\cdot \zz\oplus (0,1)\cdot \zz\subset N_Z \subset \rr^2$. In this set-up, the Galois group of $p|_{S_X}$ is isomorphic to the quotient 
$N_Z\cap v_3\cdot \rr/v_3\cdot \zz$
of the two available lattices on the ray spanned by $v_3$.
Let $n_i\in\nn$ be the largest positive integer such that $\frac{1}{n_i}v_i\in N_Z$, then the restriction $p|_{S_X}$ is an isomorphism if and only if $n_3=1$. In the curve $S$ on the toric surface $Z$, there are still exactly two singular points, of orbifold indices $2$ and $3$. This means that $\{v_1,v_3\}$ span $3N_Z$, and that $\{v_2,v_3\}$ span $2N_Z$. In particular, $\frac{3}{n_3}v_3 \in v_1\cdot\zz +v_3\cdot\zz$, which yields (looking at the second coordinate) $\frac{3}{n_3}\in\zz$. Similarly with $v_2$, we obtain that $n_3$ divides $2$. Since $2$ and $3$ are coprime, we conclude that $n_3=1$, as wished.
By the projection formula, we still have $S_X\cdot C_X =\frac{3}{2}$, and $S_X$ and $C_X$ intersect transversally at one smooth point of $X$ and at the singular point $x_2$. In particular, since $X$ has Picard rank $1$, the pre-image $C_X$ remains a single reduced and irreducible curve.



Let $Y$ be the normal toric surface of Picard rank $2$ of fan $\Sigma_Y$ spanned by $(1,0),(0,1),(-2,-3),(0,-1)$, obtained by a toric blow-up $\varepsilon:Y\to X$ with exceptional divisor $E$. Clearly, $Y$ has two singular points of type $A_1$: One lies at the intersection of the prime exceptional divisor $E$ and the strict transform $S_Y$ of $S_X$, the other at the intersection of $S_Y$ and of the strict trasnform $C_Y$ of $C_X$. From the fan of $Y$, we see that $(S_Y)^2=0$ and $-K_Y\cdot S_Y > 0$, so we have a Mori fiber space $f:Y\to\mathbb{P}^1$ contracting $S_Y$. Its fibers are the members of the pencil spanned by the Cartier divisor $2S_Y$ and by the strict transform $\ell_Y$ of the torus-invariant curve $\{y=0\}\subset \pp(1,2,3)$: They are all smooth, reduced and irreducible, except for $2S_Y$ itself; hence the induced curve pair on the base of $f$ is isomorphic to $(\pp^1,\{0\})$, which has trivial orbifold fundamental group.
Let $F$ be the general fiber of $f$ and let $C_Y$ be the strict transform of $C_X$. 
The center of the blow-up $\varepsilon$ avoids the curve $C_X$ altogether, so $C_Y\cdot F = 2\,  C_Y\cdot S_Y = 3$, i.e., $C_Y$ is a trisection of $f$. On the other hand, $E$ and $S_Y$ intersect transversally at a unique singular point of type $A_1$, so $E\cdot F = 2\,  E\cdot S_Y = 1$, i.e., $E$ is a section of $f$. 

To conclude, it suffices to show that, $\Delta:=S_Y+E_Y+\frac{1}{2}C_Y$, the orbifold fundamental group
$\pi_1^{\rm orb}(Y,\Delta)$ has a normal subgroup of finite index that is cyclic. We proceed by an argument similar to the proof of Lemma \ref{lem:Nori-reg-fun}, and slightly more general. Let $y\in Y$ be the unique smooth point at which the curves $S_Y$ and $C_Y$ intersect transversally. Let $U$ be an analytic ball in $Y$ centered at $y$, disjoint from $E$ and in which the two pieces of curves $S_Y$ and $C_Y$ are biholomorphic to the pieces of the coordinate axes. In particular, we have
\begin{equation}\label{eq:pi1U}
\pi_1^{\rm orb}(U,\Delta|_{U})\simeq \zz\times\zz/2\zz.
\end{equation}
Note that the triple cover $f|_{C_Y}:C_Y\to \pp^1$ ramifies with index two at the point $y$.
Let $p\in\pp^1$ be a point in the analytic open set $f(U)$ such that $F_p\cap C_Y$ consists in three distinct points, with exactly two of them contained in the neighborhood $U$.


We use Lemma \ref{lem:nori} and the pushforwards by the natural inclusions to draw a commutative diagram
\[
\xymatrix{
\pi_1^{\rm orb}(F_p\cap U,\Delta|_{F_p\cap U})\ar[r]^-{r}\ar[d]^-{\pi} & 
\pi_1^{\rm orb}(U,\Delta|_{U}) \ar[d]\\
\pi_1^{\rm orb}(F_p,\Delta|_{F_p})\ar@{->>}[r]^-{s} &
\pi_1^{\rm orb}(Y,\Delta).
}
\]

We note that $\pi_1^{\rm orb}(F_p,\Delta|_{F_p})$ is generated by three involutions, coming from the loops $c_1,c_2,c_3$ along the three distinct intersection points of $C_Y$ with $F_p$. Indeed, the fourth loop that should potentially be taken into account, the one around the intersection point of $E$ with $F_p$, can be written $e=c_1c_2c_3$ in $\pi_1^{\rm orb}(F_p,\Delta|_{F_p})$ using the fact that $F_p\simeq \pp^1$. Note that two of the three loops, say $c_1$ and $c_2$, can be chosen inside the analytic $U$, thus have preimages $\tilde{c}_1$ and $\tilde{c}_2$ by the group homomorphism $\pi$.

So $\pi_1^{\rm orb}(Y,\Delta_Y)$ is generated by three generators $s(\pi(\tilde{c}_1)),s(\pi(\tilde{c}_2)),s(c_3)$, which are each of order dividing two. Consider the classes of loops $r(\tilde{c}_1),r(\tilde{c}_2)\in \pi_1^{\rm orb}(U,\Delta|_{U})$, and recall that, by Equation \eqref{eq:pi1U}, the group $\pi_1^{\rm orb}(U,\Delta|_{U})$ has exactly one non-trivial element of order $2$. Hence, either one of the classes $r(\tilde{c}_1)$ or $r(\tilde{c}_2)$ is trivial in this group, or these two classes coincide, i.e., $r(\tilde{c}_1)=r(\tilde{c}_2)$. In either case, the subgroup of $\pi_1^{\rm orb}(Y,\Delta_Y)$ spanned by the two elements $s(\pi(\tilde{c}_1))$ and $s(\pi(\tilde{c}_2))$ is in fact already spanned by one of them alone. Thus, the group $\pi_1^{\rm orb}(Y,\Delta_Y)$ is spanned by zero, one, or two involutions, thus has a normal subgroup of index at most two that is cyclic, as wished.
\end{proof}

\begin{lemma}\label{lem:t2-t4}
Let $(Z,S+\frac{2}{3}C)$ be a plt Fano surface, where  $S$ and $C$ are curves in $Z$ and $Z$ has Picard rank $1$. Assume that $S\cap Z_{\rm sing}$ consists in two points, $x$ and $y$, of respective orbifold indices $2$ and $4$, and that $S\cap C$ consists in exacty one point $z$, distinct from $x$ and $y$, at which $S$ and $C$ intersect transversally.
Then the group 
$\pi_1^{\rm orb}(Z,S+\frac{2}{3}C)$ is residually finite.
\end{lemma}

\begin{proof}
We start in a similar way to the proof of Lemma \ref{lem:t3-t3}. We apply the adjunction formula (that is, Lemma \ref{lem:coeff-under-adj}) for the pair $(Z,S)$ with respect to the curve $S$. It  provides the curve pair $(S,\frac{1}{2}\{x\}+\frac{3}{4}\{y\})$, which admits a $1$-complement $(\pp^1,\{0\}+\{\infty\})$. By Lemma~\ref{lem:lifting-complements}, it extends to a $1$-complement $(Z,S+\Gamma)$. 
If $\Gamma$ has two components $\Gamma_0$ and $\Gamma_1$, then Proposition~\ref{prop:toric-characterization} shows that $Z$ is a normal toric surface, and $S+\Gamma$ is exactly the toric boundary divisor. Then, applying Lemma~\ref{lem:toric-fund-group-1} to the pair $(Z,S+\Gamma_0+C)$ concludes.

We now assume that $\Gamma$ has exactly one reduced irreducible component. Considering that the dual complex $\mathcal{D}(Z,S+\Gamma)$ contains a loop, by Remark \ref{rem:dual-complex}, it is piecewise linearly homeomorphic to the manifold $S^1$, thus $\Gamma$ contains no further log canonical center. Hence, the pair $(Z,S+\Gamma)$ is plt, and by adjunction to $\Gamma$, the surface $Z$ is smooth along $S\cup \Gamma \setminus\{x,y\}$. Since $K_Z+S+\Gamma\sim 0$, every point in $Z\setminus \{x,y\}$ is klt and Gorenstein, hence canonical. As for $x$ and $y$: The point $x$ is an $A_1$ singularity, and the point $y$ of index $4$ is a quotient singularity by a group $G$ of order $4$ acting faithfully and freely in codimension $1$. In particular, any element of order two in $G$ must be ${\rm diag}(-1,-1)$, so $G$ is cyclic and the singular point is either an $A_3$ singularity, or modelled after $\mathbb{C}^2/\langle{\rm diag}(\zeta_4,\zeta_4)\rangle$, which we call a $C_4$ singularity.

\medskip

If $y$ is a $C_4$ singularity, then the proof of Case 2 in Lemma \ref{lem:t3-t3} applies almost {\it verbatim.}

\medskip

We can thus focus on the case where $y$ is an $A_3$ singularity. Then, the surface $Z$ has canonical Gorenstein singularities. Since $Z$ is also a del Pezzo surface of Picard number one, by \cite[Lemma 3]{MZ88} the number and type of its singular points is prescribed as $2A_1+A_3, A_1+2A_3$, or $2A_1+2A_3$. 

If the singularities of $Z$ are of one of the types $A_1+2A_3$ or $2A_1+2A_3$, then by \cite[Lemma 6, Table I]{MZ88}, there is a finite Galois quasi-\'etale cover $q:\pp^1\times \pp^1\to Z$. Taking $S':=q^*S$ and $C':=q^*C$, by Proposition \ref{prop:pushforward}, it suffices to show that $\pi_1^{\rm orb}(\pp^1\times\pp^1,S'+\frac{2}{3}C')$ is residually finite. Since $\pp^1\times\pp^1$ is a clear Mori fiber space to a curve, Proposition \ref{prop:mfs} concludes.

Otherwise, the singularities of $X$ are of type $2A_1+A_3$, and by \cite[Lemma 6, Table I]{MZ88}, there is a quasi-\'etale double cover $q\colon X\rightarrow Z$ such that $\rho(X)=1$ and $X$ has a single $A_1$ singularity.
Let $S_X:=q^*S$ and $C_X:=q^*C$, so that
$(X,S_X+\frac{2}{3}C_X)$ remains a plt Fano surface pair of coregularity $1$. The single $A_1$ point of $X$ is contained in $S_X$, and away from $C_X$.

We run a $(K_{X}+S_X+\frac{2}{3}C_X)$-MMP. 
If it terminates with a Mori fiber space, then Proposition~\ref{prop:mfs} concludes.
Assume that it terminates with a klt Fano surface $Y$ of Picard rank one, underlying a log canonical Fano pair $(Y,S_Y+\frac{2}{3}C_Y)$ where $S_Y$, the image of $S'$, is a curve in $Y$ (since $S_Y$ is an lc center and $Y$ is klt), whilst $C_Y$, the image of $C'$, is a either curve or zero in $Y$. Note that the pair $(Y,S_Y+\frac{2}{3}C_Y)$ has no other log canonical center than the obvious one along $S_Y$: It thus is a plt pair.
If $C_Y$ is zero, we perform adjunction of $(Y,S_Y)$ to $S_Y$, provide a $1$-complement on $S_Y$ and lift it to $Y$ by Lemma \ref{lem:lifting-complements}, and conclude by Proposition \ref{prop:lc-cy}.
If $C_Y$ is non-zero, Proposition \ref{prop:1comp-sporadic233} concludes.
\end{proof}

We now prove Proposition \ref{prop:1comp-sporadic234}.

\begin{proof}[Proof of Proposition \ref{prop:1comp-sporadic234}]
Spelling out the adjunction formula (Lemma \ref{lem:coeff-under-adj}) of the pair $(Z,S+L)$ to $S$ allows to list the finitely many possible configurations:

\begin{center} 
\includegraphics[scale=0.3]{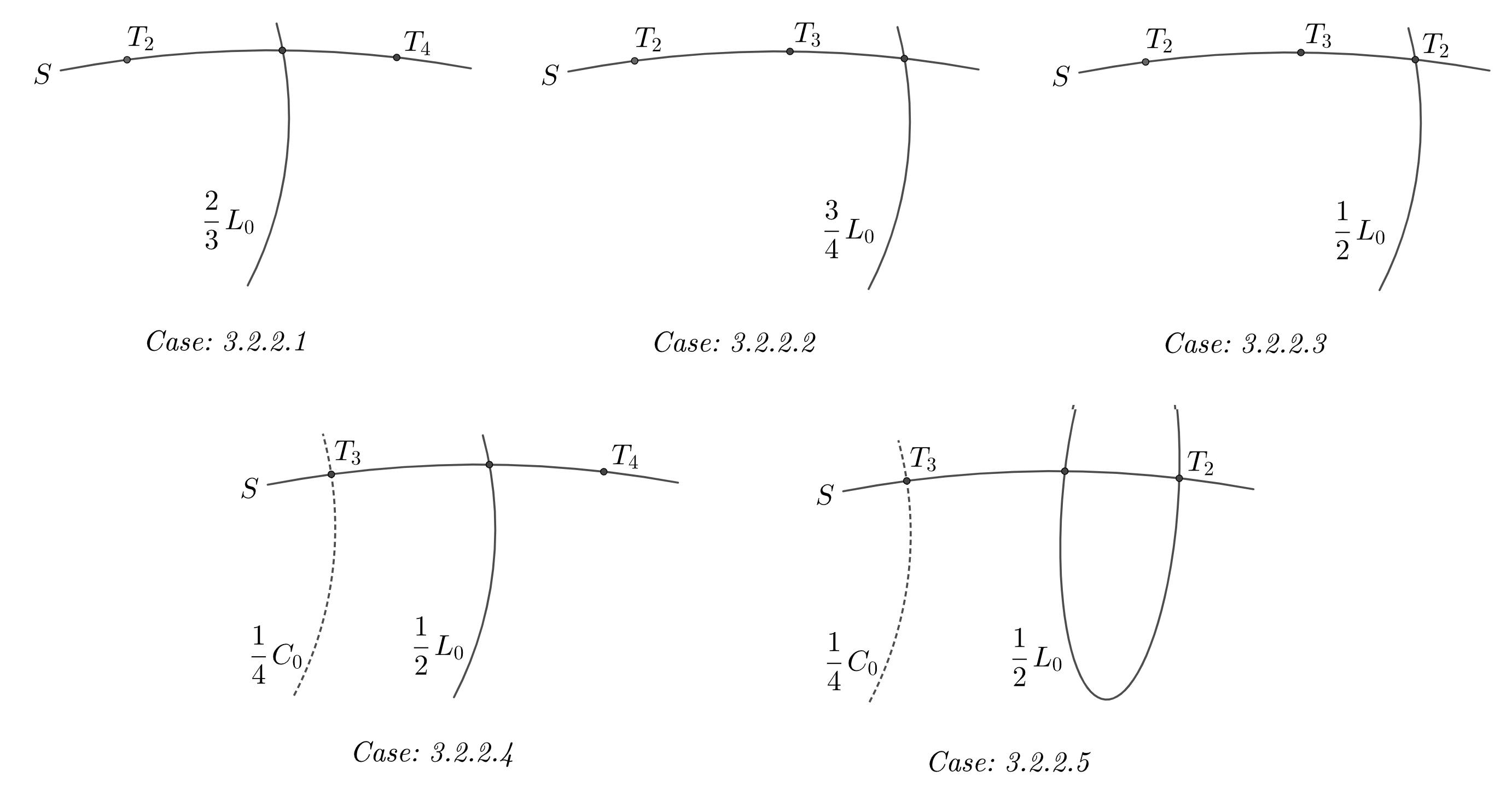}
\end{center} 

First of all, {\it Case 2.4} is handled by Lemma \ref{lem:t2-t4} directly.

In {\it Cases 2.1, 2.2, 2.3,} and {\it 2.5}, we use the fact that the pair $(S,\Delta_S)$ admits a $4$-complement (increasing the coefficient $\frac{2}{3}$ to $\frac{3}{4}$). By Lemma \ref{lem:lifting-complements}, it extends to a $4$-complement $(Z,S+L+\frac{1}{4}C_0)$, where $C_0$ is a curve that intersects $S$ transversally at its unique point of singular type $T_3$. Note that in {\it Cases 2.1, 2.3,} and {\it 2.5}, the pair $(Z,S+\frac{1}{2}L_0+\frac{1}{4}C_0)$ is Calabi--Yau, while in {\it Case 2.2}, it is Fano. Also note that $L_0-\frac{3}{2}C_0$ is ample in {\it Cases 2.1, 2.2, 2.3} and numerically trivial in {\it Case 2.5}.

Let us focus on {\it Cases 2.1, 2.2, 2.3} for a moment. There, we just showed that for small enough $\varepsilon>0$ the pair $(Z,S+\varepsilon L_0+C_0)$ is again a log canonical Fano pair. Thus by Proposition \ref{prop:toric-characterization}, the surface $Z$ is toric, and the curves $S$ and $C_0$ are in its toric boundary divisor. Applying Lemma~\ref{lem:toric-fund-group-1} concludes in {\it Cases 2.1} and {\it 2.2}. In {\it Case 2.3}, Lemma \ref{lem:toric-fund-group-4} concludes.

In {\it Case 2.5}, note that $(Z,S+C_0)$ is a log canonical Calabi--Yau pair of coregularity zero. By Lemma \ref{lem:coreg-vs-complementCY}, it is a $2$-complement. It cannot be a $1$-complement though, as $K_Z+S+C_0$ is not a trivial Cartier divisor locally near the singular points of type $T_2$ contained in $S$ and not in $C_0$.  By Lemma \ref{lem:index-one-cover}, we take the double cover $p:X\to Z$ associated to the torsion divisor class $K_Z+S+C_0$. This finite cyclic cover is quasi\'etale, and denoting $S_X=q^*S$ and $L_{0,X}=q^*L_0$, it is compatible with the log canonical Fano pairs $(X,S_X+\frac{1}{2}L_X)$ and $(Z,S+L)$. By Proposition \ref{prop:galois1}, it suffices to show that $\pi_1^{\rm orb}(X,S_X+\frac{1}{2}L_{0,X})$ is residually finite to conclude. We note that $S_X$ and $L_{0,X}$ intersect at exactly on point, that is smooth in $X$, that $S_X$ contains exactly two singular points of $X$ of type $T_3$, that the surface $X$ itself is klt, and that the log canonical centers of the pair $(X,S_X+\frac{1}{2}L_X)$ still occur only in the obvious way along $S_X$ (i.e., the dual complex to this pair is a single point). Running a $\left(K_X+S_X+\frac{1}{2}L_{0,X}\right)$-MMP, it now terminates with 
\begin{itemize}
    \item a Mori fiber space onto a curve, in which case Proposition \ref{prop:mfs} concludes;
    \item or a klt Fano surface $Z'$ of Picard rank one underlying a log canonical Fano pair of coregularity $1$. Then, since $Z'$ is klt, the log canonical center $S_X$ was not contracted, hence was sent to a curve $S'$ on $Z'$. Thus, the pair obtained is of the form $(Z',S'+\frac{1}{2}L')$ with $L'$ the image of $L_{0,X}$ in $Z'$, which is a reduced irreducible curve, or equals zero. This is a plt Fano pair, and the curve $S'$ still contains two singular points of type $T_3$ of $Z'$. Thus, if $L'$ is non-zero, Proposition \ref{prop:1comp-sporadic233} concludes. Finally, if $L'$ is zero, we perform adjunction of the pair $(Z',S')$ to $S'$, obtain a pair $(S',\frac{2}{3}u+\frac{2}{3}v)$ supported at the two singular points of type $T_3$ of $Z'$ in $S'$, and lift its natural $1$-complement by Lemma \ref{lem:lifting-complements} to a $1$-complement of $(Z',S')$. Proposition \ref{prop:lc-cy} and Lemma \ref{lem:fun-group-increasing-coeff} then conclude.
\end{itemize}

\end{proof}

\subsection{Proof of Proposition \ref{prop:lc-Fano-case}}\label{subsec:lc-Fano-surf}
We have now gathered everything to prove the main result of this section, Proposition \ref{prop:lc-Fano-case}.

\begin{proof}[Proof of Proposition \ref{prop:lc-Fano-case}]
If ${\rm coreg}(X,\Delta^{\rm st})=2$, then the pair $(X,\Delta^{\rm st})$ is klt, so $\pi_1^{\rm orb}(X,\Delta)$ is finite by~\cite[Theorem 2]{Bra20}. In particular, it clearly is residually finite, and Lemma \ref{lem:fano-herratcon} and Corollary \ref{cor:res-finite-for-surf} conclude. If ${\rm coreg}(X,\Delta^{\rm st})=0$, then by Lemma~\ref{lem:coreg-vs-complementFano}, we know that $(X,\Delta^{\rm st})$ admits a $2$-complement $(X,\Gamma)$. Then Corollary \ref{cor:2cplt-sur} and Proposition~\ref{prop:lc-cy} conclude.

From now on, we assume that ${\rm coreg}(X,\Delta^{\rm st})=1$.
By definition of the coregularity of a Fano pair, there exists a log canonical Calabi--Yau pair $(X,\Gamma)$ of coregularity $1$ such that $\Gamma\ge \Delta^{\rm st}$. 
Let $\varepsilon:(Y,\Gamma_Y)\rightarrow (X,\Gamma)$ be a dlt modification of the pair $(X,\Gamma)$, let $E$ be its reduced exceptional divisor.
The pair $(Y,E+\Gamma_Y)$ is still a log Calabi--Yau pair. Denoting by $\Delta_Y$ the strict transform of $\Delta$, we have $\pi_1^{\rm orb}(Y,E+\Delta_Y) \simeq \pi_1^{\rm orb}(X,\Delta)$. 
By~\cite[Theorem 1.1]{FS20} and by our assumption on coregularity, the divisor $E+\lfloor\Gamma_Y\rfloor$ has either one, or two disjoint components.

If $E$ itself has two disjoint components, then by \cite[Proposition 3.3.2]{Pro01c}, there is a contraction to a smooth curve $\phi:Y\to C$ with general fiber $\pp^1$ such that the two components of $E$ are disjoint sections of $\phi$. We can decompose it into a birational contraction $\mu: (Y,E+\Delta_Y)\to (Z,s_0+s_1+\Delta_Z)$ and a Mori fiber space $f:Z\to C$. Lemma \ref{lem:fun-under-surj} and Proposition \ref{prop:mfs} conclude.

From now on, we can assume that $E$ has exactly one component.
Starting with the plt pair $(Y,E+\Delta_Y)$, we run a $(K_Y+\Delta_Y+E)$-MMP.
If it terminates with a Mori fiber space onto a curve, Lemma \ref{lem:fun-under-surj} and Proposition \ref{prop:mfs} conclude again. Otherwise, it terminates with a klt Fano surface $Z$ of Picard rank one and a plt non-klt Fano pair $(Z,E_Z+\Delta_Z)$. 
Since $Z$ is $\qq$-factorial and has Picard rank $1$, and since the pair is plt non-klt, the divisor $E_Z+\Delta_Z^{\rm st}$ has exactly one component of coefficient one. We write $S+L := E_Z+\Delta_Z^{\rm st}$, where $S$ is a smooth curve in $Z$ by Remark \ref{rem:plt}, and $L$ is a $\qq$-divisor with standard coefficients smaller than one.
The pair $(Z,S+L)$ remains a plt Fano pair. 
By Lemmas \ref{lem:fun-under-surj}, \ref{lem:resfinimage} and \ref{lem:fano-herratcon} and Corollary \ref{cor:res-finite-for-surf}, it suffices to prove that the group $\pi_1^{\rm orb}(Z,S+L)$ is residually finite to conclude.

\medskip

By Lemma \ref{lem:coeff-under-adj}, we have a log canonical Fano curve pair $(S,\Delta_S)$ obtained by adjunction of $(Z,S+L)$ to $S$. By Lemma \ref{lem:lifting-complements}, if the pair $(S,\Delta_S)$ obtained by adjunction of $(Z,S+L)$ to $S$ admits a $2$-complement, then the pair $(Z,S+L)$ itself admits a $2$-complement, and Corollary \ref{cor:2cplt-sur} and Proposition \ref{prop:lc-cy} conclude. This concludes the proof whenever the curve pair $(S,\Delta_S)$ is of toric type. By adjunction, the pair $(S,\Delta_S)$ has strictly negative degree, hence is not of elliptic type. We are left with the case when $(S,\Delta_S)$ is of sporadic type, i.e., 
$$(S,\Delta_S)\cong \left(\pp^1,\frac{1}{2}\{0\}+\frac{2}{3}\{1\}+\frac{n-1}{n}\{\infty\}\right)$$
for some $n\in\{3,4,5\}$.
We will shortly use this terminology in a case disjunction.

\medskip

Prior to the case disjunction based on the value of $n\in\{3,4,5\}$, let us consider and treat some of the simpler options for the divisor $L$.
If $L=0$, then ~\cite[Corollary 1.6]{KM99} shows that $\pi_1(Z_{\rm reg}\setminus S)$ is virtually abelian, hence residually finite as wished. If $L$ has at least three components, then Proposition~\ref{prop:toric-fund-group-3comps} concludes. If $L$ has exactly two components, then Proposition \ref{prop:toric-fund-group-2comps} concludes. Finally, we assume that $L$ has exactly one component $L_0$. As mentioned previously, we can distinguish three cases depending on the sporadic type of the pair $(S,\Delta_S)$, and Propositions \ref{prop:1comp-sporadic233}, \ref{prop:1comp-sporadic234}, and \ref{prop:1comp-sporadic235} then conclude. 
\end{proof}

\section{Proof of the theorems}
\label{sec:proofs}

Note that Theorem~\ref{introthm:fun-group-klt} is proved in Section \ref{sec:klt-cy-stand}, and Theorem \ref{introthm:Jordan-dlt-Fano} is proved in Section \ref{sec:Jordan}.
In this section, we prove the other main theorems of the article.

\begin{proof}[Proof of Theorem~\ref{introthm:fun-group-lcy}]
Let $(X,\Delta)$ be a log canonical Calabi--Yau surface pair.
By Lemma~\ref{lem:dlt-mod-vs-fun}, taking a dlt modification leaves the orbifold fundamental group unchanged, so we may assume that $(X,\Delta)$ is dlt. In particular, the surface $X$ is klt.
We run a $K_X$-MMP: We denote by $X\rightarrow Y$ be the sequence of divisorial contractions. Note that $Y$ has klt singularities, and let us further discuss its geometry.

First, assume that $Y$ is a minimal model, i.e., $K_Y$ is nef. Since the pseudoeffective cone $\overline{\mathrm{NE}}(Y)$ contains the nef cone and is non-degenerate, the push-forward of $\Delta$ to $Y$ is numerically trivial, thus $K_Y$ is numerically trivial.
By Lemma~\ref{lem:fun-under-surj}, it suffices to prove the conclusion of Theorem \ref{introthm:fun-group-lcy} for the group $\pi_1^{\rm orb}(Y)$.
This follows from Theorem~\ref{introthm:fun-group-klt}.

Now, assume that $Y$ is a Mori fiber space (to a curve or a point).
By Lemma~\ref{lem:fun-under-surj},
it suffices to prove the conclusion of Theorem \ref{introthm:fun-group-lcy} for the pair $(Y,\Delta_Y)$, where $\Delta_Y$ is the push-forward of $\Delta$ in $Y$.
If $Y$ is a Mori fiber space to a curve, then Proposition~\ref{prop:mfs} concludes.
Otherwise, $Y$ is a Mori fiber space to a point, i.e., a klt Fano surface of Picard rank $1$. If $\Delta_Y$ has standard coefficients, then Proposition~\ref{prop:lc-cy} concludes. If $\Delta_Y$ does not have standard coefficients, then the approximation $(Y,\Delta_Y^{\rm st})$ is a log canonical Fano pair, and Proposition~\ref{prop:lc-Fano-case} concludes.
\end{proof}

\begin{proof}[Proof of Theorem~\ref{introthm:fun-group-not-virt-ab}]
Let $(X,\Delta)$ be a log canonical Calabi--Yau surface pair for which the group $\pi_1^{\rm orb}(X,\Delta)$ is not virtually abelian. We follow the strategy of proof of Theorem \ref{introthm:fun-group-lcy}: By Lemma~\ref{lem:dlt-mod-vs-fun}, we may assume that $(X,\Delta)$ is dlt. In particular, the surface $X$ is klt.
We run a $K_X$-MMP: We denote by $X\rightarrow Y$ be the sequence of divisorial contractions. Note that $Y$ has klt singularities, and let us further discuss its geometry.

If $Y$ is a minimal model, by Lemma~\ref{lem:fun-under-surj}, we obtain the conclusion of Theorem \ref{introthm:fun-group-klt} for the group $\pi_1^{\rm orb}(X,\Delta)$: It is virtually abelian, a contradiction to our initial assumption.

Now, assume that $Y$ is a Mori fiber space (to a curve or a point) and denote by $\Delta_Y$ the push-forward of $\Delta$ in $Y$.
If $Y$ is a Mori fiber space to a curve, then Proposition~\ref{prop:mfs} concludes with the wished geometric description.

Assume finally that $Y$ is a Mori fiber space to a point, i.e., a klt Fano surface of Picard rank $1$. If $\Delta_Y$ has standard coefficients, then the 'Furthermore' part of Proposition~\ref{prop:lc-cy} (together with Proposition \ref{prop:mfs}) concludes. If $\Delta_Y$ does not have standard coefficients, then the approximation $(Y,\Delta_Y^{\rm st})$ is a log canonical Fano pair, and the 'Moreover' part of Proposition~\ref{prop:lc-Fano-case} (together with Proposition \ref{prop:mfs}) concludes.
\end{proof}

\begin{proof}[Proof of Theorem~\ref{introthm:fun-group-Z4}]
If $(X,\Delta)\simeq (A/G,{\rm Branch}(p))$ as in the statement of Theorem \ref{introthm:fun-group-Z4}, then Proposition \ref{prop:galois1} shows that $\zz^4$ appears as a normal subgroup of finite index in $\pi_1^{\rm orb}(X,\Delta)$. 

Conversely, assume that $\pi_1^{\mathrm{orb}}(X,\Delta)$ contains a subgroup isomorphic to $\zz^4$. 
As in the proof of Theorem~\ref{introthm:fun-group-lcy}, after taking a dlt modification $(X',\Delta')\rightarrow (X,\Delta)$ we run a $K_{X'}$-MMP to get a pair $(Y,\Delta_Y)$. Recall that $\pi^{\rm{orb}}(Y,\Delta_Y)$ surjects onto $\pi_1^{\rm{orb}}(X,\Delta)$ (by Lemma~\ref{lem:dlt-mod-vs-fun} and Lemma~\ref{lem:fun-under-surj}). Hence, $\pi_1^{\rm{orb}}(Y,\Delta_Y)$ also contains a normal subgroup isomorphic to $\mathbb{Z}^4$. Proposition~\ref{prop:mfs}, Proposition~\ref{prop:lc-cy} and Proposition~\ref{prop:lc-Fano-case} combined with the strategy of proof of Theorem \ref{introthm:fun-group-lcy} imply that $(Y,\Delta_Y)$ is a klt Calabi--Yau pair with $\Delta_Y=0$. Clearly, this means that $(X,\Delta)$ already was a klt Calabi--Yau pair. As a klt surface, $Y$ has isolated quotient singularities, which always have rational discrepancies, so $\Delta$ is a $\qq$-Weil divisor. 

We now follow the proof of Theorem \ref{introthm:fun-group-klt}: By the non-vanishing result \cite[V.4.9 Corollary]{Nak04}, there is a positive integer $M$ such that $M(K_{X}+\Delta)\sim 0$, hence the index one cover of $(X,\Delta)$ given by Lemma \ref{lem:index-one-cover} is a klt Gorenstein surface $X_1$ with $K_{X_1}\sim 0$. In particular, it has canonical singularities. Take the minimal resolution of singularities $r:A\rightarrow X_1$. Then $A$ is a smooth K3 surface or an abelian surface.
We still have a normal subgroup isomorphic to $\zz^4$ appearing in $\pi_1(A)$, thus $A$ is an abelian surface. It contains no rational curve, thus $A\simeq X_1$. Let $G$ be the Galois group of the index one cover $p: A\simeq X_1 \to (X,\Delta)$. We then obtain $(X,\Delta)=(A/G,\mathrm{Branch}(p))$, as wished.
\end{proof}

\begin{proof}[Proof of Theorem~\ref{introthm:fun-group-Z3}]
We proceed as in the proof of Theorem~\ref{introthm:fun-group-Z4}. First, we take a dlt modification $(X',\Delta')\rightarrow (X,\Delta)$, then run a $K_{X'}$-MMP to get to a pair $(Y,\Delta_Y)$. By Lemmas~\ref{lem:dlt-mod-vs-fun} and \ref{lem:fun-under-surj}, we have a surjection of fundamental groups: 
$\pi^{\rm{orb}}(Y,\Delta_Y)\twoheadrightarrow\pi_1^{\rm{orb}}(X,\Delta)$.

If the group $\pi_1^{\mathrm{orb}}(Y,\Delta_Y)$ contains a copy of $\zz^4$, then the proof of Theorem~\ref{introthm:fun-group-Z4} still works; It shows that $(X,\Delta) =(A/G,{\rm Branch}(p))$ for a finite quotient $:p:A\to A/G$ of an abelian surface. Thus, by Theorem~\ref{introthm:fun-group-Z4} again, $\pi_1^{\rm orb}(X,\Delta)$ contains a copy of $\zz^4$ too. That contradicts our assumption.

Thus, the group $\pi_1^{\mathrm{orb}}(Y,\Delta_Y)$ contains no copy of $\zz^4$. We also note that $\pi_1^{\mathrm{orb}}(Y,\Delta_Y)$ cannot admit a normal subgroup of finite index that has torsion-free rank at most $2$, since it surjects onto a group that contains a copy of $\zz^3$. Thus the pair $(Y,\Delta_Y)$ cannot satisfy both Hypotheses $(1)$ and $(2)$ of Corollary~\ref{cor:res-finite-for-surf}. By Theorem \ref{introthm:fun-group-lcy} however, the group $\pi_1^{\mathrm{orb}}(Y,\Delta_Y)$ always is residually finite. Thus the pair $(Y,\Delta_Y)$ cannot be hereditarily rationally connected.

\medskip

\noindent\underline{\textit{Case a):}}
The surface $Y$ is a Mori fiber space onto a curve. 

\medskip

\noindent If $(Y,\Delta_Y)$ is nevertheless hereditarily uniruled, then applying Lemma~\ref{lem:sarkisovlink} and Proposition \ref{prop:ellipticbase} shows that $\pi_1^{\mathrm{orb}}(Y,\Delta_Y)$ has a normal subgroup of finite index that is abelian of rank at most $2$, or that is a quotient of a Heisenberg group $H_k$ for $k\ge 0$. Note that for $k\ge 1$, the group $H_k$ cannot surject to any group that contains a copy of $\zz^3$. It is easy to check from there that $\pi_1^{\mathrm{orb}}(Y,\Delta_Y)$ surjects to a group containing a copy of $\zz^3$ if and only if it contains a copy of $H_0\simeq \zz^3$. Proposition \ref{prop:mfs} concludes that the pair $(Y,\Delta_Y)$ has the geometry predicted by Theorem \ref{introthm:fun-group-Z3}.

If $(Y,\Delta_Y)$ is not hereditarily uniruled, we follow the proof of Proposition \ref{prop:mfs} and obtain that $(Y,\Delta_Y)$ is a klt Calabi--Yau pair with standard coefficients. By the proof of Theorem \ref{introthm:fun-group-klt}, and since $\pi_1^{\rm orb}(Y,\Delta_Y)$ contains no copy of $\zz^4$, the pair $(Y,\Delta_Y)$ has a compatible finite cover $(Y_0,0)$ such that $Y_0$ has a smooth K3 surface as its minimal resolution. Thus, $\pi_1({Y_0}_{\rm reg})=\{1\}$, and thus $\pi_1(Y,\Delta_Y)$ is finite, a contradiction.

\medskip

\noindent\underline{\textit{Case b):}}
The pair $(Y,\Delta_Y)$ is a non-klt log canonical Calabi-Yau pair with standard coefficients, i.e., we are in the situation of  Proposition~\ref{prop:lc-cy}. 

\medskip

Looking at the proof of Proposition~\ref{prop:lc-cy}, the so-called ``Case 1'' is the only possibility, since the pair $(Y,\Delta_Y)$ is not hereditarily rationally connected. From the discussion of that ``Case 1'' in the proof of Proposition~\ref{prop:lc-cy}, we can find a log canonical Calabi--Yau pair that is birationally equivalent to $(Y,\Delta_Y)$, whose orbifold fundamental group surjects onto that of $(Y,\Delta_Y)$, but whose surface is a Mori fiber space to a curve. Then, we can conclude using {\it Case a)} right above.

\medskip

\noindent\underline{\textit{Case c):}}
The pair $(Y,\Delta_Y^{\rm st})$ is a log canonical Fano pair, i.e., we are in the situation of Proposition~\ref{prop:lc-Fano-case}. 

\medskip

Again, since $(Y,\Delta_Y)$ is not hereditarily rationally connected, we can skim through the proof of Proposition~\ref{prop:lc-Fano-case} and see that ${\rm coreg}(Y,\Delta_{Y}^{\rm st})=0$. As in the proof of Proposition~\ref{prop:lc-Fano-case}, the pair $(Y,\Delta_{Y}^{\rm st})=0$ admits a $2$-complement $(Y,\Gamma)$ of coregularity zero, which is now covered by {\it Case b)}. What {\it Case b)} shows here, is that the pair $(Y,\Gamma)$ is birationally equivalent to a pair which, after a compatible finite cover, is of the form
$$(\mathbb{P}(\mathcal{O}_E\oplus M),s_0+s_{\infty}),$$ 
where $E$ is an elliptic curve, $M$ is a numerically trivial line bundle on $E$, and $s_0,s_{\infty}$ are two sections of this $\pp^1$-bundle.
The fact that $\Delta_Y^{\rm st} <\Gamma$ however forces a group of the form
$\zz/n\zz\times \zz^2$ (obtained by strictly decreasing some coefficients of the divisor $s_0+s_{\infty}$) to be normal of finite index in a group that surjects to $\pi_1^{\rm orb}(Y,\Delta_Y^{\rm st})$, thus to $\pi_1(X,\Delta)$. This contradicts the fact that $\pi_1(X,\Delta)$ contains a copy of $\zz^3$, thus {\it Case c)} cannot occur.
\end{proof}

\begin{proof}[Proof of Theorem~\ref{introthm:fun-group-inf}]
Let $(X,\Delta)$ be a log Calabi--Yau surface for which $\pi_1^{\rm orb}(X,\Delta)$ is infinite. Let us fix a positive integer $N_0$ such that $N_0\Delta$ is an integral Weil divisor. Note that if $\pi_1^{\rm orb}(X,\Delta)$ is not virtually abelian, then Theorem \ref{introthm:fun-group-not-virt-ab} already concludes as wished. We can thus assume that it is virtually abelian. By Theorem \ref{introthm:fun-group-lcy} and Proposition \ref{prop:galois2}, there is a compatible finite Galois cover
$p:(Y,\Delta_Y)\to (X,\Delta)$ such that $\pi_1^{\rm orb}(Y,\Delta_Y)$ is of the form $\zz^r$, for some integer $r=1,2,3,$ or $4$. Since Theorems \ref{introthm:fun-group-Z4} and \ref{introthm:fun-group-Z3} deal with the cases $r=3,4$, we can now assume that $r=1$ or $2$. In particular, we obtain natural surjections $s_n:\pi_1^{\rm orb}(Y,\Delta_Y)\twoheadrightarrow (\zz/n\zz)^r$ for every positive integer $n$. By Proposition \ref{prop:galois2}, there are corresponding compatible finite cyclic covers $q_n:(Y_n,\Delta_n)\to (Y,\Delta_Y)$ for all $n$. Moreover, note that $q_n$ factors through $q_m$ if and only if $m$ divides $n$.

Fix $n$. Let $Y_n\to Y'_n$ be the outcome of a dlt resolution, then a $(\zz/n\zz)^r$-equivariant $K_{Y_n}$-MMP. Denote by $(Z_n,\Gamma_n)$ the quotient of $(Y'_n,\Delta'_n)$ by the action of $(\zz/n\zz)^r$. 
The group $\pi_1^{\rm orb}(Y'_n,\Delta'_n)$ surjects onto $\pi_1^{\rm orb}(Y_n,\Delta_n)\simeq \zz^r$. 

If $K_{Z_n}$ nef, then $\Gamma_n=0$ and the pair $(Z_n,0)$ satisfies the assumptions of Theorem \ref{introthm:fun-group-klt}. Thus, since it has infinite orbifold fundamental group, it has a normal subgroup of finite index isomorphic to $\zz^4$, and Theorem \ref{introthm:fun-group-Z4} concludes. 

Next, we consider the case when $Y'_n$ is a $(\zz/n\zz)^r$-Mori fiber space onto a point. Then, the klt surface $Y'_n$ is clearly of Fano type. The pair $(Y'_n,\lfloor \Delta'_n\rfloor)$ is a log canonical Calabi--Yau or Fano pair, that still has infinite orbifold fundamental group. If it is klt, then we conclude easily by Theorem \ref{introthm:fun-group-Z4} (in the Calabi--Yau case) and \cite{Bra20} (in the Fano case). Otherwise, it has an $N$-complement for some integer $N\le 6$. The result \cite[Theorem 3.2]{Mor20b} provides an integer $M\in\nn$ such that for $n\ge M$, after a suitable birational equivalence, we have a $(\cc^*)^r$-action on $Y'_n$ that commutes with the action of any subgroup of $(\zz/n\zz)^r$ of index at most $M$, and a $(\cc^*)^r$-invariant divisor $B_n\ge \lfloor\Delta'_n\rfloor$ such that $(Y'_n,B_n)$ is a log canonical Calabi--Yau pair. This concludes this case.

Finally, assume that $Y'_n$ is a $(\zz/n\zz)^r$-Mori fiber space onto a curve. If the base curve pair is of elliptic type, then we are done. Otherwise, the fibration is of the form $f:(Y'_n,\Delta'_n)\to (\pp^1,B_n)$, with a curve pair $(\pp^1,B_n)$ that is of toric or sporadic type. The action of $(\zz/n\zz)^r$ commutes with $f$, inducing an action on the base pair $(\pp^1,B_n)$ as well as one the general fiber $(F,\Delta_n|_F)$ of $f$.

Assume that, for some $n$ that is divisible by $4N_0$, we have $\zz/n\zz$ acting faithfully on the base pair $(\pp^1,B_n)$. Then, up to possibly a compatible double cover, the action on $\pp^1$ is by multiplication by a root of unity of order at least $2N_0$, and $B_n$ is supported on its two fixed points $0,\infty\in\pp^1$. Since the compatible cyclic quotient recovers a curve pair with coefficients in $\frac{1}{N_0}\zz$ (by definition of $N_0$), by the ramification formula, we see that $B_n=\{0\}+\{\infty\}$. By the canonical bundle formula, the pair $(Y'_n,\Delta'_n)$ admits a log canonical center above the point $p\in\pp^1$. In particular, after a $\zz/(n/N_0)\zz$-equivariant elementary transformation, $\Delta'_n$ contains vertical components above $0$ and $\infty$. Since $F\cdot \Delta'_n = 2$, the divisor $\Delta'_n$ also has horizontal components with non-zero coefficients, of course. This shows that $Y'_n$ is of Fano type, and then \cite[Theorems 3.1 and 3.2]{Mor20b} conclude.

Finally, assume that, for some $n$ that is divisible by $4N_0$, we have $\zz/n\zz$ acting faithfully on the fiber pair $(F,\Delta'_n|_F)$. By a similar argument as above (see also \cite[Proof of Proposition 2.24]{Mor20b}), we see that $\Delta'_n$ contains two horizontal components with coefficient $1$. If they intersect, then up to a $\zz/n\zz$-equivariant elementary transformation, the divisor $\Delta'_n$ also contains a vertical component, thus $Y'_n$ is of Fano type, and \cite[Theorems 3.1 and 3.2]{Mor20b} conclude. If the two horizontal components of $\Delta'_n$ are disjoint, then $Y'_n$ is birational to a smooth Hirzebruch surface, with $\Delta'_n$ being the sum of the minimal section and a disjoint positive section, thus the existence of a $(\cc^*)^2$-action preserving $\Delta'_n$, as wished. This concludes this proof.
\end{proof}

\section{Examples and Questions}

In this section, we collect some interesting examples and questions related to the main results.

\subsection{Failure of virtual abelianity in Theorem~\ref{introthm:fun-group-lcy}}

\begin{example}\label{ex:circle-over-elliptic}
{\em 
In this example, we compute the fundamental group of all the principal $S^1$-bundles over an elliptic curve $E$.

First note that for any smooth real manifold $M$, its classes of isomorphic principal $S^1$-bundles is $\mathrm{H}^2(M,\mathbb{Z})$. In particular, the isomorphic class of principal $S^1$-bundles of an elliptic curve $E$ is $\mathrm{H}^2(E,\mathbb{Z})\cong \mathbb{Z}$.

Topologically $E$ is obtained by gluing a disk $\mathbb{D}$ to $S^1\vee S^1$. Set $U\subset E$ to be the interior of the disc and $V\subset E$ be an open neighborhood of $S^1\vee S^1$ such that $V$ is homotopical to $S^1\vee S^1$ and there is a deformation retraction $r:U\cap V\rightarrow S^1$.
Now for each $k\in \mathbb{Z}$ we set $B_k$ to be the $S^1$-bundle over $E$ by the transition map $t_k:(U\cap V)\times S^1\rightarrow (U\cap V)\times S^1$, 
defined by $(x,\theta)\rightarrow (x,\theta+kr(x))$, where $S^1$ is regarded as an additive group. 

We apply van Kampen theorem to the open cover $\{U\times S^1, V\times S^1\}$ of $B_k$. We get that $\pi_1(V\times S^1)=\pi_1((S^1\vee S^1)\times S^1)=(\mathbb{Z}\ast\mathbb{Z})\times \mathbb{Z}$ with generators $a,b,c$ corresponding to the first, second and third factor, $\pi_1(U\times S^1)=\mathbb{Z}$ with generator $d$,  and $\pi_1((U\cap V)\times S^1)=\pi_1(S^1\times S^1)=\mathbb{Z}\times \mathbb{Z}$ with generators $x,y$ corresponding to the first and second factor. We may further assume that $r$ retracts $U\cap V$ to a representative of $x$. Now regard $(U\cap V)\times S^1$ as a subspace of $V\times S^1$ being glued to $U\times S^1$ by $t_k$. Then in $(U\cap V)\times S^1$, the inclusion morphism send $x$ to the commutator $[a,b]$ and $y$ to $c$. Via $t_k$, the element $x$ is mapped to $kd$ and $y$ is mapped to $d$. We thus get a presentation of $\pi_1(B_k)=\langle a,b,c,d|[a,b]d^{-k},[a,c],[b,c],cd^{-1}\rangle$ which can then be simplified to be 
\begin{equation*}
    \pi_1(B_k)=\langle a,b,c|[a,b]c^{-k},[a,c],[b,c]\rangle.
\end{equation*}
The long exact sequence of the fibration $B_k\rightarrow E$ now has the form:
\begin{center}
    $1\rightarrow \langle c\rangle \rightarrow \langle a,b,c|[a,b]c^{-k},[a,c],[b,c]\rangle \rightarrow \langle a,b|[a,b] \rangle \rightarrow 1$,
\end{center}
where the second arrow maps $c$ to the neutral element in $\mathbb{Z}^2$. Thus we see that for any $u,v\in \pi_1(B_k)$ its commutator $[u,v]$ lies in $\langle c\rangle$. Hence the derived length of $\pi_1(B_k)$ is at most $2$, i.e., $\pi_1(B_k)$ is metabelian.

If $k\neq 0$, then $\pi_1(B_k)$ is not virtually abelian. Indeed, first note that \begin{equation}
\label{commutator}
  [a^m,b^m]=(a^mba^{-m})^mb^{-m}=(c^{mk}b)^mb^{-m}=c^{km^2}   
\end{equation}

for any positive integer $m$. Suppose that $H\leqslant \pi_1(B_k)$ is an abelian normal subgroup of finite index $m$, then its image $H'$ in $\langle a,b|[a,b] \rangle=\mathbb{Z}^2$ is also of finite index $m$. In particular $(m,0)\in H'$ and $(0,m)\in H'$. Thus for some integers $u,v$, the elements $a^mc^u$ and $b^mc^v$ belong to $H$. However their commutator is $[a^mc^u,b^mc^v]=[a^m,b^m]=c^{km^2}\neq 1$. Hence $\pi_1(B_k)$ is not virtually abelian.

}
\end{example}

\begin{example}\label{ex:p1-over-elliptic}
{\em Let $E$ be an elliptic curve and $L$ be a degree $k\in \mathbb{Z}$ line bundle over $E$. The projective bundle $X_k:=\mathbb{P}(\mathcal{O}\oplus L)\rightarrow E$ has two disjoint sections $s_0$ and $s_{\infty}$ corresponding to the two factors of $\mathcal{O}\oplus L$. The pair $(X,s_0+s_\infty)$ is log Calabi--Yau.

First we claim that $X_k\setminus (s_0\cup s_\infty)$ is homotopical to the $S^1$-bundle $B_k$ in Example~\ref{ex:circle-over-elliptic}. Indeed, we cover $E$ by open subsets $\{U_\alpha\}$ such that on each $U_{\alpha}$ there is a trivialization $\phi_{\alpha}:L|_{U_{\alpha}}\cong U_\alpha\times \mathbb{C}$. We denote by $\phi_{\alpha\beta}:=\phi_{\alpha}\circ \phi_{\beta}^{-1}$ the transition function. We can also normalize the transition function. Let $h$ be a Hermitian metric over $L$ and denote by $h_\alpha:U_\alpha\rightarrow \mathbb{R}^+$ its corresponding positive function on $U_\alpha$. Now consider the trivialization of $L$ as a complex vector bundle $\tilde{\phi}_\alpha(v_x):=h_\alpha(x)\phi_\alpha(v_x)$. The transition functions become $\tilde{\phi}_{\alpha\beta}=\phi_{\alpha\beta}{\vert\phi_{\alpha\beta}^{-1}\vert}$. The transition function for the $\mathbb{P}^1$-bundle $X_k=\mathbb{P}(\mathcal{O}\oplus L)$
is given by the class of 
\[
\begin{pmatrix}
1 & 0 \\
0 & \phi_{\alpha\beta}^{-1}
\end{pmatrix}
\text{ in } \mathrm{PGL}(1,\mathbb{C}).
\]
The two sections $s_0$ and $s_\infty$ in each $\mathbb{P}^1\cong S^2$-fiber correspond to the north and south poles. Hence after deformation retracting to an $S^1$-bundle $B$, the transition functions become $\{\tilde{\phi}_{\alpha\beta}^{-1}\}$.
Consider now the following two exponential sequences:
\begin{center}
    $\xymatrix{
0\ar[r]\ar[d] & \mathbb{Z} \ar[r]\ar[d]&\mathcal{C}_{\mathbb{R}}\ar[d]\ar[r] & \mathcal{U}\ar[r]\ar[d] & 0 \ar[d]\\
0\ar[r] & \mathbb{Z} \ar[r] & \mathcal{C}_{\mathbb{C}}\ar[r] &\mathcal{C}_{\mathbb{C}}^{\ast}\ar[r] & 0}$,
\end{center}
where $\mathcal{U}$ is the sheaf of germs of complex-valued smooth functions of norm $1$. Take the cohomology, we get the following diagram
\begin{center}
    $\xymatrix{H^1(E,\mathcal{U})\ar[r]\ar[d] & H^2(E,\mathbb{Z}) \ar[d]^{{\rm id}}\\
    H^1(E,\mathcal{C}_{\mathbb{C}}^{\ast})\ar[r]^\delta & H^2(E,\mathbb{Z})}$.
\end{center}
The upper row is an isomorphism that maps a cocycle of a $S^1$-bundle $B$ to its characteristic class; the lower row maps a cocycle of a complex line bundle $L$ to $\delta(L)=-c_1(L)$ (\emph{cf.}~\cite[Proposition~4.4.12]{Huy05}). In particular, beginning with the cocycle $\{\tilde{\phi}_{\alpha\beta}^{-1}\}$, we get $\int_E-c_1(L^\ast)=k$. Hence $X_k\setminus (s_0\cup s_\infty)$ is homotopic to $B_k$.

Thus $\pi_1^{\rm orb}(X,s_0+s_\infty)=\langle a,b,c|[a,b]c^{-k},[a,c],[b,c]\rangle$, which is nilpotent of length 2, and virtually abelian if and only if $k=0$.
}
\end{example}

The following example shows that
even if the orbifold fundamental group of a klt log Calabi--Yau surface is virtually abelian, it is not effectively
virtually abelian.
\begin{example}\label{ex:p1-over-elliptic-standard}
{\em 
Let $X_k,s_0$ and $s_\infty$ be the $\mathbb{P}^1$-bundle defined in Example~\ref{ex:p1-over-elliptic}. 
Note that $s_\infty-K_{X_k}\sim s_0+2s_\infty$ is nef and big and $s_\infty$ is nef and big. By base point free theorem there exists $N\gg 0$ such that $|Ns_\infty|$ has no base point. Let $D\in |Ns_\infty|$ be a general member. Then $K_X+s_0+(1-\frac{1}{m})s_\infty+\frac{1}{mN}D\sim 0$. Hence $(X_k,s_0+(1-\frac{1}{m})s_\infty+\frac{1}{mN}D)$ is log Calabi--Yau. Note that by definition 
\[
G_m:=\pi_1^{\rm orb}\left(X_k,s_0+\left(1-\frac{1}{m}\right)s_\infty+\frac{1}{mN}D\right)=\pi_1^{\rm orb}\left(X_k,s_0+\left(1-\frac{1}{m}\right)s_\infty\right).
\]
Let $\gamma_\infty$ be a meridian loop $s_\infty$. Then $G_{m}=\pi_1^{\rm orb}(X_k,s_0+s_\infty)/N$, where $N$ is the normal sub-group generated by $\gamma_\infty^m$. We represent $\pi_1^{\rm orb}(X_k,s_0+s_\infty)=\pi_1(B_k)=\langle a,b,c|[a,b]c^{-k},[a,c],[b,c]\rangle$ as in Example~\ref{ex:circle-over-elliptic}. Note that $c$ comes from an $S^1$-fiber over $E$. Hence $\gamma_\infty$ will have $c$ as image in $\pi_1(B_k)$. Thus we get 
\begin{equation*}
    \pi_1^{\rm orb}\left(X_k,s_0+\left(1-\frac{1}{m}\right)s_\infty\right)=\langle a,b,c|[a,b]c^{-k},[a,c],[b,c],c^{m}\rangle.
\end{equation*}

 Consider the subgroup $A$ generated by $a^m$ and $b^m$. By Equation~\ref{commutator} one has $[a^m,b^m]=c^{km^2}=1$. One easily checks that $[a,b^m]=c^{-mk}=1$ and $[b,a^m]=c^{mk}=1$. Hence $A$ is a normal abelian subgroup of $G_{m}$ with finite index. Hence $G_m$ is virtually abelian. 

However, here for the family $(G_m)_{m\in \nn}$ we do not have effective virtual abelianity. In other words, we can not find a universal constant $C$ such that each group $G_{m}$ has a normal abelian subgroup of finite index at most $C$.
Indeed, suppose that $A'\leqslant G_{m}$ is the normal abelian subgroup of smallest index $l$. The image of $A'$ in $\mathbb{Z}^2$ is a subgroup of finite index $l$. The argument in the end of Example~\ref{ex:circle-over-elliptic} shows that $[a^l,b^l]=c^{kl^2}$. Hence we have that $l\geq \sqrt{\frac{m}{k}}$.

}
\end{example}

\subsection{Sharp and otherwise large indices}

The next example realizes the index 3840 appearing in Theorem \ref{introthm:fun-group-klt}, and is due to Kondo \cite{Kondo99}.

\begin{example}\label{ex:kondo}
    Consider the Kummer surface $S$ obtained from the abelian surface $E_i\times E_i$, where $i$ is the usual fourth root of unity and $E_i=\cc/\zz[i]$. Then \cite[Lemma 1.2]{Kondo99} shows that an extension $G$ of $\zz/4\zz$ by Mathieu group $M_{20}$ acts faithfully on $S$. Let $p=S\to S/G$ denote the quotient map. 
    The pair $(S/G,{\rm Branch}(p))$ is a klt Calabi--Yau pair, and by Galois correspondence, $\pi_1^{\rm orb}(S/G,{\rm Branch}(p))$ is isomorphic to $G$. As wished, $|G|=3480$.
\end{example}

The article~\cite{CKO03} discusses the fundamental group of the smooth locus of a normal Calabi-Yau surface with canonical singularities. Inspired by a question from F. Catanese, we describe some finite abstract groups that can appear as such fundamental groups.

\begin{example}\label{ex:mukai}
{\em
    Let $G$ be a subgroup of one of the groups in the list by Mukai \cite{Mukai88}, acting faithfully and symplectically on a smooth algebraic K3 surface $S$.
    Consider $p:S \to  \Sigma := S/G$. Since $G$ acts symplectically, any element of $G$ fixes finitely many points, in particular the finite morphism $p$ is quasi--\'etale. The surface $\Sigma$ is thus a normal Gorenstein Calabi--Yau surface, with quotient, hence klt, singularities. Hence, $\Sigma$ has in fact canonical singularities.
    By the Galois correspondence (Lemma \ref{lem:galoiscorr}) and since $S$ is simply connected, we have $\pi_1(\Sigma_{\rm reg})=G$.
}
\end{example}

The following classical example of the group $\mathfrak{A}_6$ lies in Mukai's list. It allows us to realize $\mathfrak{A}_6$ as the fundamental group of the smooth locus of a normal Calabi--Yau surface with canonical singularities.

\begin{example}
 Consider the $K3$ surface $S\subset\mathbb{P}^5:=\{\sum x_i=\sum x_i^2=\sum x_i^3=0\}$. The group $\mathfrak{S}_6$ acts on $\mathbb{P}^5$ by permuting the coordinates, and this action restricts to $S$. Every transposition is anti-symplectic, and fixes a locus of codimension $1$ in $S$. This yields a symplectic action of $\mathfrak{A}_6$, to which Example \ref{ex:mukai} applies.

Note that the finite Galois cover $p:S\to S/\mathfrak{S}_6$ ramifies in codimension $1$. It is compatible with the klt Calabi--Yau pairs $(S,0)$ and $(S/\mathfrak{S}_6,{\rm Branch}(p))$, so by Galois correspondence (Lemma \ref{lem:galoiscorr}), we have $\pi_1^{\rm orb}(S/\mathfrak{S}_6,{\rm Branch}(p))=\mathfrak{S}_6$.
\end{example}

The following example shows that there exists log Calabi--Yau surface $(X,\Delta)$ for which $\pi_1^{\rm orb}(X,\Delta)$ is finite and admits no nilpotent subgroup of index smaller than $7200$.

\begin{example}\label{ex:quotient-p1xp1}
We first construct a log Calabi--Yau surface $(X,\Delta)$ with fundamental group $\pi_1^{\rm orb}(X,\Delta)=(\mathfrak{A}_5\times \mathfrak{A}_5)\rtimes \mathbb{Z}/2\mathbb{Z}$. Then we show that $(\mathfrak{A}_5\times \mathfrak{A}_5)\rtimes \mathbb{Z}/2\mathbb{Z}$ has no metabelian subgroups except $\{1\}$. Note that the order of $(\mathfrak{A}_5\times \mathfrak{A}_5)\rtimes \mathbb{Z}/2\mathbb{Z}$ is $7200$ hence it has no metabelian subgroups of index less than $7200$.

Recall that there is a faithful projective representation of $\mathfrak{A}_5$ into $\mathrm{PGL}(2,\mathbb{C})$, yielding an action $\mathbb{P}^1$, and thus an action of $(\mathfrak{A}_5\times \mathfrak{A}_5)\rtimes \mathbb{Z}/2\mathbb{Z}$ on $\mathbb{P}^1\times \mathbb{P}^1$. Now the finite subgroup of $\mathrm{Aut}(\mathbb{P}^1\times \mathbb{P}^1)$  generated by $\mathfrak{A}_5\times \mathfrak{A}_5$ and the involution $\tau$ is $(\mathfrak{A}_5\times \mathfrak{A}_5)\rtimes \mathbb{Z}/2\mathbb{Z}$ with the $\mathbb{Z}/2\mathbb{Z}$ acting on $\mathfrak{A}_5\times \mathfrak{A}_5$ by $\tau.(g_1,g_2)=(g_2,g_1)$. 

Set $Y=\mathbb{P}^1\times \mathbb{P}^1$ and $X$ the quotient of $Y$ by $(\mathfrak{A}_5\times \mathfrak{A}_5)\rtimes \mathbb{Z}/2\mathbb{Z}$. We have 
\begin{center}
    $\pi:(Y,0)\rightarrow (X,\Delta)$,
\end{center}
with $\pi^{\ast}(K_X+\Delta)\sim K_Y$ where $\Delta$ is the branch divisor of $\pi$. In particular $(X,\Delta)$ is a log Fano surface. By Galois correspondence, 
we have
\begin{equation*}
    1\rightarrow \pi_1^{\rm orb}(Y,0)\rightarrow \pi_1^{\rm orb}(X,\Delta)\rightarrow (\mathfrak{A}_5\times \mathfrak{A}_5)\rtimes \mathbb{Z}/2\mathbb{Z}\rightarrow 1.
\end{equation*}
As $\pi_1(Y)=1$, we have thus 
\begin{equation*}
    \pi_1^{\rm orb}(X,\Delta)= (\mathfrak{A}_5\times \mathfrak{A}_5)\rtimes \mathbb{Z}/2\mathbb{Z}.
\end{equation*}
Now take one general curve $C_1 \cong \mathbb{P}^1\times \{\mathrm{pt}\}$ and one general curve $C_2\cong \{\mathrm{pt}\} \times \mathbb{P}^1$. In particular the generic points of these curves lie outside the ramification locus of $\pi$. Set $D':=\frac{1}{7200}( \pi_{\ast}C_1+\pi_{\ast}C_2)$ then $\pi^{\ast}(D')=C_1+C_2$ and $D'$ has no common components with $\Delta$. The pair $(X,\Delta+D')$ is then log Calabi--Yau and $(\Delta+D')^{\mathrm{st}}=\Delta$. Thus $\pi_1^{\rm orb}(X,\Delta+D')= (\mathfrak{A}_5\times \mathfrak{A}_5)\rtimes \mathbb{Z}/2\mathbb{Z}$

We now show that $G:=\pi_1^{\rm orb}(X,\Delta)$ has no metabelian normal subgroups other than \{1\}. Consider the short exact sequence
\begin{center}
    $1\rightarrow \mathfrak{A}_5\times \mathfrak{A}_5\rightarrow G\rightarrow \mathbb{Z}/2\mathbb{Z}\rightarrow 1$.
\end{center}
First $G$ is not metabelian. Indeed, its commutator subgroup $[G,G]$ is mapped to $1$ in $\mathbb{Z}/2\mathbb{Z}$ hence $[G,G]=\mathfrak{A}_5\times \mathfrak{A}_5$ which is not abelian. Now suppose that $H$ is a proper normal subgroup of $G$ then $H\leqslant \mathfrak{A}_5\times \mathfrak{A}_5$. In particular, the subgroup $H$ can only be $\mathfrak{A}_5\times \mathfrak{A}_5$, $\mathfrak{A}_5\times \{1\}$, $\{1\}\times \mathfrak{A}_5$ or $\{1\}$. Except for the trivial case, each of the first three cases has non-abelian commutator subgroup, hence is not metabelian.
\end{example}

We conclude with two questions of interest.

\begin{question}
Which are the finite groups that act on log Calabi--Yau surfaces?
\end{question}

The finite subgroups of the plane Cremona are precisely the finite groups that act on some smooth rational surface.
These groups are classified
by the work of Dolgachev and Iskovskikh~\cite{DI09}.
It would be interesting to have a complete classification of all 
finite groups that act on a surface of log Calabi--Yau type.
Every such subgroup is the homomorphic image of the fundamental group of a log Calabi--Yau surface.
Thus, Theorem~\ref{introthm:fun-group-lcy} will serve as a tool to give a classification of these finite groups.
Note that every finite group acts on some canonically polarized surface. Hence, this classification is only interesting for finite actions on surfaces with non-negative curvature.
Note that an answer to the previous question may shed light on group actions on log Calabi--Yau surfaces, whenever such a group has interesting finite subgroups.

\begin{question}\label{quest:infinite}
Let $(X,\Delta)$ be a $n$-dimensional log Calabi--Yau pair with $\pi_1^{\rm orb}(X,\Delta)$ being infinite. How does this reflect on the geometry pair $(X,\Delta)$? 
\end{question}

Theorem~\ref{introthm:fun-group-inf} gives an answer to this in dimension $2$. Up to a finite cover and a birational transformation, we either have a $\mathbb{G}_m$-action or a fibration onto an elliptic curve.
In higher dimensions, it is natural to expect a similar behavior. 
For instance, in the context of Question~\ref{quest:infinite}, it is natural to expect that up to a finite cover and a birational transformation
$X$ admits a fibration to an abelian variety or a variety with a conic bundle structure.
However, in higher dimensions, very little is known about the orbifold fundamental group of log Calabi--Yau pairs.
Hence, it may be more reasonable to study the algebraic fundamental group $\pi_1^{\rm alg}(X,\Delta)$ first.

\bibliographystyle{plain}
\bibliography{mybib}

\newpage

\appendix

\section{Basic properties of the orbifold fundamental group}\label{appA}

The results of this appendix should be well-known to experts, as they feature basic properties that are standard in the study of any type of fundamental group: We see how the right type of morphisms between pairs induces pushforward maps between orbifold fundamental groups; We show that this notion of pushforward is functorial; We provide a Galois correspondence for orbifold fundamental groups. 

For the sake of a complete and self-contained reading experience, we provide here statements and proofs of these basic, useful facts.

\subsection{The loop around a divisor}

\begin{definition}
 Let $X$ be a normal quasi-projective variety, and $D$ be a non-trivial prime effective Weil divisor on $X$. 
    We say that an analytic open subset $U$ of $X$ is a \emph{trimmed neighborhood} of $D$ in $X$ if there are a Zariski open set $U_{\rm Zar}\subset X$ intersecting $D$, and an (analytic) tubular neighborhood $U_{\rm tub}$ for $D\cap (X_{\rm reg}\setminus D_{\rm sing})$ in $X_{\rm reg}\setminus D_{\rm sing}$, such that $U=U_{\rm Zar}\cap U_{\rm tub}$. If $U$ is a trimmed neighborhood of $D$, we denote by $U^*$ the topological manifold $U\setminus D\cap U$. It can be thought of as a pointed neighborhood of $D$.
    
    Of course, the notion of trimmed neighborhood for a fixed divisor $D$ is stable by finite intersection and arbitrary union.

    An inclusion of trimmed neighborhoods $V\subseteq U$ induces inclusions $V_{\rm tub}\subseteq U_{\rm tub}$ and $V_{\rm Zar}\subseteq U_{\rm Zar}$, thus a surjection
    $$\pi_1(V^*)\simeq \pi_1({U_{\rm tub}}^*\cap V_{\rm Zar})\twoheadrightarrow \pi_1(U^*),$$
     where we use that there is a homeomorphism between ${V_{\rm tub}}^*$ and ${U_{\rm tub}}^*$ that preserves $V_{\rm Zar}$ (since its complement has finitely many components, it suffices to preserve them).
    
\end{definition}

\begin{definition}
Let $X$ be a normal quasi-projective variety, and $D$ be a non-trivial prime effective Weil divisor on $X$.
We define the loop around $D$ as the data, for any trimmed neighborhood $U$ of $D$, of the class $\gamma_{D}|_U$ of the positive oriented loop in the fiber of the normal circle bundle of $D\cap U$, in the fundamental group $\pi_1(U^*)$. Since the aforementioned surjection $\pi_1(V^*)\twoheadrightarrow\pi_1(U^*)$ induced by an inclusion $V\subset U$ sends $\gamma_D|_{V}$ to $\gamma_D|_{U}$, we might omit to precise the trimmed neighborhood used and simply write $\gamma_D$ whenever possible.
\end{definition}

\begin{definition}\label{def-canpb}
 Let $X$ and $X'$ be normal projective varieties, together with a proper map $p:X'\to X$. Assume that $p$ is flat in codimension $0$ and $1$ in the source. By \cite[~21.10.1]{EGA-IV}, there is a well-defined pullback map $p^*:{\rm Cl}({\rm Im}(p))\to {\rm Cl}(X')$ at the level of Weil divisors.
    
    Let $D'$ be a non-trivial prime effective Weil divisor on $X'$. We say that $p$ {\em ramifies} at $D'$ with degree $m$ if $D:=p(D')$ has codimension one in ${\rm Im}(p)$, and $D'$ appear as a component of $p^*D$ with coefficient $m$.
    
    We define the {\em ramification divisor} of $p$ as 
    $${\rm Ram}(p):=\sum_{(m,D)\in R} (m-1)D,$$
    where $R$ is the set of pairs $(m,D)$ such that $m\ge 2$ is an integer and $p$ ramifies at $D$ with degree $m$. Note that the index set $R$ is finite.
\end{definition}

The following result is an elementary consequence of a computation in local coordinates, and explains how ramification affects fundamental groups.

\begin{lemma}
    Let $X$ and $Y$ be normal projective varieties, together with a proper map $p:X\to Y$ of image $Y_0$ in $Y$, that is flat in codimension 0 and 1 on the source. Let $D$ be an irreducible effective Weil divisor in $X$, along which $p$ ramifies with degree $m$. Let $D_0 = p_0(D)$ in $Y_0$. 
    Then we have $${p_0}_*(\gamma_D)={\gamma_{D_0}}^m.$$
\end{lemma}

\begin{proof}
    In local coordinates, we want to pushforward the positively oriented loop $\gamma$ generating $\pi_1(\mathbb{C}^*)$ by the map $f:z\in\mathbb{C}^* \mapsto z^m\in\mathbb{C}^*$. Clearly, $f_*(\gamma)=\gamma^m$, as wished.
\end{proof}

\subsection{A pushforward map for orbifold fundamental groups}

In this subsection, we construct a functorial pushforward map for orbifold fundamental groups of pairs.

\begin{definition}\label{def:compatible_annex}    
Let $p:X'\to X$ be a proper surjective map of normal projective varieties that can pullback. Let $\Delta'$ and $\Delta$ be effective $\mathbb{Q}$-divisors on $X'$ and $X$, whose components appear with coefficients at most one each.
We say that $p$ is \emph{compatible} with the pairs $(X',\Delta')$ and $(X,\Delta)$ if
    $$p^*\Delta={\rm Ram}(p)+\Delta',$$
    and the following conditions are satisfied:
    \begin{enumerate}
        \item the divisor $\Delta'$ is an effective $\mathbb{Q}$-divisor on $X'$ whose irreducible components appear with coefficient at most one each,
        \item for any component $D'$ appearing in $\Delta'$ with coefficient $a>0$ and in ${\rm Ram}(p)$ with coefficient $m-1>0$, there is a component $D$ appearing in $\Delta$ with coefficient $b>0$ such that $D'$ appears in $p^*D$ and we have:
        $$m(1-b^{\rm st})=1-a^{\rm st}.$$
    \end{enumerate}

\noindent 
For simplicity, we may say that 
$p\colon (X',\Delta')\rightarrow (X,\Delta)$ is compatible. 
\end{definition}

\begin{remark}
    One may note that for a compatible morphism $p:(X',\Delta')\to (X,\Delta)$, it holds $$p^{\ast}(\Delta^{\rm st})={\rm Ram}(p)+(\Delta')^{\rm st}.$$
\end{remark}

\begin{remark}
   If $p:X'\to X$ is a finite quasi-\'etale cover, and $\Delta$ is an effective $\mathbb{Q}$-divisor on $X$ whose components appear with coefficient at most one each, then setting $\Delta':=p^*\Delta$, and we note that the map $p$ is compatible with $(X',\Delta')$ and $(X,\Delta)$.
\end{remark}

An example of a compatible map that is not a finite cover, but a Mori fiber space onto a curve, is given in Remark \ref{rem:mfscompatible}.

\begin{construction}\label{con:pushforward}  
 Let $(X',\Delta')$ and $(X,\Delta)$ be pairs, and $p:(X',\Delta') \to (X,\Delta)$ be a compatible map.
We construct a pushforward group homomorphism:
$$p_{\bullet}:\pi_1^{\rm orb}(X',\Delta')\to\pi_1^{\rm orb}(X,\Delta).$$
The construction goes as follows.
Let $B:=X_{\rm sing}\cup {\rm Supp}(\Delta^{\rm st})$.
We can restrict and corestrict $p$ to obtain a proper surjective map
$$\overline{p}:X'\setminus p^{-1}(B)
    \to X\setminus B.$$

We claim that $p^{-1}(B)$ is contained in $X'_{\rm sing}\cup{\rm Supp}({\Delta'}^{\rm st} + {\rm Ram}(p))\cup Z$, for some Zariski closed subset $Z$ in $X'_{\rm reg}$ of codimension at least two. Indeed, it is clear that $p^{-1}({\rm Supp}\,\Delta^{\rm st})={\rm Supp}\,{\Delta'}^{\rm st} + {\rm Ram}(p)$. If $X$ is a smooth curve, that is enough. Otherwise, $p:X'\to X$ is a finite cover, and since $X$ is normal, $p^{-1}(X_{\rm sing})$ has codimension at least two in $X'$.

Therefore, we have an open immersion 
    $$i:X'_{\rm reg}\setminus {\rm Supp}({\Delta'}^{\rm st}+{\rm Ram}(p))\cup Z \hookrightarrow X'\setminus p^{-1}(B)$$
    and we define $p_0:=\overline{p}\circ i.$
Since $Z$ is a Zariski closed subset of codimension at least two in $X'_{\rm reg}$, it does not affect fundamental groups. So we obtain the following diagram
    $$\xymatrix{
    \pi_1(X'_{\rm reg}\setminus {\rm Supp}({\Delta'}^{\rm st}+{\rm Ram}(p))) \ar@{->>}[d]_{s'} \ar[r]^-{{p_0}_*} 
    & \pi_1(X_{\rm reg}\setminus {\rm Supp}\left(\Delta^{\rm st}\right)) \ar@{->>}[d]^s
    \\
    \pi_1(X'_{\rm reg}\setminus {\rm Supp}({\Delta'}^{\rm st})) 
    \ar@{->>}[d]_{t'}
    \ar@{.>}[r]_{p_{\rm int}} 
    & \pi_1^{\rm orb}(X,\Delta)
    \\
    \pi_1^{\rm orb}(X',\Delta') 
    \ar@{.>}[ur]_{p_{\bullet}}
    & 
    }
    $$

Here, we claim that there exists a unique group homomorphism $p_{\bullet}$ that makes this diagram commute. Provided it exists, its unicity is clear from the fact that $s'$ and $t'$ are surjective.

First, we prove that $s\circ {p_0}_*$ factors through $s'$. Applying Van Kampen's theorem to add trimmed neighborhoods of components one by one, we see that
$$\ker(s')=\langle \gamma_R\mid R\mbox{ component of }{\rm Supp}\,{\rm Ram}(p)\setminus{\rm Supp}\,{\Delta'}^{\rm st}\rangle_n.$$
Fixing such a component $R$ of ramification order $m\ge 2$, we see by Definition \ref{def:compatible_annex} that $D:=p(R)$ appears in $\Delta^{\rm st}$ with coefficient $1-\frac{1}{m}$. So we have 
$s({p_0}_*\gamma_R)=s({\gamma_D}^m)=1.$
That proves the existence of a group homomorphism
$p_{\rm int}:\pi_1(X'_{\rm reg}\setminus {\rm Supp}({\Delta'}^{\rm st}))\to \pi_1^{\rm orb}(X,\Delta)$
making the diagram commutative.

\medskip

Second, we prove that $p_{\rm int}$ factors through $t'$. Recall that
$$\ker(t') = \left\langle {\gamma_{D'}}^k\mid D'\mbox{ component of }\,{\Delta'}^{\rm st}\mbox{ with coefficient }1-\frac{1}{k}\right\rangle_n.$$
Fixing such a pair $(D',k)$ and denoting by $m\ge 1$ the ramification order of $p$ along $D'$, we see by Definition \ref{def:compatible_annex} that $D:=p(D')$ appears in $\Delta^{\rm st}$ with coefficient $1-\frac{1}{km}$. So we have
$$p_{\rm int}({\gamma_{D'}}^k)=s({p_0}_*{\gamma_{D'}}^k)=s({\gamma_D}^{km})=1.$$
This shows that $p_{\rm int}$ factors through $t'$, hence the existence of the group homomorphism $p_{\bullet}$ as wished.
\end{construction}

We show that our construction is, in some sense, functorial.

\begin{proposition}\label{prop:can_compose}
Let $(X'',\Delta'')$, $(X',\Delta')$, and $(X,\Delta)$ be pairs. Let $q:(X'',\Delta'') \to (X',\Delta')$ and $p:(X',\Delta') \to (X,\Delta)$ be compatible maps, and assume that $p$ is a finite cover. Then the composition $p\circ q:(X'',\Delta'')\to (X,\Delta)$ is a compatible map.
\end{proposition}

\begin{proof}[Proof of Proposition \ref{prop:can_compose}]
Clearly, it suffices to show that
 $${\rm Ram}(p\circ q)=q^*{\rm Ram}(p)+{\rm Ram}(q).$$
Fix an irreducible effective divisor $D''$ in $X''$ and see whether it appears with the same coefficient on the left handside and on the right handside. 
If $q(D'')=X'$, then $D''$ does not appear in ${\rm Ram}(p\circ q)$ or in ${\rm Ram}(q)$. It does not appear in $q^*{\rm Ram}(p)$ either, since $q(q^*{\rm Ram}(p))\neq X'$. So $D''$ appears neither on the left, nor on the right handside.

Otherwise, $D':=q(D'')$ is an irreducible effective Weil divisor in $X'$. Since $p$ is a finite cover, $D:=p(D')$ is an irreducible effective divisor in $X$.
Let $m',m\ge 1$ be the multiplicities of $D''$ and $D'$ in $q^*D'$ and $p^*D$, respectively. The multiplicity of $D''$ in $q^*p^*D$ is then exactly $mn$, and of course $mn-1=m(n-1)+(m-1)$ as wished.
\end{proof}

\begin{proposition}\label{prop:functorialpf}
Let $(X'',\Delta'')$, $(X',\Delta')$ and $(X,\Delta)$ be pairs. Let $q:(X'',\Delta'') \to (X',\Delta')$ and $p:(X',\Delta') \to (X,\Delta)$ be compatible maps, and assume that $p\circ q$ is also compatible with the corresponding pairs. Then we have
$(p\circ q)_{\bullet}=p_{\bullet}\circ q_{\bullet}.$
\end{proposition}

\begin{proof}[Proof of Proposition \ref{prop:functorialpf}]
Note that the compatible map $q$ can be restricted and corestricted into a map
$$q_{0,p}:X''_{\rm reg}\setminus {\rm Supp}(\Delta''^{\rm st} + {\rm Ram}(p\circ q))\cup Z' \to X'_{\rm reg}\setminus {\rm Supp}(\Delta'^{\rm st} + {\rm Ram}(p))$$
where $Z'$ is a Zariski closed subset of $X''_{\rm reg}$ of codimension at least 2.
Clearly, the two group homomorphisms ${(p\circ q)_0}_*$ and ${p_0}_*\circ {q_{0,p}}_*$ then coincide. Factoring through the natural surjections of fundamental groups induced by the relevant open immersions, we see that ${q_{0,p}}_*$ and ${q_0}_*$ both induce the same unique group homomorphism $q_{\bullet}$ as in Construction \ref{con:pushforward}. The fact that $(p\circ q)_{\bullet}=p_{\bullet}\circ q_{\bullet}$ then follows from the fact that ${(p\circ q)_0}_* = {p_0}_*\circ {q_{0,p}}_*$.
\end{proof}

\subsection{A Galois correspondence for orbifold fundamental groups}\label{subsec:galois}
In this subsection, we prove a Galois correspondence for orbifold fundamental groups.
Recall that a finite cover $p:X'\to X$ is called {\it Galois} if there is a finite subgroup ${\rm Gal}(p)$ of ${\rm Aut}(X')$ such that $p$ is isomorphic to the natural quotient map of $X'\to X'/{\rm Gal}(p)$. 
Recall also that a finite cover is called {\it cyclic} if it is a Galois cover with a cyclic Galois group.

\begin{proposition}\label{prop:pushforward}
    Let $(X',\Delta')$ and $(X,\Delta)$ be pairs and $p:(X',\Delta') \to (X,\Delta)$ be a compatible map. 
    If $p$ is a finite Galois cover, then $p_{\bullet}$ is injective and $p_{\bullet}\pi_1^{\rm orb}(X',\Delta')$ is a normal subgroup of finite index in $\pi_1^{\rm orb}(X,\Delta)$, with quotient isomorphic to the Galois group ${\rm Gal}(p)$.  
\end{proposition}

\begin{proof}
By Zariski's purity of the branch locus, the map $\overline{p}$ defined at the beginning of Construction~\ref{con:pushforward} is a finite Galois étale cover. Moreover, $p^{-1}(X_{\rm sing}\cup{\rm Supp}\,\Delta^{\rm st})$ coincides exactly with $X'_{\rm sing}\cup{\rm Supp}({\Delta'}^{\rm st}+{\rm Ram}(p))\cup Z$, where $Z$ is a Zariski closed subset of $X_{\rm reg}$ of codimension at least $2$, so that the pushforward $i_*$ by the open inclusion $i$ is an isomorphism of groups.

Using the usual Galois correspondence for fundamental groups, the pushforward homomorphism $\overline{p}_*$ yields an exact sequence
$$1\to\pi_1(X'\setminus p^{-1}(B))\overset{\overline{p}_*}{\longrightarrow}\pi_1(X\setminus B)\to {\rm Gal}(p)\to 1.$$
Moreover, an easy computation shows that the image of $\overline{p}_*$ contains the kernel of $s$. Hence, by the first isomorphism theorem and since $s$ is surjective,
the image of $s\circ\overline{p}_*$ is a normal subgroup of $\pi_1^{\rm orb}(X,\Delta)$ with corresponding quotient group isomorphic to ${\rm Gal}(p)$.
Since $i_*$ is an isomorphism, the image of $s\circ {p_0}_*$ is exactly the image of $s\circ\overline{p}_*$. Since $s'$ and $t'$ are surjective, this is again the same as the image of $p_{\bullet}$. Hence, the image of $p_{\bullet}$ is a normal subgroup of finite index in $\pi_1^{\rm orb}(X,\Delta)$, with quotient isomorphic to the Galois group ${\rm Gal}(p)$, as wished.

\medskip

From here on, assume that $p(X'_{\rm sing})\subseteq X_{\rm sing}\cup p({\rm Ram}(p))$, and let us show that $p_{\bullet}$ is injective. By the usual Galois correspondence, the group homomorphism $\overline{p}_*$ is injective.  
Since the pushforward map $i_*$ is an isomorphism, ${p_0}_*$ is injective too.

We now want to prove that, with the notation of Construction~\ref{con:pushforward},
$${p_0}_*\ker(t'\circ s')=\ker(s).$$
Consider the subgroups
$$K':=\langle {\gamma_Z}^k\mid Z\mbox{ is a component of }{\Delta'}^{\rm st}+{\rm Ram}(p),\, 1-k^{-1} \mbox{ is the coefficient of }Z\mbox{ in }{\Delta'}^{\rm st}\rangle$$
inside $G':=\pi_1(X'_{\rm reg}\setminus {\rm Supp}(\Delta'^{\rm st}+{\rm Ram}(p)))$, and
$$K:=\langle {\gamma_D}^n\mid D\mbox{ is a component of }\Delta^{\rm st}\mbox{ with coefficient }1-n^{-1}\rangle,$$
inside $G:=\pi_1(X_{\rm reg}\setminus {\rm Supp}\,\Delta^{\rm st}).$ 
Let us fix a base point $x'\in X'_{\rm reg}\setminus {\rm Supp}(\Delta'^{\rm st}+{\rm Ram}(p))$ and let $x=p(x')$ be fixed as our base point in $X_{\rm reg}\setminus {\rm Supp}\,\Delta^{\rm st}$.

We want to show that
${p_0}_*K'_n = K_n$, where $K_n$ and $K'_n$ are smallest normal subgroups containing $K$ and $K'$ respectively. We already proved in Construction \ref{con:pushforward} that ${p_0}_*K'=K$.
Clearly, $({p_0}_*K'_n)_n = ({p_0}_*K')_n = K_n$, so it suffices to show that ${p_0}_*K'_n$ is a normal subgroup of $G$ to conclude.

Note that the action of $G$ by conjugation on ${p_0}_*G'$ induces a group homomorphism from ${\rm Gal}(p)\simeq G/{p_0}_*G'$ to ${\rm Aut}(G')/{\rm Inn}(G')$.  This action can also be more concretely viewed as follows:
Fix for each $h\in{\rm Gal}(p)$ a path $c_h$ from $h(x')$ to $x'$. For every $h\in{\rm Gal}(p)$, we can consider the automorphism of $G'$ sending a loop $\gamma'$ based at $x'$ to another loop $c_h h(\gamma') c_{h}^{-1}$ based at $x'$. Its class $[h]$ in the group ${\rm Aut}(G')/{\rm Inn}(G')$ does not depend on the choice of the path $c_h$, and we obtain in this way our natural group homomorphism ${\rm Gal}(p)\to {\rm Aut}(G')/{\rm Inn}(G')$. In other words, ${\rm Gal}(p)$ acts on the set of $G'$-conjugacy classes in $G'$.

Now, let $h\in{\rm Gal}(p)$ and ${\gamma_Z}^k$ a generator of $K'$ as in the beginning of this proof, based at $x'$. Then it is easy to check that $h$ sends the conjugacy class of ${\gamma_Z}^k$ to the conjugacy class of ${\gamma_{h(Z)}}^k$, still based at $x'$, which is an element of $K'$ as well. In particular, applying $h$ to the normal subgroup ${K'}_n$ of $G'$, we obtain $K'_n$ again. Hence, ${p_0}_*K'_n$ is indeed normal in $G$, as wished.
\end{proof}

\begin{lemma}\label{lem:galoiscorr}
    Let $(X,\Delta)$ be a pair, and let $H$ be a normal subgroup of finite index in $\pi_1^{\rm orb}(X,\Delta)$.
    Then there exists a pair $(X',\Delta')$ and a compatible map $p:(X',\Delta')\to (X,\Delta)$ that is a finite Galois cover, such that 
    $p_{\bullet}\pi_1^{\rm orb}(X',\Delta')$ coincides with $H$ as a subgroup of $\pi_1^{\rm orb}(X,\Delta)$.
\end{lemma}

\begin{proof}
    Let $H_0$ be the pre-image of $H$ in $\pi_1(X_{\rm reg}\setminus {\rm Supp}\,\Delta^{\rm st})$. It is again a normal subgroup of finite index, and by the usual Galois correspondence, we have a finite Galois unramified cover in the analytic topology 
    $$p_0:U\to X_{\rm reg}\setminus {\rm Supp}(\Delta^{\rm st})$$ 
    such that ${p_0}_*\pi_1(U)$ coincides with $H_0$ as a normal subgroup of $\pi_1(X_{\rm reg}\setminus {\rm Supp}(\Delta^{\rm st}))$. 
    
    By \cite[Theorems 3.4 and 3.5]{DG94}, there is an algebraic finite cover $p:X'\to X$ extending $p_0$. Since $X$ is normal, $p$ is a finite Galois cover and ${\rm Gal}(p)={\rm Gal}(p_0)$, which identifies with the quotient of $\pi_1^{\rm orb}(X,\Delta)$ by the normal subgroup $H$. 
    We set $\Delta':=p^*\Delta-{\rm Ram}(p)$, and want to show that $(X',\Delta')$ is a pair that makes $p$ a compatible map. 

    For that, it suffices to track the coefficients in $p^*\Delta,p^*\Delta^{\rm st},\Delta',\Delta'^{\rm st},$ and ${\rm Ram}(p)$ of any fixed irreducible effective divisor in $X'$. Fix $Z$ such a divisor.
    If $Z$ is not contained in ${\rm Ram}(p)$, then it appears with the same coefficient in $\Delta'$ as in $p^*\Delta$ and with the same coefficient in their standard approximations, respectively. That is enough.
    Assume now that $Z$ appears in ${\rm Ram}(p)$. Let $m$ be the ramification order of $p$ along $Z.$ Since $p_0$ is \'etale, the irreducible effective divisor $D:=p(Z)$ appears as a component of $\Delta^{\rm st}$ with coefficient $b^{\rm st}$. If $b^{\rm st} = 1$, then $Z$ appears with coefficient $1$ in each of the divisors $p^*\Delta,p^*\Delta^{\rm st},\Delta',\Delta'^{\rm st}$, and that's enough.

    Otherwise, we have $b^{\rm st} = 1-\frac{1}{n}$ for some $n\ge 2$. Let $a$ the (possibly zero) coefficient of $Z$ in $\Delta'$. Then clearly, $a=mb-(m-1)\in [0,1]$. From there on, it suffices to check that $n$ is divisible by $m$ to conclude the proof. Let $\gamma_Z$ be the loop around $Z$, and $\gamma_D$ be the loop around $D$. Note that ${\gamma_D}^n$ belongs to the kernel of the defining surjection $\pi_1(X_{\rm reg}\setminus{\rm Supp}\,\Delta^{\rm st})\twoheadrightarrow \pi_1^{\rm orb}(X,\Delta)$, and in particular to the normal subgroup $H_0$. By definition of the finite Galois cover $p_0$, we can find $\gamma\in\pi_1(U)$ such that ${p_0}_*\gamma = {\gamma_D}^n$.
    Take $U_Z$ and $U_D$ to be trimmed neighborhoods of $Z$ in $X'$ and $D$ in $X$. Noting that the subgroup of ${\rm Gal}(p)$ whose elements fix every single point of $Z$ is cyclic of order $m$ (because $Z$ has codimension one), we have a commutative diagram 
    $$\xymatrix{
    1 \ar[r]
& \pi_1(U_Z\setminus Z) \ar[r] \ar[d]
& \pi_1(U_D\setminus D) \ar[r] \ar[d]
& \mathbb{Z}/m\mathbb{Z} \ar[r] \ar@{^{(}->}[d]
& 1 \\
1 \ar[r]
& \pi_1(U) \ar[r]
& \pi_1(X_{\rm reg}\setminus {\rm Supp}(\Delta^{\rm st})) \ar[r]
& {\rm Gal}(p) \ar[r]
& 1     }$$
Here $\gamma_D$ in $\pi_1(U_D\setminus D)$ is sent to a generator of $\mathbb{Z}/m\mathbb{Z}$. Hence, the order of $\gamma_D$ in ${\rm Gal}(p_0)$ equals $m$. We already showed that the image of ${\gamma_D}^n$ is trivial in the quotient ${\rm Gal}(p_0)$, so $n$ is divisible by $m$, as wished.

To conclude, we note that
$p_{\bullet}\pi_1^{\rm orb}(X',\Delta')$ is the image of ${p_0}_*\pi_1(U)=H_0$ by the defining surjection $\pi_1(X_{\rm reg}\setminus {\rm Supp}\,\Delta^{\rm st})\twoheadrightarrow\pi_1^{\rm orb}(X,\Delta)$, which is exactly the initial normal subgroup $H$. 
\end{proof}

\end{document}